\documentclass[arxiv]{agn_article}
\usepackage{agn_bib}

\usepackage{framed}
\usepackage{tabularx}
\usepackage{caption}
\usepackage{lscape}
\usepackage{rotating}
\usepackage{float}
\usepackage{biblatex}
\usepackage{wrapfig}
\usepackage{amsmath,amsfonts,amssymb,amsthm}
\usepackage[T1,T2A]{fontenc}
\usepackage[bibtex]{agn_all}
\usepackage{mathrsfs}
\usepackage{epstopdf}

\usepackage{graphicx}
\usepackage{multicol}
\usepackage[inline]{enumitem}
\usepackage{array,multirow}
\usepackage[normalem]{ulem}
\usepackage{arydshln}
\usepackage[super]{nth}
\usepackage{framed}
\usepackage{enumitem}
\usepackage{amsmath,amsfonts,amssymb,amsthm,wasysym}
\PassOptionsToPackage{sortcites}{agn_bib}

\newcommand{\citet}[2][]{\citeauthor{#2} \cite[#1]{#2}}
\newcommand{\DR}{\text{\rm D}}
\newcommand{\DT}{\text{\rm D}^\bot}
\newcommand{\DivR}{\text{\rm Div}}
\newcommand{\DivT}{\text{\rm Div}^\bot}
\newcommand{\os}{\omega}
\newcommand{\Ob}{\mathcal{U}}
\newcommand{\Pj}{\mathbb{P}}
\newcommand{\pj}{\pi_{12}}
\newcommand{\G}{{\bm{\Gamma}}}
\newcommand{\dott}{\circ}

\newcommand{\SKEW}{-\matr{\skew (R\,\partial_x R^T)\\\skew (R\,\partial_yR^T)}}
\newcommand{\zwei}{2}

\newcommand{\<}{\langle}
\renewcommand{\>}{\rangle}

\begin{document}
\title{Regularity for a geometrically nonlinear flat Cosserat micropolar membrane shell with curvature}
\date{\today}
\author{%
	Andreas Gastel\thanks{%
		Corresponding author: Andreas Gastel,\quad Head of Lehrstuhl f\"{u}r Geometrische Analysis, Fakult\"{a}t f\"{u}r Mathematik, Universit\"{a}t Duisburg-Essen, Thea-Leymann Str. 9, 45127 Essen, Germany; email: andreas.gastel@uni-due.de%
	}\quad and\quad%
	Patrizio Neff\thanks{%
		Patrizio Neff,\quad Head of Lehrstuhl f\"{u}r Nichtlineare Analysis und Modellierung, Fakult\"{a}t f\"{u}r	Mathematik, Universit\"{a}t Duisburg-Essen, Thea-Leymann Str. 9, 45127 Essen, Germany, email: patrizio.neff@uni-due.de%
		}
}
\maketitle
\begin{abstract}
\noindent
	We consider the rigorously derived thin shell membrane $\Gamma$-limit of a three-dimensional isotropic geometrically nonlinear Cosserat micropolar model and deduce full interior regularity of both the midsurface deformation $m\col \omega\subset \R^2\to \R^3$ and the orthogonal microrotation tensor field $R\col \omega\subset \R^2\to \SO(3)$. The only further structural assumption is that the curvature energy depends solely on the uni-constant isotropic Dirichlet type energy term $|\DR R|^2$. We use Rivi\`ere's regularity techniques of harmonic map type systems for our system which couples harmonic maps to $\SO(3)$ with a linear equation for $m$. The additional coupling term in the harmonic map equation is of critical integrability and can only be handled because of its special structure.  
\end{abstract}
\\[.65em]
%
%

%
%
%

{\parskip=-0.5mm \tableofcontents}

	\section{Introduction}
%
%
		\subsection{Regularity background and setting of the problem}\label{sect:backgr}

                This paper is a contribution to the wide field of regularity theory of harmonic map type equations. Driven by the application to geometrically nonlinear flat Cosserat shell models, we extend known regularity results to a system that couples a harmonic map equation with another uniformly elliptic equation. The system we consider is of the form  
\begin{align}\label{Gl1}
  \Div S(\DR m,R)&=0,\\
\label{Gl2}
  \Delta R-\Omega_R\cdot\DR R-\skew\big(\DR m\dott S(\DR m,R)\big)R&=0.
\end{align}
The unknown functions here are the midsurface deformation $m\in W^{1,2}(\os,\R^3)$ and the microrotation $R\in W^{1,2}(\os,\SO(3))$, while $\omega\subset \R^2$ is a smooth domain. Moreover, there are functions $\Omega_R\in L^2(\os,(\R^2)^*\otimes \so(3))$
and the force stress tensor $S(\D m,R)\in L^2(\os,\R^{3\times2})$ involved, and $\dott$ is some bilinear product explained later. The function $\Omega_R$ is the same that makes
\begin{equation}\label{Gl3}
  \Delta R-\Omega_R\cdot \DR R=0
\end{equation}
the harmonic map equation for harmonic mappings to $\SO(3)\subset\R^{3\times3}$. The theory of harmonic map equations of $2$-dimensional domains (to any sufficiently smooth compact target manifold, here $\SO(3)$) has a long history. It has been proven in 1948 by Morrey \cite{morrey1948problem} that minimizing weakly harmonic maps are smooth. In 1981, Gr\"uter \cite{grueter1981regularity} generalized that to conformal weakly harmonic maps, and then in 1984 Schoen \cite{schoen1984analytic} to stationary ones. The regularity proof for general weakly harmonic maps was then found in 1990 by H\'elein \cite{helein1990regularite} \cite{helein1991regularite}. (Note that in our case, the target mainfold $\SO(3)$ is a Lie group, and in this case the harmonic map problem has a lot of interesting extra stucture, many aspects of which are covered in Helein's book \cite{helein2002harmonic}.) Later, in 2007, Rivi\`ere \cite{riviere2007conservation} revisited harmonic map type equations and asked for which $\Omega_R$ all weak solutions of (\ref{Gl3}) on a two-dimensional domain are smooth. It turned out that $\Omega_R$ need not come from the harmonic map equation (in which case it can be seen as the anti-symmetrized tensor derived from the second fundamental form of the target manifold), but for the regularity result only the skew-symmetry of $\Omega_R$ is needed. This gave deeper insight in the structures necessary to have regularity results, and it is Rivi\`ere's philosophy that we rely upon.

It should be pointed out that, with $\Omega_R$ and $\DR R$ being in $L^2$, the nonlinear term $\Omega_R\cdot \DR R$ in (\ref{Gl3}) is only in $L^1$, and if it would not have any further structure, it would be difficult to start with any regularity theory, due to the lack of an $L^p$-theory working for $p=1$. But it turns out that the product $\Omega_R\cdot\DR R$, after a suitable gauge transformation, is the sum of products of divergence-free vector fields and gradients
in $L^2$, which is known to be in the Hardy space ${\mathcal H}^1$ rather than $L^1$. This little bit of extra regularity is enough to perform regularity theory.

Now let us have a look at our equation (\ref{Gl2}). Compared with (\ref{Gl3}), it has an extra term $\skew(\DR m\dott S(\DR m,R))R$, and again, with $\DR R\in L^2$, $S(\DR m,R)\in L^2$, and $R\in L^\infty$, this has only $L^1$-integrability. But once more, $\DR R$ is a gradient, and $S(\DR m,R)$ is divergence free due to equation (\ref{Gl1}). This time, we have the product of a gradient $\D m$, a divergence free vector field $S(\D m,R)$, and a bounded function $R$. Based on a crucial estimate by Coifman, Lion, Meyer and Semmes \cite{coifman1993compensated}, Rivi\`ere and Struwe \cite{riviere2008partial} were able to handle such products in their work on partial regularity in dimensions $\ge3$. They encountered such products in the course of their proof for the equation (\ref{Gl3}) without any extra terms, and we can modify their arguments to handle our extra term from the coupling. The handling of the first equation, which is linear in $m$ with some right-hand side, is easier, in principle. But we have to do the iteration procedure for both equations simultaneously in the proof of H\"older continuity, resulting in some technicalities. Once we have that, classical Schauder theory helps with the higher regularity of $m$, while for the second equation controlling the smoothness of $R$, we still need some machinery.

	\subsection{Engineering background and application}
	The Cosserat model is one of the best known generalized continuum models \cite{Capriz89}. It assumes that material points can undergo translation, described by the standard deformation map $\varphi\col \mathcal{U}\to \R^3$ and independent micro rotations described by the orthogonal tensor field $R\col \mathcal{U}\to \SO(3)$, where $\mathcal{U}\subset \R^3$ describes the smooth reference configuration of the material. Therefore, the geometrically nonlinear Cosserat model induces immediately the Lie-group structure on the configuration space $\R^3\times \SO(3)$. 
	
	Both fields are coupled in the assumed elastic energy $W=W(\DR \varphi, R, \DR R)$ and the static Cosserat model appears as a two-field minimization problem which is automatically geometrically nonlinear due to the presence of the non-abelian rotation group $\SO(3)$. Material frame-indifference (objectivity) dictates left-invariance of the Lagrangian $W$ under the action of $\SO(3)$ and material symmetry (here isotropy) implies right-invariance under action of $\SO(3)$.
	
	In the early 20th century the Cosserat brothers E. and F. Cosserat introduced this model in its full geometrically nonlinear splendor \cite{cosserat1909theorie} in a bold attempt to unify field theories embracing mechanics, optics and electrodynamics through a common principal of least action. They used the invariance of the energy under Euclidean transformations \cite{cosserat1991note,gauthier1909} to deduce the correct form of the energy $W=W(R^T\DR \varphi, R^T\partial_xR,R^T\partial_yR,R^T\partial_zR)$ and to derive the equations of balance of forces (variations w.r.t the deformation $\varphi$, the force-stress tensor may loose symmetry \cite{neff2008symmetric}) and balance of angular momentum (variations w.r.t. rotations $R$). The Cosserat brothers did not provide, however, any specific constitutive form of the energy since they were not interested in applications.
	
	While the appearance of an additional rotational field $R$ for describing the elastic response of bulk material is requiring getting used to, such an appearance is most natural in the case of shell-theory. There, the Frenet-Darboux tri\`{e}dre \cite{cosserat1908theorie} (tri\`{e}dre cach\'{e} in the terminology of the Cosserats, trihedron) naturally plays a role and it is no big step to assume that this orthogonal field is supposed to be kinematically independent of the former (tri\`{e}dre mobile). Hence the Cosserat approach \cite{cosserat1908theorie}; the independent rotation field $R$ describes the rotations of the cross-sections of the shell (including in-plane drill rotations about the normal $n_m$ to the midsurface $m$) and these cross-sections are all allowed to shear with respect to the normal of the midsurface ($Re_3\neq {n}_m$).

	\begin{figure}[h!]
		\begin{center}
			\includegraphics[scale=1]{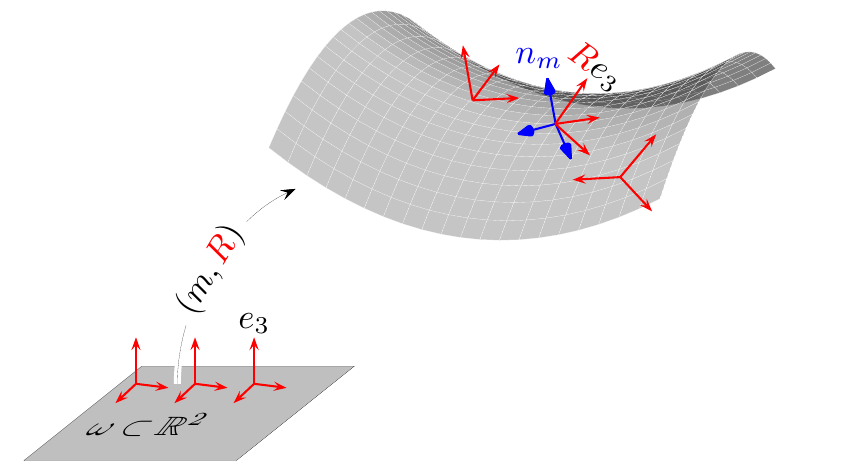}
		
			\caption{\footnotesize The mapping $m\col \os\subset \R^2\to\R^3$ describes the deformation of the flat midsurface $\os\subset \R^2$. The Frenet-Darboux frame (in blue, tri\`{e}dre cach\'{e}) is tangent to the midsurface. The independent frame mapped by $R\in \SO(3)$ is the tri\`{e}dre mobile (in red, not necessary tangent to the midsurface). Both fields $m$ and $R$ are coupled in the variational problem. }
			\label{Fig1}       
		\end{center}
	\end{figure}

	On this basis, very efficient ad-hoc Cosserat shell-models have been introduced, see e.g. \cite{Altenbach-Erem-11,Altenbach-Erem-Review}. A special case of these shell models is the family of Reissner-Mindlin shells in which the in-plane rotations are discarded (no drill energy) \cite{kikis2022two} and one is left with a one director theory \cite{le1995nonlinear} \footnote{One director geometrically nonlinear, physically linear Reissner-Mindlin shells are typically not well-posed, since the membrane stretch energy part depends quadratically on $\DR m^T\DR m-\id_2$, which is not rank-one elliptic in the compression regime, for a detailed exposition, see the Appendix and \cite{Neff2022}. }. Upon identifying/constraining the tri\`{e}dre mobile with the tri\`{e}dre cach\'{e} (microrotation equals continuum rotation, Cosserat couple modulus $\mu_c\to \infty$), canonical shell models of Kirchhoff-Love type emerge \cite{Ghiba2022}. However, engineers would often prefer the Cosserat shell models since these yield nonlinear balance equations of second order \cite{sander2016numerical,wisniewski1998shell,wisniewski2000kinematics,ibrahimbegovic1994stress,pietraszkiewicz2014drilling,bacsar1987consistent}.
	
	The precise derivation of Cosserat shell models may proceed in several different ways: integration of equilibrium equations through the thickness \cite{eremeyev2004nonlinear,pietraszkiewicz2014drilling}, direct modeling as a two-dimensional directed surface \cite{Altenbach-Erem-11,Altenbach-Erem-Review,green1965general}, or the derivation approach, which starts from a three-dimensional variational problem and introduces certain assumptions for the deformation behavior through the thickness. The second author has introduced this derivation procedure based on the geometrically nonlinear Cosserat model in his habilitation thesis \cite{neff2004geometricallyhabil,neff2004geometrically}. Lastly, there is the "ansatz-free" method of $\Gamma$-convergence \cite{braides2006handbook,braides2002gamma,bhattacharya1999theory}  (while letting the thickness $h$ tend to zero) to perform the dimensional descent.
	
	In this method, one needs to choose an energy scaling regime and obtains typically either membrane or bending like theories \cite{friesecke2002theorem,le1995nonlinear,le1996membrane,ASNSP_2010_5_9_2_253_0} when starting from classical finite strain elasticity \cite{friesecke2002theorem,friesecke2003derivation,friesecke2006hierarchy}. However, the $\Gamma$-limit membrane model \cite{le1995nonlinear,le1996membrane} has a serious shortcoming which is connected to the necessary relaxation step: it does not predict any resistance against compression and averages out the expected fine scale wrinkling response. The situation is strikingly different when starting from a three-dimensional Cosserat model, as done in \cite{neff2005gamma}. This is true since the bulk-Cosserat model already features a curvature term (derivatives of $R$) which "survives" the membrane scaling. 
	
	 The Cosserat membrane $\Gamma$-limit with remaining curvature effects can be used as an effective surrogate model to describe ultra thin graphen mono-layers. Graphen is the name given to a single atomic layer of carbon atoms tightly packed into a two-dimensional honeycomb lattice (see Figure \ref{geraphen}). It can be wrapped up to form  fullerenes, rolled into nanotubes \cite{zhang2013single} or stacked into  graphite. It's stiffness properties are extreme. Such a graphen layer has resistance against in-plane stretch and curvature changes but it's thickness is so small, that a classical membrane-bending model (where the bending terms scale with $h^3$ while the membrane terms with $h$) is clearly insufficient. It is simply impossible to speak about the "thickness" of graphen in a classical continuum framework. Researchers then usually resort to introducing an "effective bending rigidity" in order to apply concepts from classical shell theory. 
	 This can be completely avoided in the Cosserat membrane model. 
	 
	 In this paper we will consider, for the first time, the challenging regularity questions for the flat shell Cosserat membrane $\Gamma$-limit.
	 To the best knowledge of the authors, such a regularity investigation for the flat Cosserat membrane shell has never been undertaken. Two recent previous contributions consider the regularity issue for the geometrically isotropic nonlinear Cosserat bulk equations \cite{tel2019regularity,li2022regularity}, both times restricting attention to the uni-constant Dirichlet curvature energy $|\DR R|^2$, leading to a $\Delta R$-term in the Euler-Lagrange equations and allowing the sophisticated techniques for harmonic map type systems to be used.
	\begin{figure}
		\begin{center}
			\includegraphics[scale=0.4]{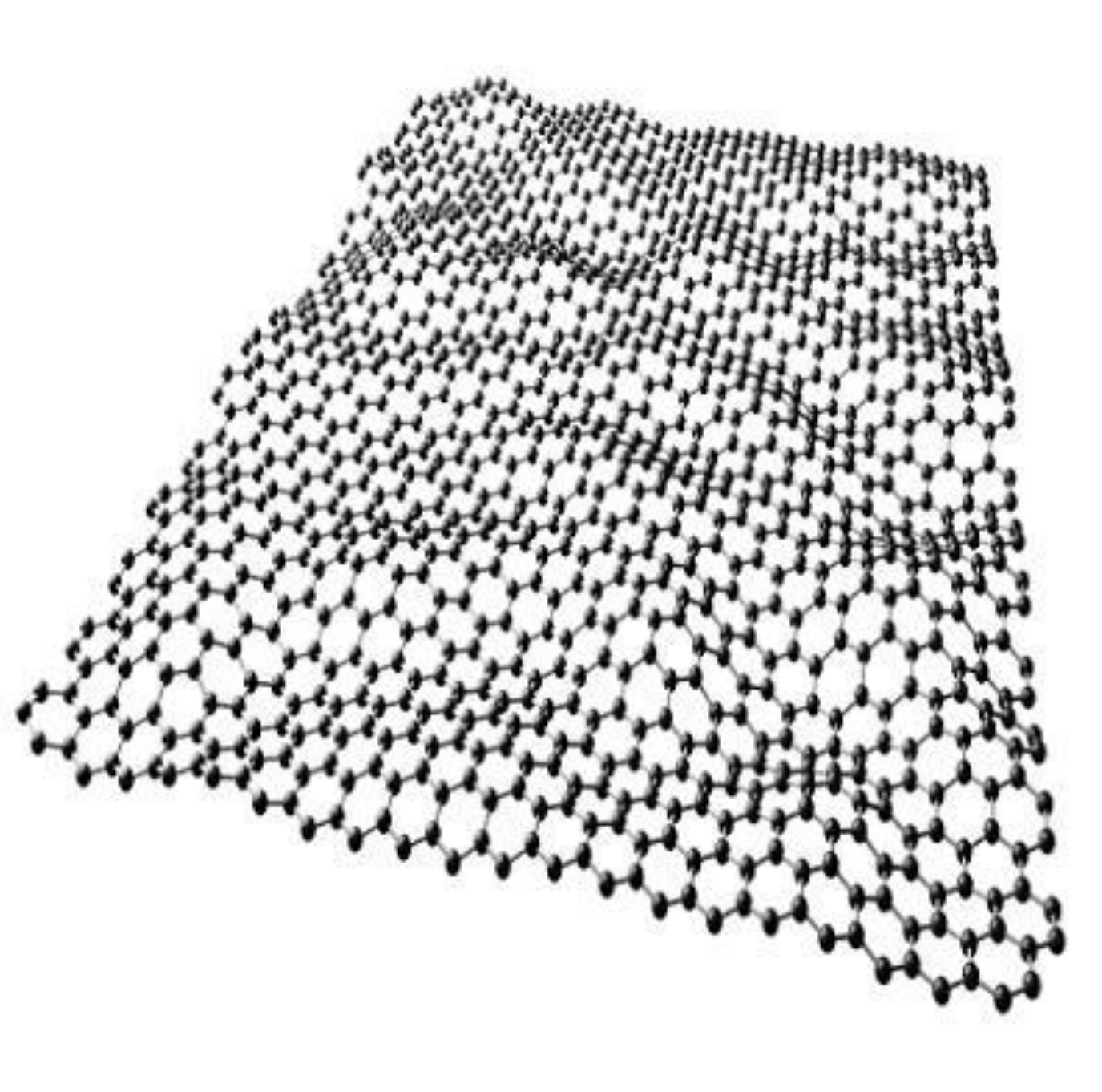}
			\caption{\footnotesize Deformed graphen mono atomic-layer resisting against in-plane stretches (membrane effects) and curvature. Classical continuum models are not any more suitable, since there is no tangible thickness, c.f. \cite{zhang2013single}. Graphen is thought to be the strongest among all known materials. Nevertheless it is soft in the sense that it can be easily bent due to its one atom thin nature.}
			\label{geraphen}
		\end{center}
	\end{figure}
	
	This paper is now structured as follows. After this introduction and the introduction of our notation, in Section 3 we will introduce the three-dimensional isotropic Cosserat model, together with a short discussion of suitable representations for the curvature term. Following, in Section 4, we briefly describe the dimensional descent towards a membrane shell, juxtaposing the result of the $\Gamma$-limit procedure and a formal engineering approach. In Section 5 we introduce the final two-dimensional Cosserat membrane shell model together with some pertinent notations and simplifications. The remainder of the paper is devoted to showing the interior Hölder regularity of these weak solutions. In the appendix we gather further useful calculations like the three-dimensional Euler-Lagrange equations in dislocation tensor format, we present a more engineering oriented derivation of the two-dimensional Euler-Lagrange equations and give a glimpse on a related Reissner-Mindlin model.
	Finally, we show some numerical experiments of the flat Cosserat membrane shell model in compression.

	\section{Notation}

	Let $a,b\in \mathbb{R}^3$. We denote the scalar product on $\mathbb{R}^3$ with $\iprod{a,b}_{\mathbb{R}^3}$ and the associated vector norm with $|a|^2_{\mathbb{R}^3}=\iprod{a,a}_{\mathbb{R}^3}$. The set of real-valued $3\times3$ second order tensors is denoted by $\mathbb{R}^{3\times3}$. 
	The standard Euclidean scalar product on $\mathbb{R}^{3\times 3}$ is given by $\iprod{X,Y}_{\mathbb{R}^{3\times3}}=\tr (XY^\text{T})$, and the associated norm is $|X|^2=\iprod{X,X}_{\mathbb{R}^{3\times3}}$. If $\id_3$ denotes the identity matrix in $\mathbb{R}^{3\times3}$, we have $\tr(X)=\iprod{X,\id_3}$. For an arbitrary matrix $X\in \mathbb{R}^{3\times3}$ we define $\text{sym}(X)=\frac{1}{2}(X+X^{\text{T}})$ and $\text{skew}(X)=\frac{1}{2}(X-X^{\text{T}})$ as the symmetric and skew-symmetric parts, respectively and the trace free deviatoric part is defined as $\dev X= X-\frac{1}{n}\tr(X)\id_n$, for all $X\in \R^{n\times n}$. We let ${\rm Sym}(n)$ and ${\rm Sym}\hspace{-0.05cm}^+\hspace{-0.05cm}(n)$ denote the symmetric and positive definite symmetric tensors, respectively. The Lie-algebra of skewsymmetric matrices is denoted by $\so(3):=\{X\in \R^{3\times 3}\,|\, X^T=X\}$ and the Lie-algebra of traceless tensors is defined by $\mathfrak{sl}(3):=\{X\in \R^{3\times 3}\,|\, \tr(X)=0\}$. We consider the orthogonal decomposition $X=\dev\sym X+\skew X+\frac{1}{3}\tr(X)\cdot\id_3=\sym X+\skew X$.
	 The \textit{canonical identification} of $\mathfrak{so}(3)$ and $\R^3$ is given by $\axl\col\so(3)\to \R^3$ and its inverse $\Anti\col\R^3\to \so(3)$. We note the following properties
	 \begin{align}
	 \axl\underbrace{\matr{0&\alpha&\beta\\-\alpha&0&\gamma\\-\beta&-\gamma&0}}_{=A}:=\matr{-\gamma\\ \beta\\-\alpha}\,,\qquad |A|^2_{\R^{3\times 3}}=2\,|\axl A|^2_{\R^3}\,,\qquad A\,v=\axl(A)\times v\,.
	 \end{align}
	 and 
	 \begin{align}
	 \Anti\matr{v_1\\ v_2\\v_3}:=\matr{0&-v_3&v_2\\v_3&0&-v_1\\-v_2&v_1&0}\in \so(3)\,.
	 \end{align}
	  A matrix having the  three  column vectors $R_1,R_2, R_3$ will be written sometimes as 
	 $
	 R=(R_1\,|\, R_2\,|\,R_3)\in \R^{3\times 3}
	 $. The matrix $\Curl$ and matrix $\Div$ are defined row-wise as 
	 \begin{align}
	 \Curl R=\matr{\text{curl}\,(R^T\hspace{-0.2cm}\cdot e_1)\\\text{curl}\,(R^T\hspace{-0.2cm}\cdot e_2)\\\text{curl}\,(R^T\hspace{-0.2cm}\cdot e_3)}\in \R^{3\times 3}\,,\qquad \Div R=\matr{\div (R^T\hspace{-0.2cm}\cdot e_1)\\\div (R^T\hspace{-0.2cm}\cdot e_2)\\\div (R^T\hspace{-0.2cm}\cdot e_3)}\,.
	 \end{align}
	For $\varphi\in C^1(\Ob,\R^3)$ and for every vector $(x,y,z)\in \R^3$, we write
	\begin{align}
	\DR\.\varphi=\matr{\varphi_{1,x}&\varphi_{1,y}&\varphi_{1,z}\\\varphi_{2,x}&\varphi_{2,y}&\varphi_{2,z}\\\varphi_{3,x}&\varphi_{3,y}&\varphi_{3,z}}=(\partial_x\varphi|\partial_y\varphi|\partial_z\varphi)\,.
	\end{align}
	The mapping $m\col \os\subset \R^2\to \R^3$ will always denote the deformation of the midsurface $\os$ and we write
	\begin{align}
	\DR m=\matr{m_{1,x}&m_{1,y}\\m_{2,x}&m_{2,y}\\m_{3,x}&m_{3,y}}=(\partial_x m|\partial_y m)\,,\qquad \DT m=\matr{-m_{1,y}&m_{1,x}\\-m_{2,y}&m_{2,x}\\-m_{3,y}&m_{3,x}}=(-\partial_y m|\partial_x m)\,.
	\end{align}
        Moreover, we will use the notations
        \begin{align}
          \DivR(A_1|A_2)=\partial_xA_1+\partial_yA_2,\quad\qquad \DivT(A_1|A_2)=\partial_xA_2-\partial_yA_1,
        \end{align}
        where $A_1,A_2$ may be numbers-, vector-, or matrix-valued functions on $\os$ of the same type. Note that it is also customary to write $\Curl$ instead of $\DivT$, but the latter underscores the symmetry of $(\DR,\DivR)$ with $(\DT,\DivT)$, hence we reserve $\Curl$ for three-dimensional domains.    
        
	We assume that $h>0$ with $h\ll 1$. The three-dimensional flat thin domain $\Ob_h\subset \R^3$ is introduced as
	\begin{align}
	\Ob_h:=\os\times [-\frac{h}{2},\frac{h}{2}]\,,\qquad \os\subset \R^2\,.
	\end{align}
	We also need to define the projection operator on the first two columns
	\begin{align}
	\pj\col \R^{3\times 3}\to \R^{3\times 2}\,,\quad\qquad \pj(X)=\pj(X_1|X_2|X_3)=(X_1|X_2)=\matr{X_{11}&X_{12}\\X_{21}&X_{22}\\X_{31}&X_{32}}\,,
	\end{align}
	and the operator
	\begin{align}
	\nonumber\Pj_{\mu,\mu_c,\kappa}&\col \quad\R^{3\times 3}\to \{\sl(3)\cap \Sym(3)\}\oplus\so(3)\oplus\R\cdot \id_3\,,\\
	\Pj_{\mu,\mu_c,\kappa}(X)&=\Pj(X)=\sqrt{\mu}\dev\sym X+\sqrt{\mu_c}\skew X+\frac{\sqrt{\kappa}}{3}\tr(X)\cdot \id_3\,,\\
	\nonumber\Pj^*\Pj(X)&=\Pj^2(X)=\mu\dev\sym X+\mu_c\skew X+\frac{\sqrt{\kappa}}{3}\tr(X)\id_3\,,\qquad \Pj^*=\Pj\,.
	\end{align}
	
	\section{Three-dimensional geometrically nonlinear isotropic Cosserat model}
	The underlying three-dimensional isotropic Cosserat model can be described in terms of the standard deformation mapping $\varphi\col \Ob\subset \R^3\to \R^3$ and an additional orthogonal microrotation tensor $R\col \Ob\subset \R^3\to \SO(3)$.
	
	The goal is to find a minimiser of the following isotropic energy
	{\small
	\begin{align}\label{isotrop ener}
	\nonumber E^{\text{3D}}(\varphi,R)&=\int_\Ob\mu \big|\dev\sym(R^T\DR\varphi-\id_3)\big|^2+\mu_c\big|\skew (R^T\DR\varphi-\id_3)\big|^2+\frac{\kappa_{\text{3D}}}{2}\tr(R^T\DR\varphi-\id_3)^2\\
	\nonumber&\quad\quad+\mu\frac{L_c^2}{2}\Big(a_1\big|\dev\sym R^T\Curl R\,\big|^2+a_2\big|\skew R^T\Curl R\big|^2+\frac{a_3}{3}\tr(R^T\Curl R)^2\Big) \,{\rm{dx}} \\
	\nonumber&=\int_\Ob\mu \big|\sym(R^T\DR\varphi-\id_3)\big|^2+\mu_c\big|\skew (R^T\DR\varphi-\id_3)\big|^2+\frac{\lambda}{2}\tr(R^T\DR\varphi-\id_3)^2\\
	\nonumber&\quad\quad+\mu\frac{L_c^2}{2}\Big(a_1\big|\dev\sym R^T\Curl R\,\big|^2+a_2\big|\skew R^T\Curl R\big|^2+\frac{a_3}{3}\tr(R^T\Curl R)^2\Big) \,{\rm{dx}} \\
	&=\int_{\Ob}W_{\text{mp}}(R^T\DR \varphi)+W_{\text{disloc}}^{\text{3D}}(R^T\Curl R)\;{\rm{dx}}\to\min \quad\text{w.r.t}\quad (\varphi,R)\,.
	\end{align}}
	The problem will be supplemented by Dirichlet boundary conditions for the deformation $\varphi$ but the microrotations $R$ can be left free. 
	Here, $\mu>0$ is the standard elastic shear modulus, $\kappa_{\text{3D}}=\frac{3\lambda+2\mu}{3}>0$ is the three dimensional elastic bulk modulus (with $\lambda$ the second elastic Lam\'{e} parameter) and $\mu_c\geq0$ is the so-called Cosserat couple modulus, $a_1,a_2,a_3$ are non-dimensional non-negative weights and $L_c>0$ is a characteristic length. The energy (\ref{isotrop ener}) is the most general isotropic quadratic representation for the Cosserat model in terms of the nonsymmetric Biot type stretch tensor $\overline{U}=R^T\DR \varphi$ (first Cosserat deformation tensor \cite{cosserat1909theorie}) and the curvature measure $R^T\Curl R$ (physically linear, small strain, but geometrically nonlinear). We call 
	\begin{align}
	{\bm{{\bm{\alpha}}}}:=R^T\Curl R\,,
	\end{align}
	the \textit{second order} dislocation density tensor \cite{birsan2016dislocation}.
	Due to the orthogonality of $\dev\sym$, $\skew$ and $\tr(.)\id_3$, the curvature energy provides a complete control of 
	\begin{align}
	|{\bm{\alpha}}|^2=\big|R^T\Curl R\big|^2\qquad\text{provided}\qquad a_1,a_2,a_3>0\,.
	\end{align}
	
	For example, we can express the uni-constant isotropic curvature term
	\begin{align}
	|\DR R|^2_{\R^{3\times 3\times 3}}&=|R^T\DR R|_{\R^{3\times 3\times 3}}^2=|R^T\partial_x R|_{\R^{3\times 3}}^2+|R^T\partial_yR|_{\R^{3\times 3}}^2+|R^T\partial_zR|_{\R^{3\times 3}}^2\\
	\nonumber&=1\cdot|\dev\sym R^T\Curl R|^2_{\R^{3\times 3}}+1\cdot|\skew R^T\Curl R|^2_{\R^{3\times 3}}+\frac{1}{12}\cdot\tr(R^T\Curl R)^2=|\Pj_{1,1,\frac{1}{12}}(\alpha)|^2\,,
	\end{align}
	where we have used (\ref{gamma}) and $|\G|^2=|\axl(R^T\partial_x R)|^2+|\axl(R^T\partial_y R)|^2+|\axl(R^T\partial_z R)|^2$ together with $2|\axl(A)|_{\R^3}^2=|A|_{\R^{3\times 3}}^2$.
	Using the result in \cite{neff2008curl} 
	\begin{align}
	\big|\Curl R\,\big|^2_{\R^{3\times 3}}\geq c^+|\DR R\,|^2_{\R^{3\times 3\times 3}}\,,
	\end{align}
	shows that (\ref{isotrop ener}) controls $\DR R$ in $L^2(\Ob, \R^{3\times 3\times 3})$. 
	
	In this setting, the minimization problem is strictly convex in the strain and curvature measures $(\overline{U},{\bm{\alpha}})$ but highly non-convex w.r.t $(\varphi,R)$. Existence of minimizers for (\ref{isotrop ener}) with $\mu_c>0$ has been shown first in \cite{neff2004existence}, see also \cite{focardi2015multi,mariano2009ground,neff2015existence, neff2004existence, lankeit2017integrability, birsan2016dislocation, neff2005gamma}. The partial regularity of minimizers/statonary solutions is investigated in \cite{tel2019regularity,li2022regularity} under additional assumptions. Note also that in \cite{tel2019regularity}, the first author gives an example of a solution that exhibits a point singularity.
	
	The Cosserat couple modulus $\mu_c$ controls the deviation of the microrotation $R$ from the continuum rotation $\text{polar}(\DR \varphi)$ in the polar decomposition of $\DR \varphi=\text{polar}(\DR \varphi)\cdot\sqrt{\DR \varphi^T \DR \varphi}$, cf. \cite{neff2014grioli}.
	
	For $\mu_c\to \infty$ the constraint $R=\text{polar}(\DR \varphi)$ is generated and the model would turn into a Toupin couple stress model.

	\subsection{Connections to the Oseen-Frank energy in nematic liquid crystals  }\label{Oseen-Frank}
	In nematic liquid crystals one considers the unit-director field $n\col \mathcal{U}\subset \R^3\to \mathbb{S}^2$, minimizing the three-parameter frame-indifferent "curvature energy" \cite{virga2018variational}
	\begin{align}\label{kap}
	\int_{\mathcal{U}}\frac{1}{2}K_1|\div n|^2+\frac{1}{2}K_2|\iprod{n,\curl n}|^2+\frac{1}{2}K_3|n\times \curl n|^2\,{\rm{dx}}\,.
	\end{align}
	The uni-constant approximation $K_1=K_2=K_3$ leads to the Dirichlet type integral\footnote{For this, we note the identity (see, \cite{ball2017mathematics} eq (2.5) and \cite{alouges1997minimizing} eq (2.6))
		\begin{align}
		\tr(\DR v)^2+\iprod{v,\curl v}^2+|v\times \curl v|^2=|\DR v|^2+(|v|^2-1)\,|\curl v|^2\,,
		\end{align} valid for all sufficiently smooth vector fields $v\col \mathcal{U}\subset \R^3\to \R^3$.} 
	\begin{align}\label{int appr}
	\int_\mathcal{U} \frac{1}{2}K_1|\DR n|^2\;{\rm{dx}}\,.
	\end{align}
	The corresponding Euler-Lagrange equations for the uni-constant case are (see e.g. \cite{alouges1997minimizing})
	\begin{align}\label{euler-lag}
	\Delta n=-|\DR n|^2\cdot n\,,
	\end{align}
	see equation (\ref{Deltaanti}) for a self-contained derivation. Since (\ref{int appr}) and (\ref{euler-lag}) are just the energy and Euler-Lagrange equation for harmonic maps to spheres, all regularity theorems for harmonic maps apply. In the three dimensional case, minimizers are smooth up to a discrete set of singularities. Stationary solutions have a co-dimension $1$ singular set. In the two dimensional case, all weak solutions of (\ref{euler-lag}) are smooth, see Section (\ref{sect:backgr}) for the literature on this.
	
	For $K_1,K_2,K_3$ positive and different, any minimizer to (\ref{kap}) is smooth except for a closed set of Hausdorff dimension strictly less than 1, cf. \cite{hardt1986existence}. Ball and Bedford \cite{ball2015discontinuous} consider the sublinear regime $|\DR n|^q$, $1<q<2$.
	
	\section{Dimensional descent towards a membrane model}
	\subsection{Membrane $\Gamma$-limit}
	We are interested in a situation, where the reference configuration is flat with uniform shell thickness $h>0$, i.e. the reference configuration is taken to be of the form (see Figure \ref{Fig3} )
	\begin{align}
	\Ob_h=\os\times [-\frac{h}{2},\frac{h}{2}]\,,\qquad\os\subset \R^2\,.
	\end{align} 
	\begin{figure}[h!]
		\begin{center}
			\includegraphics[scale=0.7]{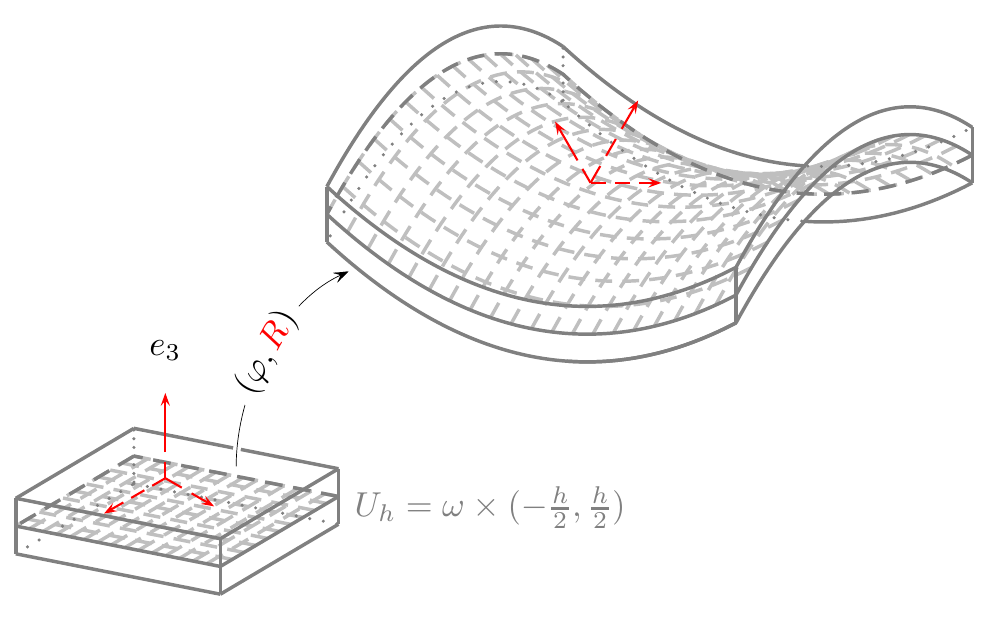}
			\caption{\footnotesize Process of dimensional reduction. Flat reference configuration with height $h$ and deformed configuration. }
			\label{Fig3}
		\end{center}
	\end{figure}
	The goal is to derive a limit two-dimensional problem, posed over the referential midsurface $\os\subset \R^2$, as $h\to 0$. This has been achieved in \cite{neff2010reissner} based on $\Gamma$-convergence arguments and using the nonlinear membrane scaling. We say that the dimensionally reduced model is a membrane, since no dedicated bending terms appear in the problem.	
	
	However, since the Cosserat model already includes curvature terms (those depending on space derivatives $\DR R$), these curvature terms "survive" in the $\Gamma$-limit procedure and scale with $h$, while canonical bending terms scale with $h^3$. This sets the Cosserat membrane model apart from more canonical membrane models \cite{neff2007geometrically1}.
	
	For the $\Gamma$-limit procedure it is useful to re-express the curvature energy from (\ref{isotrop ener}) 
	\begin{align}\label{Wcurv}
	\nonumber\mu\,\frac{L_c^2}{2}\Big(a_1&\big|\dev\sym (R^T\Curl R)\big|^2+a_2\big|\skew (R^T\Curl R)\big|^2+\frac{a_3}{3}\tr(R^T\Curl R)^2\Big)\\
	&=\mu\,\frac{L_c^2}{2}\Big(a_1\big|\dev\sym {\bm{\alpha}}\big|^2+a_2\big|\skew {\bm{\alpha}}\big|^2+\frac{a_3}{3}\tr({\bm{\alpha}})^2\Big)=\mu\frac{L_c^2}{2}|\Pj_{a_1,a_2,a_3}(\alpha)|^2=:W_{\text{disloc}}^{\text{3D}}({\bm{\alpha}})\,,
	\end{align}
	in terms of the so-called second order \textit{wryness tensor}\cite{neff2008curl, eremeyev2004nonlinear} (second Cosserat deformation tensor \cite{cosserat1909theorie})
	\begin{align}\label{wryness}
	\G:=\Big(\axl(\underbrace{R^T\partial_xR}_{\in \so(3)})|\axl (R^T\partial_y R)|\axl(R^T\partial_zR)\Big)=(\Gamma_1|\Gamma_2|\Gamma_3)\in \R^{3\times 3}\,.
	\end{align}
	Since $R^T\partial_{x_i} R\in \so(3)$, $i=1,2,3$ is skew-symmetric, we have the following relation \cite{NeffPartI,GhibaNeffPartII,nye1953some}
	\begin{align}\label{gamma}
	\G=-{\bm{\alpha}}^T+\frac{1}{2}\tr({\bm{\alpha}})\id_3\,,\hspace{2cm} {\bm{\alpha}}=-\G^T+\tr(\G)\id_3\,.
	\end{align}
	By using these formulas we note
	\begin{align}
	\dev\sym{\bm{\alpha}}=-\dev\sym\G\,,\qquad\qquad\skew {\bm{\alpha}}=\skew \G\,,\qquad\qquad\tr({\bm{\alpha}})=2\tr(\G)\,.
	\end{align}
	Now using (\ref{Wcurv}), we obtain
	\begin{align}
	\nonumber W_{\text{disloc}}^{\text{3D}}({\bm{\alpha}})&=\mu\frac{L_c^2}{2}\Big(a_1|\dev\sym{\bm{\alpha}}|^2+a_2|\skew {\bm{\alpha}}|^2+\frac{a_3}{3}\tr({\bm{\alpha}})^2\Big)\\
	\nonumber&=\mu\frac{L_c^2}{2}\Big(a_1|\dev\sym\G|^2+a_2|\skew \G|^2+4a_3\tr(\G)^2\Big)\\
	&=\mu\frac{L_c^2}{2}\Big(\widetilde{a}_1|\dev\sym\G|^2+\widetilde{a}_2|\skew \G|^2+\widetilde{a}_3\tr(\G)^2\Big)=:W_{\text{curv}}^{\text{3D}}(\G)\,,
	\end{align}
	where $\widetilde{a}_1=a_1, \widetilde{a}_2=a_2$ and $\widetilde{a}_3=4a_3$. Altogether we get 
	\begin{align}
	\nonumber W_{\text{disloc}}^{\text{3D}}({\bm{\alpha}})=W_{\text{curv}}^{\text{3D}}(\G)&=\mu\frac{L_c^2}{2}\Big(\widetilde{a}_1|\dev\sym \G|^2+\widetilde{a}_2|\skew \G|^2+\widetilde{a}_3\tr(\G)^2\Big)\,, \quad\text{with}\quad \widetilde{a}_1,\widetilde{a}_2,\widetilde{a}_3>0\,,\\
	&=\mu\,\frac{L_c^2}{2}\Big(\widetilde{b}_1|\sym \G|+\widetilde{b}_2|\skew\G|^2+\widetilde{b}_3\tr(\G)^2\Big)\,,
	\end{align}
	with $\widetilde{a}_1=\widetilde{b}_1>0$, $\widetilde{a}_2=\widetilde{b}_2>0$ and $\widetilde{b}_3=\frac{\widetilde{a}_1}{3}+\widetilde{a}_3>0$. Thus, the variational problem (\ref{isotrop ener}) can be equivalently expressed as 
	\begin{align}
	E^{\text{3D}}(\varphi,R)=\int_{\mathcal{U}_h}W_{\text{mp}}(R^T\DR \varphi)+W_{\text{curv}}^{\text{3D}}(\G)\;{\rm{dx}}\quad\to \min\quad \text{w.r.t}\quad (\varphi , R)\,.
	\end{align}
	Applying the nonlinear scaling \cite{fonseca2001inadequacy}, allows to rewrite the problem on a domain $\mathcal{U}_1=\os\times [-\frac{1}{2},\frac{1}{2}]$ with unit thickness in terms of properly scaled variables $\varphi^\natural,R^\natural$ in (thickness) $z$-direction
	\begin{align}
	E_h^{\text{3D}}(\varphi^\natural,R^\natural)=\int_{\mathcal{U}_1}W_{\text{mp}}(R^{\natural,T}\DR \varphi^\natural)+W_{\text{curv}}^{\text{3D}}(\G^\natural)\,{\rm{dx}}\,.
	\end{align}
	The descaled $\Gamma$-limit of $E_h^\text{3D}$ as $h\to 0$ is then given by \cite{neff2005gamma}
	\begin{align}
	E^{\text{2D}}(m,R)=\int_\os h\,\Big(W_{\text{mp}}^{\text{hom}}(R^T(\DR m|R_3))+W_{\text{curv}}^\text{hom}(\widehat{\G})\Big)\,{\rm{dx}}\to \quad\text{min}\quad\text{w.r.t}\quad(m,R)\,,
	\end{align}
	where $m\col\os\subset \R^2\to \R^3$ describes the deformation of the midsurface, $R\col \os\subset \R^2\to \SO(3)$ and 
	\begin{align}
	\nonumber W_{\text{mp}}^\text{hom}\big(R^T(\DR m|R_3)\big):&=\inf_{d\in \R^3} W_{\text{mp}}\big(R^T(\DR m|d)\big)\\
	\nonumber&=\mu|\sym\big((R_1|R_2)^T\DR\.m-\id_2\big)|^2+\mu_c|\skew\big((R_1|R_2)^T\DR\.m-\id_2\big)|^2\\
	\nonumber &\qquad\qquad+\frac{2\mu\.\mu_c}{\mu+\mu_c}\big(\iprod{R_3,\partial_xm}^2+\iprod{R_3,\partial_ym}^2\big)+\frac{\mu\.\lambda}{2\mu+\lambda}\tr\Big((R_1|R_2)^T\DR\.m-\id_2\Big)^2\,,\\
	\nonumber W_{\text{curv}}^{\text{hom}}(\widehat{\G}):&=\inf_{A\in \so(3)}W_{\text{curv}}^\text{\text{3D}}\big(\axl(R^T\partial_xR),\axl(R^T\partial_y R),\axl(A)\big)\\
	&=\mu\frac{L_c^2}{2}\Big(\widetilde{b}_1\Big|{\sym\matr{\widehat{\Gamma}_{11}&\widehat{\Gamma}_{12}\\\widehat{\Gamma}_{21}&\widehat{\Gamma}_{22}}}\Big|^2+\widetilde{b}_2\Big|\skew\matr{\widehat{\Gamma}_{11}&\widehat{\Gamma}_{12}\\\widehat{\Gamma}_{21}&\widehat{\Gamma}_{22}}\Big|^2\\
	&\qquad\qquad\qquad+\frac{\widetilde{b}_1\,\widetilde{b}_3}{\widetilde{b}_1+\widetilde{b}_3}\tr\matr{\widehat{\Gamma}_{11}\nonumber&\widehat{\Gamma}_{12}\\\widehat{\Gamma}_{21}&\widehat{\Gamma}_{22}}^2+2\,\frac{\widetilde{b}_1\,\widetilde{b}_2}{\widetilde{b}_1+\widetilde{b}_2}\Big|\hspace{-0.1cm}\matr{\widehat{\Gamma}_{31}\\\widehat{\Gamma}_{32}}\hspace{-0.1cm}\Big|^2\Big)\,,
	\end{align}
	where the matrix $\widehat{\G}=\big(\axl(R^T\partial_xR)|\axl(R^T\partial_y R)\big)=\pj(\G)$ is in the form (see \cite{Ghiba2022})
	\begin{align}\widehat{\G}=\big(\axl(R^T\partial_xR)|\axl(R^T\partial_y R)\big)=
	\begin{pmatrix}
	\begin{array}{c c}
	\widehat{\Gamma}_{11}& \widehat{\Gamma}_{12} \\
	\widehat{\Gamma}_{21}&\widehat{\Gamma}_{22} \\
	\widehat{\Gamma}_{31}& \widehat{\Gamma}_{32}
	\end{array}
	\end{pmatrix}\in \R^{3\times 2}\,.
	\end{align}
	 	We set $\G_\square=\matr{\widehat{\Gamma}_{11}&\widehat{\Gamma}_{12}\\\widehat{\Gamma}_{21}&\widehat{\Gamma}_{22}}$ and $\G_\perp=\matr{\widehat{\Gamma}_{31}\\\widehat{\Gamma}_{32}}$. Thus we can write the $\Gamma$-limit minimization problem as\footnote{Note the four fold appearance of the harmonic mean $\mathcal{H}$, i.e. 
	 	\begin{align}
	 	\frac{2\,\mu\mu_c}{\mu+\mu_c}=\mathcal{H} (\mu,\mu_c)\,,\qquad\qquad\frac{\mu\lambda}{2\mu+\lambda}=\frac{1}{2}\mathcal{H}(\mu,\frac{\lambda}{2})\,,\qquad\qquad \frac{\widetilde{b}_1\widetilde{b}_3}{\widetilde{b}_1+\widetilde{b}_3}=\frac{1}{2}\mathcal{H}(\widetilde{b}_1,\widetilde{b}3)\,,\qquad\qquad \frac{2\,\widetilde{b}_1\widetilde{b}_2}{\widetilde{b}_1+\widetilde{b}_2}=\mathcal{H}(\widetilde{b}_1,\widetilde{b}_2)\,.
	 	\end{align}}
	\begin{align}\label{e.q.2D}
	\nonumber\tiny  E_{\G\text{-lim}}^{\text{2D}}(m,R)&=\int_{\os}h\Big\{\mu|\sym\big((R_1|R_2)^T\DR\.m-\id_2\big)|^2+\mu_c|\skew\big((R_1|R_2)^T\DR\.m-\id_2\big)|^2\\
	 &\qquad\qquad+\frac{2\mu\.\mu_c}{\mu+\mu_c}\big(\iprod{R_3,\partial_x m}^2+\iprod{R_3,\partial_y m}^2\big)+\frac{\mu\.\lambda}{2\mu+\lambda}\tr\Big((R_1|R_2)^T\DR\.m-\id_2\Big)^2\\
	&\qquad\qquad+\mu\frac{L_c^2}{2}\Big(\widetilde{b}_1|\sym \widehat{\G}_\square|^2+\widetilde{b}_2|\skew \nonumber\widehat{\G}_\square|^2+\frac{\widetilde{b}_1\,\widetilde{b}_3}{\widetilde{b}_1+\widetilde{b}_3}\tr(\widehat{\G}_\square)^2+2\,\frac{\widetilde{b}_1\,\widetilde{b}_2}{\widetilde{b}_1+\widetilde{b}_2}|\widehat{\G}_\perp|^2\Big)\Big\}\,{\rm{dx}}\,.
	\end{align}
	If we assume that in the underlying Cosserat bulk curvature energy we have the uni-constant expression 
	\begin{align}
	 W_{\text{curv}}^{\text{3D}}(R^T\DR R)&=\frac{\mu L_c^2}{2}|R^T\DR R|^2=\frac{\mu L_c^2}{2}|\DR R|^2=\mu L_c^2\Big(|\axl(R^T\partial_xR)|^2+|\axl(R^T\partial_y R)|^2+|\axl(R^T\partial_zR)|^2\Big)\\
 	 \nonumber&=\frac{\mu L_c^2}{2}\Big(|(R^T\partial_xR)|^2+|(R^T\partial_y R)|^2+|(R^T\partial_zR)|^2\Big)\,,
	\end{align}
	then the homogenized curvature energy is given by \cite{birsan2016dislocation, fonseca2001inadequacy}
	\begin{align}\label{homo para 1}
	\nonumber W_{\text{curv}}^{\text{hom}}(R^T\DR R)&=\inf_{A\in \so(3)}W_{\text{curv}}^{\text{3D}}\Big(\axl(R^T\partial_xR),\axl(R^T\partial_y R),\axl(A)\Big)\\
	&=\mu\frac{L_c^2}{2}\Big(|R^T\partial_xR|^2+|R^T\partial_y R|^2\Big)=\mu\frac{L_c^2}{2}|R^T\DR R|^2=\mu L_c^2\,|\widehat{\G}|_{\R^{3\times 2}}^2=\mu L_c^2\,|\pj(\Gamma)|_{\R^{3\times 2}}^2\,.
	\end{align}

%
	\subsection{Alternative engineering ad-hoc dimensional descent}
	In \cite{neff2004geometrically} the three-dimensional Cosserat model has been reduced to a flat shell problem by proposing an engineering ansatz for the deformation $\varphi$ and the microrotation $R$ over the shell thickness. We let again $m\col \os\subset\R^2\to\R^3$ denote the midsurface deformation, $\overline{U}:=R^T(\DR \.m|R_3)$ the non-symmetric membrane stretch tensor and $R\col \Ob_h\subset\R^2\to \SO(3)$ denotes the microrotation tensor field with $R^T\DR R\cong(R^T\partial_xR,R^T\partial_yR)$. Since we are only interested in the membrane like response, we will neglect terms related to bending effects right away while keeping the curvature change\footnote{The missing Cosserat bending terms scaling with $h^3$ are of the type \cite[(4.5)]{neff2004geometrically} 
	\begin{align}
	\frac{h^3}{12}\Big\{\mu|\sym \Big(R^T(\DR R_3|0)\Big)|^2+\mu_c|\skew \Big(R^T(\DR R_3|0)\Big)|^2+\frac{\mu\lambda}{2\mu+\lambda}\tr\Big(\sym \Big(R^T(\DR R_3|0)\Big)\Big)^2\Big\}\,,
	\end{align}
	and the uni-constant case would appear for $\mu=\mu_c$, $\lambda=0$.} scaling with $h$.

	The dimensionally reduced energy reads then \cite[(4.5)]{neff2004geometrically}
	\begin{align}\label{compa ener}
	\nonumber E_{\text{eng}}^{\text{2D}}&=\int_{\os}h\Big\{\mu|\sym(\overline{U}-\id_3)|^2+\mu_c|\skew(\overline{U}-\id_3)|^2+\frac{\mu\lambda}{2\mu+\lambda}\tr(\overline{U}-\id_3)^2+\mu\frac{L_c^2}{2}|R^T\DR R|^2\Big\}\,\rm{dx}\\
	\nonumber&=\int_{\os}h\Big\{\mu|\dev\sym(\overline{U}-\id_3)|^2+\mu_c|\skew(\overline{U}-\id_3) |^2+\Big(\underbrace{\frac{\mu\lambda}{2\mu+3\lambda}+\frac{\mu}{3}\.\mu}_{=:\,\frac{\kappa^{\text{hom}}}{2}}\Big)\tr(\overline{U}-\id_3)^2+\mu\frac{L_c^2}{2}|R^T\DR R|^2\Big\}\,{\rm{dx}}\\
	&=\int_\os h\, \Big\{\mu|\dev\sym(R^T(\DR m|R_3)-\id_3)|^2+\mu_c|\skew(R^T(\DR m|R_3)-\id_3)|^2\\
	\nonumber&\qquad\qquad+\frac{\kappa^{\text{hom}}}{2}\tr(R^T(\DR m|R_3)-\id_3)^2 +\mu\frac{L_c^2}{2}|\DR R|^2\Big\}\,{\rm{dx}}\\
	\nonumber&=\int_{\os} h\,\Big\{\mu|\underbrace{\sym \big((R_1|R_2)^T\DR\.m-\id_2\big)}_{\text{shear-stretch energy}}|^2+\mu_c|\underbrace{\skew \big((R_1|R_2)^T\DR\.m-\id_2\big)}_{\text{drill-energy}}|^2\\
	\nonumber&\qquad\qquad+\frac{\mu+\mu_c}{2}\underbrace{\Big(\iprod{R_3,\partial_xm}^2+\iprod{R_3,\partial_ym}^2\Big)}_{\text{transverse shear energy}}+\frac{\mu\lambda}{2\mu+\lambda}\underbrace{\tr\Big((R_1|R_2)^T\DR\.m-\id_2\Big)^2}_{\text{elogational stretch energy}}+\mu\underbrace{\frac{L_c^2}{2}|R^T\DR R^2|}_{\text{curvature energy}}\Big\}\,{\rm{dx}}.
	\end{align} 
	
	Letting $\mu_c\to \infty$ in the reduced membrane model implies on the on hand that $R_3={n}_m$ is normal to the midsurface $m$ and on the other hand $\skew(R^T(\DR m|{n}_m))=0$ implies $R=\text{polar}(\DR m|{n}_m)$ (tri\`{e}dre cache\'{e}).
	
	In contrast to the representation of the energy in (\ref{compa ener}) the rigorously derived $\Gamma$-limit membrane model \cite{neff2007geometrically1}  has the energy (see equation (\ref{e.q.2D}))
	\begin{align}\label{curv}
	 E_{\G\text{-lim}}^{\text{2D}}(m,R)&=\int_{\os}h\Big\{\mu|\sym\big((R_1|R_2)^T\DR\.m-\id_2\big)|^2+\mu_c|\skew\big((R_1|R_2)^T\DR\.m-\id_2\big)|^2\\
	\nonumber &\qquad\qquad+\frac{2\mu\.\mu_c}{\mu+\mu_c}\big(\iprod{R_3,\partial_xm}^2+\iprod{R_3,\partial_ym}^2\big)+\frac{\mu\.\lambda}{2\mu+\lambda}\tr\Big((R_1|R_2)^T\DR\.m-\id_2\Big)^2+W_{\text{curv}}^{\text{hom}}(\widehat{\G})\Big\}\,\rm{dx}\,,
	\end{align}
	where $\widehat{\G}=\big(\axl(R^T\partial_xR)|\axl (R^T\partial_y R)\big)$. 
	Thus, the engineering formulation in (\ref{compa ener}) coincides with the membrane $\Gamma$-limit if and only if
	\begin{align}\label{harm arit}
	\mathcal{A}(\mu,\mu_c)=\frac{\mu+\mu_c}{2}&=\frac{2\mu\,\mu_c}{\mu+\mu_c}=\mathcal{H}(\mu,\mu_c)\qquad\Longleftrightarrow \qquad \mu=\mu_c\,,
	\end{align}
	and 
	\begin{align}
	|R^T\DR R|^2=W_{\text{curv}}^{\text{hom}}(\widehat{\G})\qquad\Longleftrightarrow\qquad \widetilde{b}_1=\widetilde{b}_2=1\,, \qquad\widetilde{b}_3=0\,.
	\end{align}
	In (\ref{compa ener})$_2$, we are also led to define the appropriate modified bulk modulus $\kappa^{\text{hom}}$ via\footnote{In linear elasticity theory for the displacement $u\col \mathcal{U}\hspace{-0.1cm}\subset\hspace{-0.1cm} \R^3\hspace{-0.1cm}\to\hspace{-0.1cm} \R^3$, the common bulk modulus $\kappa$ appears in the form $\mu\,|\dev\sym \DR u|^2+\frac{\kappa}{2}\tr(\DR u)^2$ and not as $\mu\,|\dev\sym \DR u|^2+\frac{\kappa}{3}\tr(\DR u)^2$, which would be more natural from the perspective of orthogonality of $\dev\sym\DR u$ and $\tr(\DR u)\cdot\id_3$.}
	\begin{align}
	\frac{\kappa^{\text{hom}}}{2}&:=\frac{\mu\lambda}{2\mu+\lambda}+\frac{\mu}{3}=\frac{2\mu}{3}\.\frac{2\lambda+\mu}{2\mu+\lambda} \qquad (\text{effective two-dimensional bulk modulus})\,.
	\end{align}
	Since we will need $\kappa^{\text{hom}}>0$ for our subsequent regularity  analysis, (\ref{harm arit})$_2$ implies $2\mu+\lambda>0$ and $ 2\lambda+\mu>0$.
	One can show that the latter implies for the engineering Poisson number	$\nu:=\displaystyle\frac{\lambda}{2(\mu+\lambda)}$ the bound $\nu>-\frac{1}{2}$ (instead of $\nu>-1$ for three-dimensional linear elasticity).\footnote{$2\lambda+\mu>0$ and $\mu>0$ implies $2\lambda+2\mu=2(\lambda+\mu)>0$. Therefore, $\nu=\frac{\lambda}{2(\mu+\lambda)}>-\frac{1}{2}\Leftrightarrow \frac{\lambda}{\mu+\lambda}>-1\Leftrightarrow\lambda>-(\mu+\lambda)\Leftrightarrow2\lambda+\mu>0$.}

	\section{The two-dimensional Euler-Lagrange equations}
	Henceforth, we skip all unnecessary material parameters in (\ref{compa ener}) in order to arrive at a compact representation. Again, we consider the midsurface deformation $m\col \os\subset\R^2\to \R^3$ and the orthogonal microrotation tensor $R\col\os\subset \R^2\to \SO(3)$. We set $h=1$ and assume the normalization $\mu\frac{L_c^2}{2}=1$. Moreover, we set $\kappa=\frac{3\kappa^{\text{hom}}}{2}$. Thus, the corresponding energy function describing the two-dimensional membrane shell problem is 
	\begin{align}\label{energy}
	E(m,R)&:=\int_\os\hspace{-0.05cm}\mu|\dev\sym(R^T(\DR m|R_3)-\id_3)|^2\hspace{-0.05cm}+\hspace{-0.05cm}\mu_c|\skew(R^T(\DR m|R_3)-\id_3)|^2\hspace{-0.05cm}\nonumber\\&\qquad\quad+\hspace{-0.05cm}\frac{\kappa}{3}\tr(R^T(\DR m|R_3)-\id_3)^2 +|\DR R|^2\,{\rm{dx}}\,.
	\end{align}
	We assume $\mu,\mu_c,\kappa$ to be positive.
	Remember that we have defined a linear operator $\Pj:\R^{3\times3}\to\R^{3\times3}$ by
	\begin{align}
	\Pj_{\mu,\mu_c,\kappa}(X)=\Pj(X)=\sqrt{\mu}\,\dev\sym X+\sqrt{\mu_c}\,\skew X
	+\frac{\sqrt{\kappa}}3\,(\tr X)\id_3.
	\end{align}
	Using the mutual orthogonality of $\dev\sym X$, $\skew X$, and $(\tr X)\id_3$,
	we can write down the functional in a simplified form, it reads
	\begin{align}\label{function E}
	E(m,R)=\int_\os|\Pj(R^T(\DR m|R_3)-\id_3)|^2+|\DR R|^2\,\rm{dx}.
	\end{align}
	
	Now we are going to calculate the Euler-Lagrange equations for the dimensionally reduced problem based on $E$.
	The first variation of $E$ in the direction of
	$(\vartheta,0):\Ob\to\R^3\times\R^{3\times3}$ is
	\begin{align}
 \delta E(m,R;\vartheta,0)&=2\int_\os \iprod{\Pj(R^T(\DR m|R_3)-\id_3), \Pj(R^T(\DR \vartheta|0))}\,{\rm{dx}}=2\int_\os \iprod{\Pj(R^T(\DR m|0)-(\id_2|0)), \Pj(R^T(\DR\vartheta|0))}\,{\rm{dx}},
	\end{align}
	and the first variation in the direction of
	$(0,Q):\Ob\to\R^3\times\R^{3\times3}$
	with $Q(x)\in \text{T}_R\SO(3)$ for almost all $x\in\Ob$ is
	\begin{align}
	\nonumber\delta E(m,R;0,Q)&=2\int_\os[\iprod{\Pj(R^T(\DR m|R_3)-\id_3), \Pj(Q^T(\DR m|0))}
	+\iprod{\DR R, \DR Q}]\,{\rm{dx}}\\
	&=2\int_\os[\iprod{\Pj(R^T(\DR m|0)-(\id_2|0)), \Pj(Q^T(\DR m|0))}
	+\iprod{\DR R, \DR Q}]\,{\rm{dx}}.  
	\end{align}
	Using $\Pj^*=\Pj$ and $\Pj(X^T)=\Pj(X)^T$ such that $\Pj^*\Pj=\Pj^2$, and observing $\pj (v_1|v_2|v_3)=(v_1|v_2)$,
	we rewrite these as
	\begin{align}
	\nonumber\delta E(m,R;\vartheta,0)&=2\int_\os \iprod{R\,\Pj^2(R^T(\DR m|0)-(\id_2|0)),(\DR \vartheta|0)}_{\R^{3\times 3}}\,{\rm{dx}}\\\nonumber&=2\int_\os\iprod{\pj(R\,\Pj^2(R^T(\DR m|0)-(\id_2|0))), \DR \vartheta}_{\R^{3\times 2}}\,{\rm{dx}},\\
	\delta E(m,R;0,Q)&=2\int_\os[\iprod{\Pj((\DR m|0)^TR-(\id_2|0)), \Pj((\DR m|0)^TQ)}_{\R^{3\times 3}}
	+\iprod{\DR R, \DR Q}]\,{\rm{dx}}\\
\nonumber	&=2\int_\os[\iprod{(\DR m|0)\Pj^2((\DR m|0)^TR-(\id_2|0)), Q}+\iprod{\DR R, \DR Q}]\,{\rm{dx}}.
	\end{align}
	The pair of Euler-Lagrange equations then consists of
	\begin{align}\label{ELa1}
	\Div[\pj(\zwei R\,\Pj^2(R^T(\DR m|0)-(\id_2|0)))]=0\,,
	\end{align}
	and
	\begin{align}\label{ELa2}
	\Delta R-(\DR m|0)\,\Pj^2((\DR m|0)^TR-(\id_2|0))\perp {{\rm{T}}_R\SO(3)}\,.
	\end{align}
	Note that it is {\em not\/} true that $X^T\Pj^2(X)=X^T\Pj^*\Pj X$ is symmetric for
	all matrices $X$; this is because $\Pj$ is not a matrix. Therefore,
	$(\DR m|0)\Pj^2(\DR m|0)^TR$ is not automatically orthogonal to
	$\text{T}_R\SO(3)$. And this term, being formally only in $L^1$ due to
	$\DR m$ being in $L^2$, makes the structure of the equation interesting,
        as explained in Section \ref{sect:backgr}.

For readability, we introduce a product which shares aspects of scalar products and matrix multiplication. We define $\dott:\R^{3\times 2}\times\R^{3\times2}\to\R^{3\times 3}$ by 
	\begin{align}\label{dott}
	B\dott C:=\frac{1}{2}\,B\,C^T=\frac{1}{2}\,(B|0){C^T\choose0}\,.
	\end{align}
Defining
\begin{align}
S(\DR m,R):=\pj(\zwei R\,\Pj^2(R^T(\DR m|0)-(\id_2|0))),
\end{align}
we rewrite the second term of (\ref{ELa2}) as
	\begin{align}
	\nonumber(\DR m|0)\Pj^2((\DR m|0)^TR-(\id_2|0))&=\DR m\dott\pj(\zwei \Pj^2(R^T(\DR m|0)-(\id_2|0)))\\
	&=\DR m\dott R^T\pj(\zwei R\,\Pj^2(R^T(\DR m|0)-(\id_2|0)))\\
	\nonumber&=\DR m\dott R^TS(\DR m,R)=\big(\DR m\dott S(\DR m,R)\big)R.
	\end{align}
        Noting that the projection of any matrix $X\in\R^{3\times3}$ to $\text{T}_R\SO(3)$ is
	$R\skew(R^TX)$, we find that the projection of $(\DR m\dott S(\DR m,R))R$ is
	$R\skew(R^T(\DR m\dott S(\DR m,R))R)=\skew\big(\DR m\dott S(\DR m,R)\big)R$.
        This means that the pair of Euler-Lagrange equations (\ref{ELa1})--(\ref{ELa2}) can be rewritten as
        \begin{align}\label{ELa1bis}
          &\Div S(\DR m,R)=0,
        \end{align}
        \begin{align}\label{ELa2bis}
          &\Delta R-\skew(\DR m\dott S(\DR m,R))R\quad\perp\quad T_R\SO(3).
        \end{align}
        The latter is a relation rather than an equation, but we can rewrite it as an equation. In geometric analysis, this is usually done using the second fundamental form of $\SO(3)$, but we present the calculation in a more elementary way. Our aim is to calculate the tangential part $(\Delta R)^\top$ of $\Delta R$.

Differentiating $R\,R^T\equiv \id_3$ gives
\begin{align}\label{h1}
0=\partial_i(RR^T)=(\partial_i\,R)R^T+R\,\partial_iR^T=2\sym(R\,\partial_iR^T)\,.
\end{align}
Differentiating $R^TR\equiv \id_3$ twice and summing over $i$, we find
\begin{align}
\nonumber 0&=(\Delta R^T)R+R^T\Delta R+2\sum_i\partial_iR^T\partial_iR=(\Delta R)^TR+R^T\Delta R+2\sum_i\partial_iR^T\partial_iR\\
&=2\sym(R^T\Delta R)+2\sum_i\partial_iR^T\partial_iR,
\end{align}
implying
\begin{align}\label{h2}
\sym(R^T\Delta R)=-\sum_i\partial_iR^T\partial_iR.
\end{align}
For any fixed matrix $R\in \SO(3)$, we have ${\rm{T}}_R\SO(3)=R\,\so(3)$, where $\so(3)$
is the space of skew-symmetric matrices in $\R^{3\times3}$. The projections of
any $X\in\R^{3\times3}$ to ${\rm{T}}_R\SO(3)$ or its orthogonal complement
$[{\rm{T}}_R\SO(3)]^T$ therefore are
\begin{align}
X^\top=R\skew(R^TX),\quad\qquad X^\bot=R\sym(R^TX).
\end{align}
Therefore, we can calculate the orthogonal component of $\Delta R$ as 
\begin{align}
  (\Delta R)^\bot=&R\sym(R^T\Delta R)=-\sum_iR\,\partial_iR^T\partial_iR
  =-\sum_i\skew(R\,\partial_iR^T)\partial_iR.
\end{align}
We have used (\ref{h2}) in the second ``$=$'', and (\ref{h1}) in the third.
We now abbreviate
\begin{align}
  \Omega_R:=\matr{(\Omega_R)_1\\(\Omega_R)_2}=-\matr{R\,\partial_xR^T\\ R\,\partial_y R^T}=\SKEW,\quad\quad\qquad \Omega_R\cdot\DR R:=\sum_{i=1}^2(\Omega_R)_i\,\partial_iR\in \R^{3\times 3}\,,
\end{align}
and hence have
\[
  (\Delta R)^\top=\Delta R-(\Delta R)^\bot=\Delta R-\Omega_R\cdot\DR R\,.
\]
Combining with the result of (\ref{ELa2bis}), we have calculated the tangential part of the left-hand side of (\ref{ELa2}) as $$\Delta R-\Omega_R\cdot\DR R-\skew\big(\DR m\dott S(\DR m,R)\big)R\,,$$ and thus have derived the Euler-Lagrange equations in their final form. We summarize
\begin{center}
\fbox{\begin{minipage}{14cm}
        \begin{align}\label{EL1}
          \Div S(\DR m,R)&=0,\\
               \label{EL2}
         \Delta R-\Omega_R\cdot\DR R-\skew\big(\DR m\dott S(\DR m,R)\big)R&=0,
        \end{align}
        where here
        \[
          S(\DR m,R):=\pj(\zwei R\,\Pj^2(R^T(\DR m|0)-(\id_2|0))),\qquad
          \Omega_R:=-\matr{\skew(R\,\partial_xR^T)\\\skew (R\,\partial_yR^T)}.
          \]
\end{minipage}}
\end{center}

\begin{remark} In engineering language, (\ref{EL1}) is the balance of forces,
  while \ref{EL2} is the balance of angular momentum equation. The tensor
  \begin{align}
    T(\DR m,R):=\zwei\,\Pj^2\,(R^T(\DR m|0)-(\id_2|0))
  \end{align}
  is the {\em non-symmetric Biot-type stress tensor\/} (symmetric if
  $\mu_c=0$), while
  \begin{align}
    S(\DR m,R)=\pi_{12}(R\,T(\DR m,R))
  \end{align}
  is the {\em first Piola-Kirchhoff type force-stress tensor.} Note the
  analogy to the corresponding tensors in the 3D-Cosserat model presented in (\ref{S_1}) and
  (\ref{S_11}).
\end{remark}
        
        \section{Regularity}

        The objective of this section is to prove our main theorem.
        \begin{theorem}[interior regularity]\label{MainTh}
          Every weak solution $(m,R)\in W^{1,2}(\os,\R^3\times \SO(3))$ of (\ref{EL1})-(\ref{EL2}) is smooth on the interior of $\os$.
        \end{theorem}
        
  \begin{remark}
  	
  	Due to the results in \cite{neff2004geometrically, neff2007geometrically, neff2007geometrically1}, we know that energy minimizers  to problem (\ref{energy}) exist and these are weak solutions $(m,R)\in W^{1,2}(\os,\R^3\times \SO(3))$ of (\ref{EL1})-(\ref{EL2}). Since the problem is highly nonlinear, uniqueness can neither be shown nor is it expected.
  \end{remark}
	\subsection{H\"older regularity}\label{suse:hoel}
	
	We observe that the last term in (\ref{EL2}) is, up to ``skew'' and
	the harmless factor $R$,
	the product of a ``gradient'' $\DR m$ with a divergence-free quantity $S(\DR m,R)$,
	with both factors in $L^2$.
	As we know from \cite{coifman1993compensated}, such a product is in the Hardy space ${\mathcal H}^1$
	rather than just in $L^1$, and we will use arguments from \cite{riviere2008partial}
	that tell us how to handle the additional $R$ factor. A standard source for the Hardy space $\mathcal{H}^1$ is in Chapter III of Stein's book \cite{stein1993harmonic}.
	Note that \cite{riviere2008partial}(see also \cite{tel2019regularity}) is about harmonic maps in $\ge3$ dimensions, and it is
	Rivi\`ere's paper \cite{riviere2007conservation} about two-dimensional harmonic maps that is
	mostly the basis of what we are doing here. Schikorra \cite{schikorra2010remark} found some
	simplification to the arguments of \cite{riviere2007conservation} and \cite{riviere2008partial}, and the most
	accessible account of all these arguments to date is the textbook \cite{giaquinta2013introduction}
	which allows us to handle the Euler-Lagrange equation (\ref{EL2}) quite
	flexibly. Note that our equation (\ref{EL2}) is more general than the
	equations of the form $\Delta R-\Omega\cdot DR=0$ studied in those papers and the book \cite{giaquinta2013introduction},
	since we have the extra term $-R\skew(\DR m\dott S(\DR m,R))$ of order $0$ in
	$R$. We are lucky that we have the additional structure coming from
	$S(\DR m,R)$ being divergence-free, again implying that up to a bounded factor the extra term is in ${\mathcal H}^1$. Without that additional information, we would not know how to
	incorporate that into the existing regularity theory.
	
        It will be crucial to use Morrey norms, at least locally. We say that $u\in L^p(U)$ is in the Morrey space $M^{p,s}(U)$ if
        \begin{align}
          [u]_{M^{p,s}(U)}^p:=\sup\Big\{r^{-s}\int_{B_r(x_0)\cap U}|u|^p\,{\rm{dx}}\;\Big|\;x_0\in U,r\in(0,1)\Big\}<\infty.
        \end{align}
        Having this, we define the Morrey norm by $\|u\|_{M^{p,s}(U)}:=[u]_{M^{p,s}(U)}+\|u\|_{L^p(U)}$.    

	We need the following lemmas. The first one is a special case of Lemma~A.1
	in \cite{schikorra2010remark}, in the spirit of similar estimates from \cite{coifman1993compensated}. This is
	where Hardy-BMO duality comes in as a hidden ingredient of our proof.

	\begin{lemma}\label{L1}
		There is a constant $C$ such that for all choices of $x_0\in\R^2$, $r>0$,
		and functions
		$a\in W^{1,2}(B_{2r}(x_0))$, $\Gamma\in L^2(B_r(x_0),(\R^2)^*)$,
		$b\in W^{1,2}_0\cap L^\infty(B_r(x_0))$ with $\Div\Gamma=0$ in the weak sense
		on $B_r(x_0)$, we have
		\[
		\Big|\int_{B_r(x_0)}\langle\DR a,\Gamma\rangle\,b\;{\rm{dx}}\,\Big|
		\le C\,\|\Gamma\|_{L^2(B_r(x_0))}\,\|\DR b\|_{L^2(B_r(x_0))}\,
		\|\DR a\|_{M^{3/2,1/2}(B_{2r}(x_0))}.
		\]
	\end{lemma}
	Another one, due to Rivi\`ere \cite{riviere2007conservation} and Schikorra \cite{schikorra2010remark}, can be found
	as a special case of Theorem~10.57 in \cite{giaquinta2013introduction}.
	
	\begin{lemma}\label{L2}
		For every $\Omega\in L^2(B^2,(\R^2)^*\otimes \so(3))$, there
		exists $G\in W^{1,2}(B^2,\SO(3))$ such that
		\begin{align}
		\Div(G^{-1}\Omega G-G^{-1}\DR G)=0\mbox{ \ \ in }B^2
		\end{align}
		and \footnote{$(\R^2)^*\otimes \so(3)$ is isomorphic to $\so(3)\times \so(3)$. }
		\begin{align}
		\|\DR G\|_{L^2(B^2)}+\|G^{-1}\Omega G-G^{-1}\DR G\|_{L^2(B^2)}
		\le 3\,\|\Omega\|_{L^2(B^2)}.
		\end{align}
	\end{lemma}
	
	We also need a version of the Hodge decomposition theorem. This one is a
	special case of \cite[Corollary 10.5.1]{iwaniec2001geometric}, adapted
        from the differential forms version to 2-dimensional vector calculus
        as in \cite[Corollary 10.70]{giaquinta2013introduction}	
	\begin{lemma}\label{L3}
	  Let $p\in(1,\infty)$. On $B_r(x_0)\subset\R^2$, every 1-form
          $V\in L^p(B_r(x_0),(\R^2)^*)$ can be decomposed uniquely as
		\[
		V=\DR\alpha+\DT\beta+h,
		\]
          where ${\alpha}\in W^{1,p}(B_r(x_0))$, $\beta\in W^{1,p}_0(B_r(x_0))$,
	  and $h\in C^\infty(B_r(x_0),(\R^2)^*)$ is harmonic. Moreover, there
          is a constant $C$ depending only on $p$, such that
          \begin{align}
		\|\alpha\|_{W^{1,p}(B_r(x_0))}+\|\beta\|_{W^{1,p}(B_r(x_0))}+\|h\|_{L^p(B_r(x_0))}
		\le C\,\|V\|_{L^p(B_r(x_0))}.
	  \end{align}
	\end{lemma}
	
        We now start our regularity proof. Our first step is local H\"older continuity.

        \begin{proposition}\label{Prop:Hoel} Assume that $(m,R)\in W^{1,2}(\os,\R^3\times \SO(3))$ is a weak solution of (\ref{EL1})--(\ref{EL2}). Then there is $\beta>0$ such that $m$ and $R$ are $C^{0,\beta}$-H\"older continuos locally on $\omega$.
        \end{proposition}
          
        {\bf Proof. }
        We write $B_\rho$ for any ball $B_\rho(x_0)\subset\os$. We assume $r$ to be small enough such that $B_{2r}(x_0)\subset\os$. We will collect more smallness conditions on $r$ during the proof.

	We choose $G$ according to Lemma \ref{L2} and find, abbreviating
	$\Omega^G:=G^{-1}\Omega_R G-G^{-1}\DR G$,
	\begin{align}
	\nonumber\DivR(G^{-1}\DR R)&=\DR(G^{-1})\cdot\DR R+G^{-1}\Delta R\\
	&=-G^{-1}(\DR G)G^{-1}\cdot\DR R+G^{-1}\Omega\cdot\DR R
	+G^{-1}\skew\big(\DR m\dott S(\DR m,R)\big)R\\
	\nonumber&=\Omega^G\cdot G^{-1}\DR R+G^{-1}\skew\big(\DR m\dott S(\DR m,R)\big)R\,.
	\end{align}
	Now we Hodge-decompose $G^{-1}\DR R$ according to Lemma \ref{L3}. We find
	$f\in W^{1,2}(B_r,\R^{3\times 3})$,
	$g\in W^{1,2}_0(B_r,\R^{3\times3})$ with $dg=0$, and a component-wise harmonic
	$1$-form $h\in C^\infty(B_r,L(\R^2,\R^{3\times3}))$
	such that
	\begin{align}\label{HD1}
	G^{-1}\DR R=\DR f+\DT g+h
	\end{align}
	almost everywhere in $B_r$. Using the well-known relations $\DivR\,\DR=\DivT\DT=\Delta$ and $\DivR\,\DT=\DivT\DR=0$, we calculate 
	\begin{align}\label{Del1}
	\Delta f=\DivR\,\DR f=\DivR(G^{-1}\DR R)=\Omega^G\cdot G^{-1}\DR R
	+G^{-1}\skew\big(\DR m\dott S(\DR m,R)\big)R\,,
	\end{align}
	and
	\begin{align}\label{Del2}
	\Delta g=\DivT\DT g=\DivT(G^{-1}\DR R)=\DT G^{-1}\cdot\DR R=\DivR((\DT G^{-1})(R-R_0))\,,
	\end{align}
	for any constant $R_0\in\R^{3\times 3}$ (not necessarily a rotation).
        Both terms on the right-hand side, multiplied with some $\varphi\in W^{1,3}_0(B_r,\R^{3\times3})$, can be estimated using Lemma \ref{L1}. Choosing
        $a:=R^{k\ell}$, $b:=(G^{-1})^{jk}\varphi^{i\ell}$, $\Gamma^s:=(\Omega^G)^{ij}_s$,
        we find
        \begin{align}\label{rhs1}
        \int_{B_r}\iprod{\Omega^G\cdot G^{-1}\DR R,\varphi}\,{\rm{dx}}
        \le C\|\Omega^G\|_{L^2(B_2)}(\|\DR G\|_{L^2(B_r)}\|\varphi\|_{L^\infty(B_r)}+\|\DR\varphi\|_{L^2(B_r)})\|\DR R\|_{M^{3/2,1/2}(B_r)},
        \end{align}
        and choosing $a:=m^j$, $b:=(G^{-1}_{ij})R^{k\ell}\varphi^{i\ell}$,
        $\Gamma^s:=S(\D m,R)^k_s$, we have
        \begin{align}\label{rhs2}
        &\int_{B_r}\iprod{(G^{-1}\skew\big(\D m\dott S(\D m,R)\big)R),\varphi}\,{\rm{dx}}\\
        \nonumber&\qquad\le C\|S(\D m,R)\|_{L^2(B_r)}\Big(\|\DR G\|_{L^2(B_r)}\|\varphi\|_{L^\infty(B_r)}\|\DR R\|_{L^2(B_r)}\|\varphi\|_{L^\infty(B_r)}+\|\DR\varphi\|_{L^2(B_r)}\Big)\|\D m\|_{M^{3/2,1/2}(B_r)}.
        \end{align}
	We assume $\eps\in(0,\eps_0)$ with some $\eps_0>0$ to be determined. Choosing $r>0$ small enough, we may assume $\|\Omega_R\|_{L^2(B_r)}\le\eps$ and $\|\DR m\|_{L^2(B_r)}\le\eps$.

        We let $T:=\{\varphi\in C^\infty_0(B_r,\R^{3\times3})\,|\,\|\DR\varphi\|_{L^3(B_r)}\le1\}$. Combining the duality of $L^{3/2}$ and $L^3$ and (\ref{Del1}) with (\ref{rhs1}) and (\ref{rhs2}), we find, using $\varphi=0$ on $\partial B_r$,
	\begin{align}\label{HA1}
	\nonumber \|\DR f\|_{L^{3/2}(B_r)}
		&\leq C\sup_{\varphi\in T}\int_{B_r}\<\DR f,\DR\varphi\>\,{\rm{dx}}\\
	\nonumber &=C\sup_{\varphi\in T}\int_{B_r}\<\Omega^G\cdot G^{-1}\DR R
	+G^{-1}\skew\big(\DR m\dott S(\DR m,R)\big)R\,,\,\varphi\>\,{\rm{dx}}\\
	\nonumber&\leq C\sup_{\varphi\in T}\|\Omega^G\|_{L^2(B_r)}(\|\DR G\|_{L^2}\|\varphi\|_{L^\infty(B_r)}
	+\|\DR\varphi\|_{L^2(B_r)})\|\DR R\|_{M^{3/2,1/2}(B_{2r})}\\
	\nonumber&\quad\quad+C\sup_{\varphi\in T}\|S(Dm,R)\|_{L^2(B_r)}\Big(\|\DR G\|_{L^2(B_r)}\|\varphi\|_{L^\infty(B_r)}
	+\|\DR R\|_{L^2(B_r)}\|\varphi\|_{L^\infty(B_r)}\\
	&\quad\qquad+\|\DR\varphi\|_{L^2(B_r)}\Big)\|\DR m\|_{M^{3/2,1/2}(B_{2r})}\\
	\nonumber&\leq C\sup_{\varphi\in T}\|\Omega_R\|_{L^2(B_r)}(\|\Omega_R\|_{L^2(B_r)}\|\varphi\|_{L^\infty(B_r)}
	+\|\DR\varphi\|_{L^2(B_r)})\|\DR R\|_{M^{3/2,1/2}(B_{2r})}\\
	\nonumber&\quad\quad+C\sup_{\varphi\in T}(\|\DR m\|_{L^2(B_r)}+r)\Big(\|\Omega_R\|_{L^2(B_r)}\|\varphi\|_{L^\infty(B_r)}
	+\|\DR R\|_{L^2(B_r)}\|\varphi\|_{L^\infty(B_r)}\\
	\nonumber&\qquad\quad+\|\DR\varphi\|_{L^2(B_r)}\Big)
	\|\DR m\|_{M^{3/2,1/2}(B_{2r})}\\
	\nonumber&\leq C(\eps+r)\,r^{1/3}(\|\DR R\|_{M^{3/2,1/2}(B_{2r})}+\|\DR m\|_{M^{3/2,1/2}(B_{2r})}).
	\end{align}
	Here, in the second ``$\le$'', we have used Lemma \ref{L1}. And in the fourth
	``$\le$'', we have used
	$\|\varphi\|_{L^\infty(B_r)}\le Cr^{1/3}\|\DR\varphi\|_{L^3(B_r)}\le Cr^{1/3}$,
	$\|\DR\varphi\|_{L^2(B_r)}\le Cr^{1/3}\|\DR\varphi\|_{L^3(B_r)}\le Cr^{1/3}$,
	$\|\Omega_R\|_{L^2(B_r)}\le\eps$, and $\|\DR m\|_{L^2(B_r)}\le\eps$.
	
	Using (\ref{Del2}), we can also estimate the $L^{3/2}$-norm of $\DT g$. We find
	\begin{align}\label{HA2}
	\nonumber\|\DT g\|_{L^{3/2}(B_r)}&\leq C\sup_{\varphi\in T}\int_{B_r}\<\DT g,\DT\varphi\>
		\,{\rm{dx}}\\
	\nonumber&=C\sup_{\varphi\in T}\int_{B_r}\<\Delta g,\varphi\>\,{\rm{dx}}\\
	&=C\sup_{\varphi\in T}\int_{B_r}\<\DivR((\DT G^{-1})(R-R_{B_r}))\,,\,\varphi\>\,{\rm{dx}}\\
	\nonumber&=C\sup_{\varphi\in T}\int_{B_r}\<(\DT G^{-1})(R-R_{B_r})\,,\,\DR\varphi\>\,{\rm{dx}}\\
	\nonumber&\leq C\sup_{\varphi\in T}\|\DR\varphi\|_{L^3(B_r)}\|\DR G\|_{L^2(B_r)}
	\|R-R_{B_r}\|_{L^6(B_r)}\\
	\nonumber&\leq C\eps r^{1/3}\|\DR R\|_{M^{3/2,1/2}(B_{2r})}.
	\end{align}
	This time, we have used $\|\DR G\|_{L^2(B_r)}\le3\,\|\Omega_R\|_{L^2(B_r)}\le3\,\eps$,
	and the Sobolev embedding $W^{1,2/3}\hookrightarrow L^6$ for $R$.
	
	For $h$, being harmonic, we have the standard estimate
	\begin{align}
	\int_{B_\rho}|h|^{3/2}\,{\rm{dx}}
	\le C\Big(\frac{\rho}{r}\Big)^2\int_{B_r}|h|^{3/2}\,{\rm{dx}}\,,
	\end{align}
	for any $0<\rho<r$. From (\ref{HD1}), and then (\ref{HA1}) and (\ref{HA2}),
	we hence infer
	\begin{align}\label{It1}
	\nonumber\|\DR R\|_{L^{3/2}(B_\rho)}&=\|G^{-1}DR\|_{L^{3/2}(B_\rho)}\\
        \nonumber&\le\|h\|_{L^{3/2}(B_\rho)}+\|\DR f\|_{L^{3/2}(B_\rho)}+\|\DT g\|_{L^{3/2}(B_\rho)}\\
        &\le C\Big(\frac{\rho}{r}\Big)^{4/3}\|h\|_{L^{3/2}(B_r)}
          +\|\DR f\|_{L^{3/2}(B_\rho)}+\|\DT g\|_{L^{3/2}(B_\rho)}\\
        \nonumber&\le C\Big(\frac{\rho}{r}\Big)^{4/3}\|\DR R\|_{L^{3/2}(B_r)}
          +C(\|\DR f\|_{L^{3/2}(B_r)}+\|\DT g\|_{L^{3/2}(B_r)})\\
        \nonumber&\le C\Big(\frac{\rho}{r}\Big)^{4/3}\|\DR R\|_{L^{3/2}(B_r)}
          +C(\eps+r)r^{1/3}(\|\DR R\|_{M^{3/2,1/2}(B_{2r})}+\|\DR m\|_{M^{3/2,1/2}(B_{2r})}).
        \end{align}           
	
	Now we are going to derive a similar estimate for $\|\DR m\|_{L^{3/2}}$.
	Hodge-decompose $S(\DR m,R)$, i.e.
	\begin{align}
	\pj(\zwei \,R\,\Pj^2(R^T(\DR m|0)-(\id_2|0)))=\DT{\alpha}+\chi\,,
	\end{align}
	with ${\alpha}\in W^{1,2}_0(B_r,\R^{3\times2})$, and
	$\chi\in W^{1,2}(B_r,L(\R^2,\R^{3\times 2}))$ harmonic. This time, there is no term of the form $\DR\zeta$, since $\DivR$ of the left-hand side is $0$. This would imply that $\zeta$ is harmonic, and so would be $\DR\zeta$, which hence can be absorbed into $\chi$. We have, abbreviating $\Pj_R$ for the linear mapping
	$\xi\mapsto \zwei \, R\,\Pj^2(R^T(\xi))$,
	\begin{align}\label{dmstart}
	\nonumber\Delta\alpha&=\DivT\DT\alpha\\
	\nonumber&=\DivT[\pj(\zwei \,R\,\Pj^2(R^T(\DR m|0)-(\id_2|0)))]\\
	&=\DivT[\pj(\Pj_R(\DR m|0)-\zwei \,R\,\Pj^2(\id_2|0))]\\
        \nonumber&=(\DT|0)\cdot[\Pj_R(\DR m|0)]-\DivT[\pj(\zwei \.R\,\Pj^2(\id_2|0))]\\
	\nonumber&=\DT\Pj_R\cdot\DR m-\DivT[\pj(\zwei \.R\,\Pj^2(\id_2|0))]\\
        \nonumber&=\DivR[(\DT\Pj_R)(m-m_{B_r})]-\DivT[\pj(\zwei \.(R-R_{B_r})\Pj^2(\id_2|0))]\,.
	\end{align}    
	Using the same ideas as before, and defining $U:=\{\psi\in C^\infty_0(B_r,\R^{3\times2})\;\big|\;
	\|\DT\psi\|_{L^3(B_r)}\le1\}$, we estimate 
	\begin{align}\label{sechs27}
	\nonumber\|\DT\alpha\|_{L^{3/2}(B_r)}&\leq C\sup_{\psi\in U}\int_{B_r}\<\DT\alpha,\DT\psi\>
	  \,{\rm{dx}}\\
        \nonumber&=C\sup_{\psi\in U}\int_{B_r}\Big(\<(\DT\Pj_R)(m-m_{B_r})\,,\,\DR\psi\>
          -\<\pj(\zwei \.(R-R_{B_r})\Pj^2(\id_2|0))\,,\,\DT\psi\>\Big)\,{\rm{dx}}\\  
	&\leq C\sup_{\psi\in U}(\|\DR\psi\|_{L^3(B_r)}\|\DR\Pj_R\|_{L^2(B_r)}\|m-m_{B_r}\|_{L^6(B_r)}
	+\|\DR\psi\|_{L^3(B_r)}\|R-R_{B_r}\|_{L^{3/2}(B_r)})\\
	\nonumber&\leq C\sup_{\psi\in U}(\|\DR\psi\|_{L^3(B_r)}\|\DR R\|_{L^2(B_r)}\|\DR m\|_{L^{3/2}(B_r)}
	+r\|\DR\psi\|_{L^3(B_r)}\|\DR R\|_{L^{3/2}(B_r)})\\
	\nonumber&\leq C(\eps+r)\,r^{1/3}(\|\DR R\|_{M^{3/2,1/2}(B_{2r})}+\|\DR m\|_{M^{3/2,1/2}(B_{2r})}).
	\end{align}
	Proceeding exactly as above, we find
	\begin{align}\label{It2}
	\nonumber\|\DR m\|_{L^{3/2}(B_\rho)}&\leq C(\|\chi\|_{L^{3/2}(B_\rho)}
          +\|\DT\alpha\|_{L^{3/2}(B_\rho)}+\rho^{4/3})\\
        \nonumber&\leq C\Big(\frac{\rho}{r}\Big)^{4/3}\|\chi\|_{L^{3/2}(B_r)}
          +C(\|\DT\alpha\|_{L^{3/2}(B_\rho)}+\rho^{4/3})\\
        &\leq C\Big(\frac{\rho}{r}\Big)^{4/3}\|\DR m\|_{L^{3/2}(B_r)}
          +C(\|\DT\alpha\|_{L^{3/2}(B_r)}+\rho^{4/3})\\
        \nonumber&\leq C\Big(\frac{\rho}{r}\Big)^{4/3}\|\DR m\|_{L^{3/2}(B_r)}
          +C(\eps+r)r^{1/3}(\|\DR R\|_{M^{3/2,1/2}(B_{2r})}+\|\DR m\|_{M^{3/2,1/2}(B_{2r})})
          +C\rho^{4/3}.
        \end{align}        
	In order to do so, we have used
	\begin{align}\label{u+l}
	 C^{-1}\|\DR m\|_{L^{3/2}(B_s)}-Cs^{4/3}
	\leq \|\pi_{12}(\zwei \,R\,\Pj^2(R^T(\DR m|0)-(\id_2|0)))\|_{L^{3/2}(B_s)}\leq C(\|\DR m\|_{L^{3/2}(B_s)}+s^{4/3}).
	\end{align}
	We divide (\ref{It1}) and (\ref{It2}) by $\rho^{1/3}$ and combine them into
	\begin{align}
	\nonumber\rho^{-1/3}(\|\DR R\|_{L^{3/2}(B_\rho)}&+\|\DR m\|_{L^{3/2}(B_\rho)})\\
	\nonumber&\leq
	C\,\frac{\rho}{r^{4/3}}\,(\|\DR R\|_{L^{3/2}(B_r)}+\|\DR m\|_{L^{3/2}(B_r)})\\
	&\quad\quad+C(\eps+r)\Big(\frac{r}{\rho}\Big)^{1/3}
	(\|\DR R\|_{M^{3/2,1/2}(B_{2r})}+\|\DR m\|_{M^{3/2,1/2}(B_{2r})})+C\rho\\
        \nonumber&\le C\Big(\frac{\rho}r+(\eps+r)\Big(\frac{r}{\rho}\Big)^{1/3}\Big)
        (\|\DR R\|_{M^{3/2,1/2}(B_{2r})}+\|\DR m\|_{M^{3/2,1/2}(B_{2r})})+C\rho.
	\end{align}
        We now assume $r\le\eps$, where $\eps>0$ is yet to be determined.
        For formal reasons, we also add $\rho$ on both sides, which gives
        \begin{align}
	\nonumber\rho^{-1/3}(\|\DR R\|_{L^{3/2}(B_\rho)}&+\|\DR m\|_{L^{3/2}(B_\rho)})
        +\rho\\
	&\le C_0\Big(\frac{\rho}r+(\eps+r)\Big(\frac{r}{\rho}\Big)^{1/3}\Big)
        (\|\DR R\|_{M^{3/2,1/2}(B_{2r})}+\|\DR m\|_{M^{3/2,1/2}(B_{2r})}+2r)
        \end{align} 
        for some suitable constant $C_0$.
        Now we fix $\rho:=\frac{r}{12C_0}$ and $\eps:=(12C_0)^{-4/3}$, making
        $C_0(\frac{\rho}r+(\eps+r)(\frac{r}{\rho})^{1/3})=\frac16$. Abbreviating
        $\theta:=\frac1{12C_0}$, we thus have
        \begin{align}\nonumber
          (\theta r)^{-1/3}(\|\DR R\|_{L^{3/2}(B_{\theta r})}
            +\|\DR m\|_{L^{3/2}(B_{\theta r})})+\theta r  
          \le\frac16\,(\|\DR R\|_{M^{3/2,1/2}(B_{2r})}+\|\DR m\|_{M^{3/2,1/2}(B_{2r})}+2r).
        \end{align}
        This holds for all $B_{\theta r}(x_0)$ and $B_{2r}(x_0)\subset\omega$ which
        share the same center $x_0$. But clearly, we can replace $B_{2r}(x_0)$
        with any ball $B_s(y_0)\supset B_{2r}(x_0)$ which is still in $\omega$.
        All smallness assumptions made so far for $B_r(X_0)$ will now also
        be assumed for $s$, that is $s\le\eps$,
        $\|\Omega_R\|_{L^2(B_s(y_0))}\le\eps$, and $\|\DR m\|_{L^2(B_s(y_0))}\le\eps$.
        We then have
        \begin{align}\nonumber
          (\theta r)^{-1/3}(\|\DR R\|_{L^{3/2}(B_{\theta r}(x_0))}
            +\|\DR m\|_{L^{3/2}(B_{\theta r}(x_0))})+\theta r  
          \le\frac16\,(\|\DR R\|_{M^{3/2,1/2}(B_s(y_0))}+\|\DR m\|_{M^{3/2,1/2}(B_s(y_0))}+s),
        \end{align}
        which is valid for all $r,s,x_0,y_0$ such that
        $B_{2r}(x_0)\subset B_s(y_0)\subset\omega$. Then the $B_{\theta\rho}(x_0)$
        cover all of $B_{\theta s/2}(y_0)$. Hence, on the left-hand-side, we can
        take the infimum over all feasible $r$ and $x_0$, and find
        \begin{align}\nonumber
          &\|\DR R\|_{M^{3/2,1/2}(B_{\theta s/2}(y_0))}
            +\|\DR m\|_{M^{3/2,1/2}(B_{\theta s/2}(y_0))}+\frac{\theta s}{2}\\  
         &\qquad\le\frac12\,(\|\DR R\|_{M^{3/2,1/2}(B_s(y_0))}+\|\DR m\|_{M^{3/2,1/2}(B_s(y_0))}+s).
        \end{align}
        We may replace $s$ by $\frac{\theta}2\,s$ and iterate this, finding
        \begin{align}\nonumber
          &\|\DR R\|_{M^{3/2,1/2}(B_{(\theta/2)^ks}(y_0))}
            +\|\DR m\|_{M^{3/2,1/2}(B_{(\theta/2)^ks}(y_0))}\\  
         &\qquad\le 2^{-k}\,(\|\DR R\|_{M^{3/2,1/2}(B_s(y_0))}+\|\DR m\|_{M^{3/2,1/2}(B_s(y_0))}+s)
        \end{align}
        for all $k\in\N$. Now, for $r\approx(\theta/2)^ks$, we have
        $k\approx\frac{\log{r/s}}{\log(\theta/2)}$, and therefore
        $2^{-k}\approx (r/s)^{\frac{\log2}{\log(2/\theta)}}=:(r/s)^\beta$.
        Hence we have proven that, for all $r\le s$, the estimate
        \begin{align}\nonumber
          &\|\DR R\|_{M^{3/2,1/2}(B_r(y_0))}+\|\DR m\|_{M^{3/2,1/2}(B_r(y_0))} 
          \le Cr^\beta\,(\|\DR R\|_{M^{3/2,1/2}(B_s(y_0))}+\|\DR m\|_{M^{3/2,1/2}(B_s(y_0))}+s)
        \end{align}  
        holds. For $x_0\in B_{s/2}(y_0)$ and $r\le s/2$, we can apply the same
        with $B_s(y_0)$ replaced by $B_{2r}(y_0)\subset B_s(y_0)$, and hence find
        \begin{align}\nonumber
          &\|\DR R\|_{M^{3/2,1/2}(B_r(x_0))}+\|\DR m\|_{M^{3/2,1/2}(B_r(x_0))} 
          \le Cr^\beta\,(\|\DR R\|_{M^{3/2,1/2}(B_s(y_0))}+\|\DR m\|_{M^{3/2,1/2}(B_s(y_0))}+s),
        \end{align} 
        which implies
        \begin{align}\nonumber
           &\|\DR R\|_{M^{3/2,1/2+3\beta/2}(B_{s/2}(y_0))}+\|\DR m\|_{M^{3/2,1/2+3\beta/2}(B_{s/2}(y_0))} 
          \le C\,(\|\DR R\|_{M^{3/2,1/2}(B_s(y_0))}+\|\DR m\|_{M^{3/2,1/2}(B_s(y_0))}+s).
        \end{align}        
        This means
        \begin{align}\label{Mor}
           \DR R,\;\DR m\in M^{3/2,1/2+3\beta/2}_{\text{\rm loc}}(\omega).
        \end{align}  
        We now use the following well-known fact, which can be found in \cite[Theorem 5.7]{giaquinta2013introduction}, for example.

        \begin{lemma}[Morrey' Dirichlet growth criterion]\label{Lem:Mor}
          Assume $U\subset\R^n$ to be open, $u\in W^{1,p}_{\text{\rm loc}}(U)$,
          $Du\in M^{p,n-p+\eps}_{\text{\rm loc}}(U)$ for some $\eps>0$. Then $u\in C^{0,\eps/p}$.
        \end{lemma}

        With $p=\frac32$, $n=2$, the last estimate (\ref{Mor}) and Lemma
        \ref{Lem:Mor} imply
        $R,m\in C^{0,\beta}_{\text{\rm loc}}(\omega)$, which is the H\"older regularity
        asserted in Proposition \ref{Prop:Hoel}.\qed
        
        \begin{remark}
          It is essential that we are working in the critical dimension
          $n=2$ here, even though this may not be too obvious in the preceding
          proof which uses methods developed for supercritical dimensions.
          But the arithmetic of the exponents crucially uses $n=2$.
          In particuar, Lemma \ref{L1} for $n>2$ is only available with
          exponents
          adding up to $n$ instead of $(\frac32,\frac12)$. But we would not
          succeed in finding similarly good estimates in the corresponding
          Morrey spaces.
        \end{remark}
        
	\subsection{Higher regularity}\label{suse:high}

        In this subsection, we are going to complete the proof of Theorem \ref{MainTh}.
        
        {\bf Proof. }
	Remember we have the equations
	\begin{align}\label{Eul-Lag}
	\Div S(\DR m,R)&=0,\\
	\nonumber\Delta R-\Omega_R\cdot\DR R-\skew\big(\DR m\dott S(\DR m,R)\big)R&=0,  
	\end{align}
	where for $\xi\in \R^{3\times 2}$ we have defined
	\begin{align}
	S(\xi,R)=\pj(\zwei \,R\,\Pj^2(R^T(\xi,0)-(\id_2|0)))=\pj(\zwei \,R\,\Pj^T\Pj(R^T(\xi|R_3)-\id_3))\,,
	\end{align}
	and 
	\[
	|\Omega_R|\le C\,|\DR R|.
	\]
	Abbreviating $L_R(\xi):=\pj(\zwei \,R\,\Pj^2(R^T(\xi,0)))$, we rewrite the first
	equation (\ref{Eul-Lag}) as
	\begin{align}\label{Eq-m}
	\Div L_R(\D m)=\Div(\pj(\zwei \,R\,\Pj^2(R^T(\id_2|0)))).
	\end{align}
	For every $R\in \SO(3)$, $L_R:\R^{3\times2}\to\R^{3\times2}$ is a linear mapping
	satisfying the Legendre condition (uniform positivity) because of
	\begin{align}
	\nonumber\iprod{L_R(\xi),\xi}&=\iprod{\pj(\zwei \,R\,\Pj^2(R^T(\xi,0))),\xi}=\iprod{\zwei \,R\,\Pj^2(R^T(\xi,0)),(\xi,0)}\\
	&=\iprod{\zwei \,\Pj(R^T(\xi,0)),\,\Pj(R^T(\xi,0))}\geq \zwei \,\widehat{\lambda}\,|R^T(\xi,0)|^2=\zwei \,\widehat{\lambda}\,|\xi|^2,
	\end{align}
	where here $\widehat{\lambda}:=\min\{\mu,\mu_c,\kappa\}$ is independent of $R$, hence we have
	a uniformly elliptic operator $m\mapsto\Div L_R(\D m)$. For this operator,
	classical Schauder theory applies once it depends H\"older continuously on
	$x$ through $R(x)$. And it does, because we already know $R\in C_{\text{\rm loc}}^{0,\beta}$
	for some $\beta>0$.
	
	We use the following version of Schauder theory. The proof is well known, a good
	reference is \cite[Theorem 5.19]{giaquinta2013introduction} which reads as follows.
	
	\begin{lemma}\label{L5}
		Let $u\in W^{1,2}_{\text{\rm loc}}(\Ob,\R^m)$ be a solution to
		\begin{align*} 
		  \Div (A(x)\cdot \D u)&=-\Div \mathcal{F}\,,
		\end{align*}
                with $A$ satisfying the Legendre-Hadamard condition and having
		its components $A^{{\alpha}\beta}_{ij}$ in $C^{0,\sigma}_{\text{\rm loc}}(\Ob)$ for some $\sigma\in(0,1)$. If $\mathcal{F}^{\alpha}_i
		\in C^{0,\sigma}_{\text{\rm loc}}(\Ob)$, then also  $\DR u$ is of class
                $C^{0,\sigma}_{\text{\rm loc}}(\Ob)$.
	\end{lemma}
	
	From what was proven in the last section, we know that both $L_R$ and the
	right-hand side of (\ref{Eq-m}) are in $C_{\text{\rm loc}}^{0,\beta}$ locally, hence Lemma \ref{L5}
	implies that $m\in C_{\text{\rm loc}}^{1,\beta}$ for some $\beta>0$.
	
	This simplifies the discussion of the regularity of $R$, because the
	$\DR m$-terms in the equation for $\Delta R$ are now locally bounded. We can
	therefore rewrite it as
	\begin{align}\label{Eq-R}
	\Delta R+a(x,\DR R)=0\,,
	\end{align}
	where the function $a$ depends on $R$ and $\DR m$ additionally, but
	those are locally bounded. The function satisfies
	\begin{align}\label{DR+1}
	|a(x,\DR R)|\le C(|\DR R|^2+1).
	\end{align}
	Since $\DR R\in L^2$, this means that $\Delta R$ is in $L^1$, but $L^1$ is
	just not enough to perform regularity theory for $R$. However, the structure
	of the equation almost allows to apply the higher regularity theory for
	harmonic maps, where we could deal with $C|\DR R|^2$ instead of $C(|\DR R|^2+1)$
	on the right-hand side. A simple formal trick will care for that
	condition. Let
	\begin{align}
	u(x)=(u_0(x),u_1(x)):=(R(x),x_1)
	\end{align}
	with values in $\SO(3)\times\R$. Then, letting $\tilde{a}(x,\DR \.u):=(a(x,Du_0),0)$,
	we have
	\begin{align}
	\Delta u+\tilde{a}(x,\DR \.u)=0,
	\end{align}
	where here
	\begin{align}
	|\tilde{a}(x,\DR \.u)|=|a(x,\DR R)|\le C(|\DR R|^2+1)=C(|\DR \.u_0|^2+|\DR \.u_1|^2)=C\,|\DR \.u|^2.
	\end{align}
	Now we can follow the regularity theory for harmonic maps for a while. Note
	that in \cite{moser2005partial}, Lemma 3.7 and Proposition 3.2 assume $u$ to be a harmonic
	map, but the proof uses only $|\Delta u|\le C|\DR \.u|^2$ instead of the full
	harmonic map equation. We therefore can apply
	Lemmas 3.6 and 3.7, Proposition 3.2 in \cite{moser2005partial} to our $u$ and find that
	$\DR \.u\in L_{\text{\rm loc}}^\infty$. This means that the second term in (\ref{Eq-R})
	is in $L_{\text{\rm loc}}^p$ for all $p>1$, and standard $L^p$-theory gives us
	$u\in W_{\text{\rm loc}}^{2,p}$ for all $p>1$. The Sobolev embedding
	$W^{1,p}\hookrightarrow C^{0,1-2/p}$ for $p>2$ then gives us
	$\DR \.u\in C_{\text{\rm loc}}^{0,\beta}$ with $\beta>0$.
	Together with the result for $m$, we now have
	\[
	(m,R)\in C_{\text{\rm loc}}^{1,\beta}\qquad\mbox{ for some }\beta>0.
	\]
	Once we have this, we can iterate the Schauder estimates, i.e.\ differentiate
	the equations and apply Lemma \ref{L5} to partial derivatives of $m,R$ instead
	of $m$ and $R$ alone. Thus we find that $(m,R)\in C_{\text{\rm loc}}^{k,\beta}$ for our $\beta>0$ and all
	$k\in\N$, which means we have proven that $m$ and $R$ are
	smooth on the interior of the domain.\qed

\subsection{Body forces}

It is physically reasonable to consider the equations with an additional external body force
term in the first equation of balance of forces,
\begin{align}\label{balan equ}
  \Div S(\DR m,R)&=f,\\
  \nonumber\Delta R-\Omega_R\cdot\DR R-\skew\big(\DR m\dott S(\DR m,R)\big)R&=0  
\end{align}
with $f\in W^{1,2}(\os,\R^3)$. By integrating in one direction and setting
\begin{align}
  \mathcal{F}(x_1,x_2):=\Big(\int_{(x_0)_1}^{x_1}f(t,x_2)\,{\rm{dt}}\,\big|\,0_{\R^3}\Big)\in \R^{3\times 2}\,,
\end{align}
we can always assume $f=\Div \mathcal{F}$. Note that $\mathcal{F}$ depends on
the first component $(x_0)_1$ of the center of the ball $B_r(x_0)$ on which we
are momentarily working. We have $\mathcal{F}\in W^{1,2}(\os,\R^{3\times2})$,
implying $\mathcal{F}\in L^p(\os,\R^{3\times2})$ for all $p\in[1,\infty)$.
Now we may rewrite the first
equation as
\begin{align}
  \Div(S(\D m,R)-\mathcal{F})=0.
\end{align}
We will need to estimate $\DR{\mathcal{F}}$, which we calculate via
$\partial_1{\mathcal{F}}=(f,0)$ and
$\partial_2{\mathcal F}=(\int_{(x_0)_1}^{x_1}\partial_2f(t,x_2)\,{\rm{dt}},0)$.
The latter gives
\begin{align}
  \nonumber\int_{B_r}|\partial_2{\mathcal{F}}|^{3/2}\,{\rm{dx}}
  \nonumber&=\int_{B_r}\Big|\int_{(x_0)_1}^{x_1}\partial_2f(t,x_2)\,{\rm{dt}}\Big|^{3/2}\,{\rm{dx}}\\
  \nonumber&\le\int_{B_r}\Big(\int_{(x_0)_1-\sqrt{r^2-x_2^2}}^{(x_0)_1+\sqrt{r^2-x_2^2}}
    |\partial_2f(t,x_2)|\,{\rm{dt}}\Big)^{3/2}\,{\rm{dx}}\\
  &\le Cr^{1/3}\int_{B_r}\int_{(x_0)_1-\sqrt{r^2-x_2^2}}^{(x_0)_1+\sqrt{r^2-x_2^2}}
    |\partial_2f(t,x_2)|^{3/2}\,{\rm{dt}}\,{\rm{dx}}\\
  \nonumber&\le Cr^{4/3}\int_{B_r}|\partial_2f|^{3/2}\,{\rm{dx}}.
\end{align}
Since we can always assume $r\le1$, we have proven
\begin{align}\label{grossF}
  \|\DR{\mathcal{F}}\|_{L^{3/2}}\le C\|f\|_{W^{1,3/2}}.
\end{align}

The regularity theory for the more general equation including forces goes pretty much along
the lines of the $\mathcal{F}=0$ case presented in Section 6.1. We only indicate the necessary modifications.
We rewrite (\ref{Del1}) as 
\begin{align}
  \Delta f=\Omega^G\cdot G^{-1}\DR R+G^{-1}\skew\big(\D m\dott(S(\D m,R)-\mathcal{F})\big)R
  +G^{-1}\skew(\D m\dott \mathcal{F})R.
\end{align}
In (\ref{HA1}), we replace $\|S(\D m,R)\|_{L^2(B_r)}$ by
$\|S(\D m,R)-\mathcal{F}\|_{L^2(B_r)}$. Choosing the radius of $B_r$ sufficiently small,
we can also assume that $\|\mathcal{F}\|_{L^2(B_r)}\le\eps$, hence we can estimate
$\|S(\D m,R)-\mathcal{F}\|_{L^2(B_r)}$ by $C(\eps+r)$ just as we did for
$\|S(\D m,R)\|_{L^2(B_r)}$ in (\ref{HA1}). But we also have an additional
term on the right-hand side of that estimate. Using the boundedness of
$G^{-1}$ and $R$, it is estimated as follows, assuming also
$\|\mathcal{F}\|_{L^3(B_r)}\le\eps$. We have
\begin{align}
  \nonumber\sup_{\varphi\in T}\int_{B_r}\<G^{-1}\skew(\D m\dott \mathcal{F})R\,,\,\varphi\>\,{\rm dx}
  &\le C\sup_{\varphi\in T}r^{-1/3}\|\DR m\|_{L^{3/2}(B_r)}\|\mathcal{F}\|_{L^3(B_r)}r^{1/3}\|\varphi\|_{L^\infty(B_r)}\\
  &\le C\sup_{\varphi\in T}\|\D m\|_{M^{3/2,1/2}}\|\mathcal{F}\|_{L^3(B_r)}r^{2/3}\|\DR\varphi\|_{L^{3}(B_r)}\\\
  \nonumber&\le C\,\eps \,r^{1/3}\|\D m\|_{M^{3/2,1/2}},
\end{align}
which can be absorbed in the right-hand side of (\ref{HA1}). Hence the
conclusion of (\ref{HA1})
continues to hold also in the $\mathcal{F}\ne0$ case.

The second modification we have to make is that we now Hodge-decompose
$S(\DR m,R)-\mathcal{F}$, which means
\begin{align}
	\pj(\zwei \,R\,\Pj^2(R^T(\DR m|0)-(\id_2|0)))-\mathcal{F}=\DT\alpha+\chi\,.
	\end{align}
The additional term involving $\mathcal{F}$ on the right-hand side of (\ref{dmstart})
can be written as $-\DivT(\mathcal{F}-\mathcal{F}_{B_r})$. In (\ref{sechs27}), $(\mathcal{F}-\mathcal{F}_{B_r})$ can be
processed exactly like $(R-R_{B_R})$, resulting in an additional
$Cr\|\DR\psi\|_{L^3(B_r)}\|\DR \mathcal{F}\|_{L^{3/2}(B_r)}$, which can be estimated using
(\ref{grossF}) and $\psi\in U$ as follows, making the additional smallness assumption $\|f\|_{W^{1,2}(B_r)}\le\eps$ for $r$,
\begin{align}
Cr\|\DR\psi\|_{L^3(B_r)}\|\DR \mathcal{F}\|_{L^{3/2}(B_r)}\le C\,r\,\|f\|_{W^{1,3/2}(B_r)}
\le C\,r^{4/3}\|f\|_{W^{1,2}(B_r)}\le\eps\, r^{4/3}.
\end{align}
This additional term in (\ref{sechs27}) now contributes to the right-hand
side of (\ref{It2}), but here enlarges only the $r^{4/3}$ and $\rho^{4/3}$
terms that are there, anyway. By the same argument, taking $\mathcal{F}$ into account
also contributes only to more of $s^{4/3}$ terms in
\begin{align}
	\nonumber C^{-1}\|\DR m\|_{L^{3/2}(B_s)}-Cs^{4/3}
	&\leq \|\pi_{12}(\zwei \,R\,\Pj^2(R^T(\DR m|0)-(\id_2|0)))-\mathcal{F}\|_{L^{3/2}(B_s)}\\
	&\leq C(\|\DR m\|_{L^{3/2}(B_s)}+s^{4/3}),
\end{align}
which updates (\ref{u+l}). Hence the contributions of the modified versions
of both (\ref{sechs27}) and (\ref{u+l}) do not change the conclusion
of (\ref{It2}).

Now that we have adapted (\ref{HA1}) and (\ref{It2}) to nonvanishing body
forces, we can conclude H\"older continuity just as in the end of
Section \ref{suse:hoel}, under the weak assumption of $f$ being in $W^{1,2}$.
If we assume $f\in C^\infty$ instead, both $f$ and $\mathcal{F}$ are bounded,
and the higher regularity proof from Section \ref{suse:high} goes through
with hardly any modification. Note, for example, that (\ref{DR+1}) continues
to hold.

        \subsection{Remarks on a special case}
	
	Our system simplifies considerably when $\mu=\mu_c=\kappa$, which makes
	$\Pj$ the identity\footnote{This case corresponds to $\mu=\lambda=\mu_c$ in the Cosserat bulk model and Poisson number $\nu=\frac{\lambda}{2(\mu+\lambda)}=\frac{1}{4}$ (nearly satisfied for magnesium).}. Even though this assumption is not too natural from the
	point of applications, we would like to comment briefly on that case.
	
	The simplified variational functional reads now
	\begin{align}
	\nonumber E(m,R):&=\int_{\os} \mu\,|R^T(\DR m|R_3)-\id_3|^2+|\DR R|^2\,{\rm{dx}}=\int_{\os}\mu\,|R^T(\DR m|R_3)-R^TR|^2+|\DR R|^2\,{\rm{dx}}\\
	&=\int_{\os} \mu\,|(\DR m|0)-(R_1|R_2|0)|^2+|\DR R|^2\,{\rm{dx}}\,,
	\end{align}
	which has the Euler-Lagrange equations (c.f. (\ref{ELa1}))
	\begin{align}\label{ELs1}
	\Delta m-\Div(R_1|R_2|0)=0\,,
	\end{align}
	and
	\begin{align}\label{ELs2}
	\Delta R-\Omega_R\cdot\DR R+\mu\, R\skew\big(R^T(\DR m|0)\big)= \Delta R-\Omega_R\cdot\DR R+\mu\,\skew\big((\D m|0)R^T\big)\cdot R=0.
	\end{align}
	
	The point here is that the last term in the second equation now depends on
	$\DR m$ only linearly, making it an $L^2$-term instead of $L^1$ (the $L^1$-part is cancelled by the skew-operator). But harmonic map
	type equations with a right-hand side in $L^2$ have been studied by Moser
	in quite some generality, see the book \cite{moser2005partial} for an
	excellent exposition of the methods.
	
	In particular, Moser has two theorems that help us here. Here,
	$N\subset\R^n$ is a compact manifold, $\Ob\subset\R^d$ a domain,
        and $\mathrm{II}$ is the second findamental form of the target
        manifold, which corresponds to our term quadratic in $\DR R$, i.e. $\sum_i{\mathrm{II}}(u)(\partial_iu,\partial_iu)=\Omega_u\cdot \D u$ in our case.
	
	\begin{theorem}{\rm{\cite[Theorem 4.1]{moser2005partial}}} \label{4.1}
		Suppose $u\in W^{1,2}(\Ob,N)$ is a stationary solution of
		\[
		\Delta u-\sum_i{\mathrm{II}}(u)(\partial_iu,\partial_iu)=f\,,
		\]
		in $\Ob$, for a function $f\in L^p(\Ob,\R^n)$, where $p>\frac{d}2$
		and $p\ge2$. Then there exists a relatively closed set
		$\Sigma\subset\Ob$ of vanishing $(d-2)$-dimensional Hausdorff
		measure, such that $u\in C^{0,{\alpha}}_{\text{\rm loc}}(\Ob\setminus\Sigma,N)$ for
		a number ${\alpha}>0$ that depends only on $m$, $N$, and $p$.
	\end{theorem}
	
	\begin{theorem}{\rm{\cite[Theorem 4.2]{moser2005partial}}}\label{4.2}
		Under the assumptions of the previous theorem, if $n\le4$ and $p=2$,
		we also have $u\in W^{2,2}_{\text{\rm loc}}\cap W^{1,4}_{\text{\rm loc}}(\Ob\setminus\Sigma,N)$.
	\end{theorem}
	
	While those theorems are highly nontrivial, it is standard to deduce
	regularity of the solutions to our model in the special case considered here.
	
	\begin{theorem}[interior regularity for $\mu=\mu_c=\kappa$]
		Any solution $(m,R)\in W^{1,2}(\os,\R^3\times \SO(3))$
		of the simplified problem (\ref{ELs1})-(\ref{ELs2}) is smooth on the interior of the domain $\os$.
	\end{theorem}
	
	\begin{proof} We first consider eq.\ (\ref{ELs2}). Since $-R\skew(R^T(\D m|0))\in L^2$, we can apply Theorem~\ref{4.1} and Theorem~\ref{4.2} to find
	$R\in C^{0,{\alpha}}_{\text{\rm loc}} \cap W^{2,2}_{\text{\rm loc}}(\os,\SO(3))$. Note that $\Sigma=\emptyset$
	here, since its $0$-dimensional Hausdorff measure vanishes.
	Similarly, by $L^2$-theory
	for (\ref{ELs1}), we have $m\in W^{2,2}_{\text{\rm loc}}(\os,\R^3)$. By the embedding
	$W^{2,2}\hookrightarrow W^{1,q}$ for all $q\in[2,\infty)$, we find that
	$\Delta m$ and $\Delta R$ are in $L^q_{\text{\rm loc}}$ for every $q<\infty$, hence
	$(m,R)\in W^{2,q}_{\text{\rm loc}}(\os,\R^3\times \SO(3))$ for all $q<\infty$.
	This, in turn, embeds into $C^{1,{\alpha}}_{\text{\rm loc}}$
	for all ${\alpha}\in(0,1)$, and hence the right-hand sides are Hölder continuous.
	From here, we can use Schauder estimates to show that $(m,R)$ is $C^\infty_{\text{\rm loc}}$
	on $\os$.
	\end{proof}
\section{Conclusion and open problems} 
	We have deduced interior Hölder regularity for a Dirichlet type geometrically nonlinear Cosserat flat membrane shell. The model is objective and isotropic but highly nonconvex. Therefore, our regularity result is astonishing and shows again the great versatility of the Cosserat approach compared to other more classical models.	
	At present, we are limited to treating the uni-constant curvature case $|\DR R|^2$ since only then can sophisticated methods for harmonic functions with values in $\SO(3)$ be employed. This calls for more effort of researchers to generalize the foregoing. Progress in this direction would also allow to consider the full Cosserat membrane-bending flat shell \cite{neff2004geometrically,neff2004existence,neff2004geometricallyhabil,birsan2013existence}. Another case warrants further attention: taking the Cosserat couple modulus $\mu_c=0$ in the model (in-plane drill allowed, but no energy connected to it) may still allow for regular minimizers. However, even the existence of minimizers remains unclear at present since it 
	hinges on some sort of a priori regularity for the rotation field $R$ (the non-quadratic curvature term $|\DR R|^{2+\varepsilon}$, $\varepsilon>0$, together with zero Cosserat couple modulus  $\mu_c=0$ allows  for minimizers \cite{neff2007geometrically,neff2002korn}). Finally, it is interesting to understand regularity properties of Cosserat shell models with curved initial geometry \cite{NeffPartI,GhibaNeffPartII,Ghiba2022}.

        We expect some boundary regularity to hold, too. On the geometric analysis side, an adaptation of Rivi\`ere's boundary methods to problems with continuous Dirichlet boundary data has been performed in \cite{mueller2009boundary}, which one could try to use. But with a view on applications, partially free boundary problems would probably be more interesting. 
\vspace{0.7cm}

\begin{footnotesize}
	\noindent{\bf Acknowledgements:}   
	The authors are indebted to Oliver Sander and Lisa Julia Nebel (Chair of Numerical Mathematics, Technical University of Dresden) for preparing the calculations leading to Figures \ref{numeric figure A},\ref{numeric figure B} and \ref{numeric figure C}. The first author thanks Armin Schikorra (University of Pittsburgh) for an interesting discussion. The second author is grateful to Maryam Mohammadi Saem and Peter Lewintan (Faculty of Mathematics, University of Duisburg-Essen) for help in preparing the manuscript and the pictures.\\\strut\\
	Andreas Gastel acknowledges support in the framework of the DFG-Priority Programm 2256 "Variational Methods for Predicting Complex Phenomena in Engineering Structures and Materials" with the project title "Very singular solutions of a nonlinear Cosserat elasticity model for solids" with Project-no. \;441380936\; and Patrizio Neff acknowledges support in the framework of the DFG-Priority Programm 2256 "Variational Methods for Predicting Complex Phenomena in Engineering Structures and Materials" with the project title "A variational scale-dependent transition scheme: From Cauchy elasticity to the relaxed micromorphic continuum" with Project-no. 440935806 and also the DFG research grant \;Neff 902/8-1 Project-no. 415894848 with the title "Modelling and mathematical analysis of geometrically nonlinear Cosserat shells with higher order and residual effects".
\end{footnotesize}

\footnotesize
\section*{References}
\printbibliography[heading=none]
\addcontentsline{toc}{section}{References}

\begin{appendix}
\section{Appendix}

\subsection{Three-dimensional Euler-Lagrange equations in dislocation tensor format}
Here, for the convenience of the reader we derive the three-dimensional Euler-Lagrange equations based on the curvature expressed in the dislocation tensor $\bm{\alpha}=R^T\Curl R$. We can write the bulk elastic energy as
\begin{align}\label{Euler1}
E^{\text{3D}}(\varphi,R)=\int_\Ob W_{\text{mp}}(\overline{U})+W_{\text{disloc}}({\bm{\alpha}})\;{\rm{dx}}\,,\quad \overline{U}=R^T\DR \varphi\,,\quad {\bm{\alpha}}=R^T\Curl R\,.
\end{align}
Taking variations of (\ref{Euler1}) w.r.t. the deformation $\varphi$ leads to
\begin{align}
\delta E^{\text{3D}}{(\varphi,R)}\cdot \delta \varphi=\int_{\Ob}\iprod{\DR W_{\text{mp}}(\overline{U}),R^T\DR \delta\varphi}_{\R^{3\times 3}}&\;{\rm{dx}}=0\,,\qquad\forall\;\delta\varphi\in C_0^\infty(\Ob,\R^3)\\
\nonumber\Leftrightarrow& \int_{\Ob}\iprod{R\,\DR W_{\text{mp}}(\overline{U}),\DR \delta\varphi}_{\R^{3\times 3}}\;{\rm{dx}}=\int_{\Ob}\iprod{\Div[R\cdot \DR W_{\text{mp}}(\overline{U})],\delta\varphi}_{\R^3}\,{\rm{dx}}=0\,.
\end{align}
Taking variation w.r.t. $R\in \SO(3)$ results in (abbreviate $F:=\DR \varphi$)
\begin{align}\label{Euler2}
\nonumber \delta E^{\text{3D}}{(\varphi,R)}\cdot \delta R&=\int_{\Ob}\iprod{\DR W_{\text{mp}}(\overline{U}),\delta R^TF}+\iprod{\DR W_{\text{disloc}}({\bm{\alpha}}),\delta R^T\Curl R+R^T\Curl \delta R}\;{\rm{dx}}\\
\nonumber &=\int_{\Ob}\iprod{\DR W_{\text{mp}}(\overline{U}),\delta R^TR\cdot R^TF}+\iprod{\DR W_{\text{disloc}}({\bm{\alpha}}),\delta R^TR\cdot R^T\Curl R+R^T\Curl \delta R}\;{\rm{dx}}\\
&=\int_{\Ob}\iprod{\DR W_{\text{mp}}(\overline{U})\cdot\overline{U}^T,\delta R^TR}+\iprod{\DR W_{\text{disloc}}({\bm{\alpha}}),\delta R^TR\cdot {\bm{\alpha}}+R^T\Curl \delta R}\;{\rm{dx}}=0\,.
\end{align}
Since $R^TR=\id_3$, it follows that $\delta R^TR+R^T\delta R=0$ and $\delta R^TR=A\in \so(3)$ is arbitrary. Therefore, (\ref{Euler2}) can be written as
\begin{align}
0=\int_{\Ob}\iprod{\DR W_{\text{mp}}(\overline{U})\cdot\overline{U}^T,A}+\iprod{\DR W_{\text{disloc}}({\bm{\alpha}})\cdot{\bm{\alpha}}^T,A}+\iprod{\DR W_{\text{disloc}}({\bm{\alpha}}),R^T\Curl(RA^T)}\;\rm{dx}
\end{align}
for all $A\in C_0^\infty(\mathcal{U},\so(3))$. Using that $\Curl$ is a self-adjoint operator, this is equal to
\begin{align}
0&=\int_{\Ob}\iprod{\DR W_{\text{mp}}(\overline{U})\cdot \overline{U}^T+\DR W_{\text{disloc}}({\bm{\alpha}})\;{\bm{\alpha}}^T,A}+\iprod{\Curl(R\;\DR W_{\text{disloc}}({\bm{\alpha}})),RA^T}\;\rm{dx}\\
\nonumber&=\int_{\Ob}\iprod{\DR W_{\text{mp}}(\overline{U})\cdot\overline{U}^T+\DR W_{\text{disloc}}({\bm{\alpha}})\;{\bm{\alpha}}^T-R^T\Curl(R\;\DR W_{\text{disloc}}({\bm{\alpha}})),A}\;{\rm{dx}}\qquad\forall A\in C_0^\infty(\mathcal{U},\so(3))\,.
\end{align}
Thus, the strong form of the Euler-Lagrange equations reads
\begin{align}\label{system}
\nonumber\Div[R\;\DR W_{\text{mp}}(\overline{U})]&=0\,,\hspace{3cm}\text{"balance of forces"}\,,\\
\nonumber\skew[R^T\,\Curl(R\;\DR W_{\text{disloc}}({\bm{\alpha}}))]&=\skew(\DR W_{\text{mp}}(\overline{U})\cdot\overline{U}^T+\DR W_{\text{disloc}}({\bm{\alpha}})\cdot {\bm{\alpha}}^T)\,,\\
&\hspace{4cm}\text{"balance of angular momentum"}\,.
\end{align}
Defining the first Piola-Kirchhoff stress tensor 
\begin{align}\label{S_1}
S_1(\D \varphi,R):=\D_F [W_{\text{mp}(\overline{U})}]=R\,\D W_{\text{mp}}(\overline{U})=R\cdot T_{\text{Biot}}(\overline{U})\,,
\end{align}
where the non-symmetric Biot type stress tensor is given by
\begin{align}\label{S_11}
T_{\text{Biot}}:=\D W_{\text{mp}}(\overline{U})\,,
\end{align}
allows to rewrite the system (\ref{system}) as
\begin{align}\label{W(dilsloc)}
\nonumber\Div S_1(\D \varphi,R)&=0\\
\skew[R^T\,\Curl(R\;\DR W_{\text{disloc}}({\bm{\alpha}}))]&=\skew\Big(T_{\text{Biot}}(\overline{U})\cdot\overline{U}^T+\DR W_{\text{disloc}}({\bm{\alpha}})\cdot {\bm{\alpha}}^T\Big)\,.
\end{align}
Observe that (\ref{W(dilsloc)})$_1$ is a uniformly elliptic linear system for $\varphi$ at given $R$.
It is clear that global minimizers $\varphi\in W^{1,2}(\Ob,\R^3)$ and $R\in W^{1,2}(\Ob,\SO(3))$ are weak solutions of the Euler-Lagrange equations.\\
If $\DR W_{\text{disloc}}({\bm{\alpha}})\equiv0$ (no moment stresses) then balance of angular momentum turns into the symmetry constraint
\begin{align}
\DR W_{\text{mp}}(\overline{U})\cdot \overline{U}^T\in \Sym(3)\,.
\end{align}
A complete discussion of the solutions to this constraint can be found in \cite{neff2019explicit,neff2008symmetric}.
\subsection{Two-dimensional Euler-Lagrange equations: alternative derivation}\label{systemApp}
\begin{align}\label{E2D}
\nonumber E^{\text{2D}}{(m,R)}&=\int_\os\hspace{-0.05cm}\mu|\dev\sym(R^T(\DR m|R_3)-\id_3)|^2\hspace{-0.05cm}+\hspace{-0.05cm}\mu_c|\skew(R^T(\DR m|R_3)-\id_3)|^2\hspace{-0.05cm}+\hspace{-0.05cm}\frac{\kappa}{3}\tr(R^T(\DR m|R_3)-\id_3)^2 +|\DR R|^2\,{\rm{dx}}\\
&=\int_\os |\Pj (R^T(\DR m|R_3)-\id_3)|^2+|\DR R|^2\;{\rm{dx}}= \int_{\os} |\Pj(R^T(\DR m|R_3)-\id_3)|^2+|\partial_xR|^2+|\partial_yR|^2\; {\rm{dx}}\,.
\end{align}
Taking free variations w.r.t the midsurface deformation $m$ leads to
\begin{align}
\delta E^{\text{2D}}{(m,R)}\cdot \delta \vartheta&=\int_{\os}2\iprod{\Pj(R^T(\DR m|R_3)-\id_3),\Pj(R^T(\DR \vartheta|0))}_{\R^{3\times 3}}\;{\rm{dx}}\qquad \forall \vartheta\in C_0^\infty(\os,\R^3)\\
\nonumber&=\int_\os 2\iprod{\Pj^T\Pj(R^T(\DR m|R_3)-\id_3),(R^T(\DR \vartheta|0))}_{\R^{3\times 3}}\;{\rm{dx}}=\int_\os \iprod{2R\,\Pj^T\Pj(R^T(\DR m|R_3)-\id_3),(\DR \vartheta|0)}_{\R^{3\times 3}}\;{\rm{dx}}\\
\nonumber&=\int_\os \iprod{\pj(2R\,\Pj^T\Pj(R^T(\DR m|R_3)-\id_3)),\DR \vartheta}_{\R^{3\times 2}}\;{\rm{dx}}=\int_\os \iprod{\Div\pj(2R\,\Pj^T\Pj(R^T(\DR m|R_3)-\id_3)),\vartheta}_{\R^{3}}\;{\rm{dx}}=0\,.
\end{align}
Thus the strong form of balance of forces can be expressed as 
\begin{align}
\Div S(\DR m,R)&=0\,,
\end{align}
where 
\begin{align}
S(\DR m,R)=\pj(\zwei\, R\,\Pj^T\Pj(R^T(\DR m|R_3))-\id_3)=\pj(\zwei\, R\,\Pj^T\Pj(R^T(\DR m|0)-(\id_2|0)))\,,
\end{align}
is the {\it{first Piola-Kirchhoff type force-stress tensor}} and 
\begin{align}\label{nonsym Biot}
\nonumber T(\DR m,R)&=\zwei \,\Pj^T\Pj(R^T(\DR m|0)-(\id_2|0))=\zwei \,\Pj^T\Pj(R^T(\DR m|R_3)-\id_3)\\
&=\zwei \,\mu\dev\sym(\overline{U}-\id_3)+\zwei \,\mu_c\skew (\overline{U}-\id_3)+\frac{\zwei \,\kappa}{3}\tr(\overline{U}-\id_3)\cdot\id_3\,,\quad\quad\overline{U}=R^T(\DR m|R_3)\,,
\end{align}
is the {\it{non-symmetric Biot-type stress tensor}} (symmetric if $\mu_c=0$).
 We note the relation
\begin{align}
S(\D m,R)=\pj (R\cdot T(\D m,R))\,.
\end{align}
resembling relation (\ref{S_1}).

For balance of angular momentum we proceed similarly, but need some preparation. It is clear that
\begin{align}
(R+\delta R)^T(\DR m|(R+\delta R)e_3)-\id_3=R^T(\DR m|R_3)-\id_3+\underbrace{R^T(0|0|\delta R_3)+(\delta R^T(\DR m|R_3))}_{\text{linear increment}}+\delta R^T(0|0|\delta R_3)\,.
\end{align}
Therefore, taking variations of the energy w.r.t $R$ leads to
\begin{align}
\delta E^{\text{2D}}{(m,R)}\cdot \delta R=\int_\os 2\iprod{\Pj(R^T(\DR m|R_3)-\id_3),\Pj(R^T(0|0|\delta R_3)+\delta R^T(\DR m|R_3))}+2\iprod{\partial_xR,\partial_x\delta R}+2\iprod{\partial_y R,\partial_y \delta R}\,{\rm{dx}}=0\,.
\end{align}
Since $R^TR=\id_3$, we have $\delta R^TR+R^T\delta R=0$, hence $\delta R=RA$ for $A\in \so(3)$ arbitrary. Therefore the latter turns into
\begin{align}\label{longapp}
\nonumber \int_\os &2\iprod{\Pj^T\Pj(R^T(\DR m|R_3)-\id_3),R^T(0|0|(RA)e_3)+(RA)^T(\DR m|R_3)}-2\iprod{\partial^2_xR,RA}-2\iprod{\partial^2_yR,RA}\,{\rm{dx}}\\
\nonumber&=\int_\os 2\iprod{\Pj^T\Pj(R^T(\DR m|R_3)-\id_3),A(0|0|e_3)-AR^T(\DR m|R_3)}-\iprod{2\Delta R,RA}\,{\rm{dx}}\\
&=\int_\os 2\iprod{\Pj^T\Pj(R^T(\DR m|R_3)-\id_3),-A\big(R^T(\DR m|R_3)-(0|0|e_3)\big)}-\iprod{2R^T\Delta R,A}\,{\rm{dx}}\,,\qquad \forall A\in C_0^\infty(\os,\so(3))\\
\nonumber&=-\int_\os 2\iprod{\Pj^T\Pj(R^T(\DR m|R_3)-\id_3),AR^T(\DR m|0)}+\iprod{2R^T\Delta R,A}\,{\rm{dx}}\\
\nonumber&=-\int_\os \iprod{2\,\Pj^T\Pj(R^T(\DR m|R_3)-\id_3)(\DR m|0)^TR,A}+\iprod{2R^T\Delta R,A}\,{\rm{dx}}=0\,.
\end{align}
This implies the stationary condition in strong from
\begin{align}\label{skewdelta}
\nonumber\skew(\zwei R^T\Delta R)&=-\skew (\zwei \Pj^T\Pj(R^T(\DR m|R_3)-\id_3)\cdot(\DR m|0)^TR)=-\skew (\zwei \Pj^T\Pj(R^T(\DR m|0)-(\id_2|0))\cdot (\DR m|0)^TR)\\
&=-\skew (T(\DR m|R)\cdot (\DR m|0)^TR)\,,
\end{align} 
where $T$ is defined in (\ref{nonsym Biot})$_1$.

For $\Pj^T\Pj=\mu\cdot \id$ the last equation simplifies to
\begin{align}
\nonumber\skew (\zwei R^T\Delta R)&=-\mu\skew\Big(\zwei R^T(\DR m|0)(\DR m|0)^TR-(\id_2|0)(\DR m|0)^T R\Big)=\zwei \mu\skew\big((\id_2|0)(\DR m|0)^TR\big)\\
&=\zwei \mu\skew((\DR m|0)^TR)=-\zwei \mu \skew(R^T(\DR m|0))\,.
\end{align}
We can also rewrite (\ref{longapp})$_5$ as
\begin{align}
\nonumber 0&=\int_{\os} 2\iprod{\Pj^T\Pj(R^T(\DR m|R_3)-\id_3)(\DR m|0)^T,AR^T}+2\iprod{\Delta R,RA}\;{\rm{dx}}\\
&=\int_{\os} 2\dyniprod{\Big(\Pj^T\Pj(R^T(\DR m|R_3)-\id_3)(\DR m|0)^T\Big)^T,RA^T}+2\iprod{\Delta R,RA}\;{\rm{dx}}\\
\nonumber&=\int_{\os} 2\dyniprod{(\DR m|0)\Big(\Pj^T\Pj(R^T(\DR m|R_3)-\id_3)^T,RA}+2\iprod{\Delta R,RA}\;{\rm{dx}}\,,\qquad \forall A\in C_0^\infty(\os,\so(3))
\end{align}
being equivalent to
\begin{align}\label{perp}
\Delta R-(\DR m|0)\Big(\Pj^T\Pj(R^T(\DR m|R_3)-\id_3\Big)^T \perp {\rm{T}}_R\SO(3)\,.
\end{align}
Since
\begin{align}
\nonumber \Big(\Pj^T\Pj (R^T(\DR m|R_3)-\id_3)\Big)^T&=\Big(\Pj^T\Pj(R^T(\DR m|0)-(\id_2|0))\Big)^T=\Pj^T\Pj ((R^T(\DR m|0)(\id_2|0))^T)\,,
\end{align}
we may express (\ref{perp}) also as
\begin{align}\label{angular}
\Delta R-(\DR m|0)\Pj^T\Pj ((\DR m|0)^TR-(\id_2|0))\perp {\rm{T}}_R\SO(3)\,.
\end{align}
This is the form for balance of angular momentum given in equation (\ref{ELa2}).
\subsection{Lifting to the $\Delta$-operator}
This last equation (\ref{angular}) is not, however, the final form of the balance of angular momentum equation that we will consider. Indeed, since $R^TR=\id_3$, we can differentiate once to obtain 
\begin{align}
(\partial_xR)^TR+R^T\partial_x R=0\,,\quad\quad(\partial_yR)^TR+R^T\partial_y R=0\,.
\end{align}
Taking second partial derivatives, we get 
\begin{align}
(\partial^2_x R)^TR+(\partial_xR)^T\partial_xR+(\partial_x R)^T\partial_xR+R^T\partial_x^2R&=0\,,\\
\nonumber (\partial^2_y R)^TR+(\partial_yR)^T\partial_yR+(\partial_y R)^T\partial_yR+R^T\partial_y^2R&=0\,.
\end{align}
Summing up, shows
\begin{align}\label{symdel}
\nonumber (\Delta R)^TR+R^T\Delta R+2\big((\partial_x R)^T\partial_x R+(\partial_y R)^T\partial_y R\big)&=0\,,\\
\Longleftrightarrow \quad 2\sym(R^T\Delta R)+2[(\partial_x R)^T\partial_xR+(\partial_y R)^T\partial_R]&=0\,,\\
\nonumber \sym (R^T\Delta R)=-[(\partial_xR)^T\partial_xR+(\partial_yR)^T\partial_yR]&\,.
\end{align}
From (\ref{skewdelta}) we have
\begin{align}\label{skewdel}
\skew(\zwei R^T\Delta R)=-\skew (T(\DR m,R)(\DR m|0)^TR)\,.
\end{align}
Adding (\ref{symdel})$_2$ and (\ref{skewdel}) yields, due to the orthogonality of $\sym$ and $\skew$
\begin{align}
\zwei R^T\Delta R=\sym (\zwei R^T\Delta R)+\skew (\zwei R^T\Delta R)=-\zwei [(\partial_xR)^T\partial_xR+(\partial_y R)^T\partial_yR]-\skew (T(\DR m,R)(\DR m|0)^TR)\,.
\end{align}
Hence using the isotropy of skew we obtain
\begin{align}\label{delta}
\nonumber\zwei \Delta R&=-\zwei R\,[(\partial_xR)^T\partial_xR+(\partial_y R)^T\partial_yR]-R\skew \big(T(\DR m,R)(\DR m|0)^TR\big)\\
\nonumber&=-\zwei R\,[(\partial_xR)^T\partial_xR+(\partial_y R)^T\partial_yR]-R\skew \big(R^TR\,T(\D m,R)(\D m|0)^TR\big)\\
\nonumber&=-\zwei R\,[(\partial_xR)^T\partial_xR+(\partial_y R)^T\partial_yR]-R\,R^T\skew\big(R\,T(\D m,R)(\D m|0)^T\big)\cdot R\\
&=-\zwei R\,[(\partial_xR)^T\partial_xR+(\partial_y R)^T\partial_yR]+\skew \big((\D m|0)(R\,T(\D m,R))^T\big)\cdot R\\
\nonumber&=-\zwei R\,[(\partial_xR)^T\partial_xR+(\partial_y R)^T\partial_yR]+\skew \big((\D m|0)[(\pj (R\,T(\D m ,R)))]^T\big)\cdot R\\
\nonumber&=-\zwei R\,[(\partial_xR)^T\partial_xR+(\partial_y R)^T\partial_yR]+\skew \big(\D m\cdot S(\D m,R)^T\big)\cdot R\,,
\end{align}
giving
\begin{align}
\nonumber\Delta R&=-R\,[(\partial_xR)^T\partial_xR+(\partial_y R)^T\partial_yR]+\frac{1}{2}\skew \big(\D m\cdot S(\D m,R)^T\big)\cdot R\\
&=-R\,[(\partial_xR)^T\partial_xR+(\partial_y R)^T\partial_yR]+\skew (\D m\dott S(\D m,R))\cdot R\,.
\end{align}
where we used the definition of $\dott$ given in equation (\ref{dott}).

We set 
\begin{align}
-R\,[(\partial_xR)^T\partial_xR+(\partial_y R)^T\partial_yR]=-R\,\partial_x R^T\cdot\partial_x R-R\,\partial_y R^T\,\partial_y R=:\Omega_R\cdot \DR R\,.
\end{align}
\begin{align}
(\Omega_R)_1&:=-R\,\partial_x R^T\in \so(3)\,,\qquad (\Omega_R)_2:=-R\,\partial_y R^T\in \so(3)\,.
\end{align}
With this definition, (\ref{delta}) can be written as
\begin{align}\label{Del}
\Delta R=\underbrace{\Omega_R\cdot \DR R}_{\in L^1(\os)}-\underbrace{R\skew\big(T(\DR m,R)(\DR m|0)^TR\big)}_{\in L^1(\os)}\,.
\end{align}
Considering the special case $\Pj^T\Pj=\mu\cdot\id$, equation (\ref{Del}) turns into
\begin{align}
\Delta R&=\Omega_R\cdot \DR R+\mu R\skew((\DR m|0)^TR)=\Omega_R\cdot\DR R-\underbrace{\mu R\skew(R^T(\DR m|0))}_{\in L^2(\os)}\,.
\end{align}

We observe finally, that
\begin{align}
R^T(\Omega_R)_iR=-(\partial_i R)^T R=R^T\partial_iR\,,\qquad\qquad i=1,2\,,
\end{align}
which implies for $\Gamma_i=\axl (R^T\partial_i R)$
\begin{align}
R^T(\Omega_R)_i R=\Anti (\Gamma_i)\,,\qquad\qquad \axl(R^T(\Omega_R)_i R)=\Gamma_i\,,
\end{align}
where $\Gamma$ is the wryness tensor from eq. (\ref{wryness}).
\subsection{A glimpse on a Reissner-Mindlin type flat membrane shell model}
It is interesting to compare our Cosserat flat membrane shell model (allowing for existence of minimizers and their full regularity) with one that would appear closer to classical approaches. For this sake we consider a Reissner-Mindlin flat membrane shell model next.

In case of the one-director geometrically nonlinear, physically linear Reissner-Mindlin flat membrane shell model without independent drilling rotations, the problem can be described as a two-field minimization for the midsurface $m\col \os \subset \R^2\to \R^3$ and the unit-director field $d\col\os\subset \R^2\to \mathbb{S}^2$ of the elastic energy\footnote{The missing Reissner-Mindlin bending contribution scaling with $h^3$ would be of the form \cite[(7.25)]{neff2004geometrically}
\begin{align}\label{Delta}
\frac{h^3}{12}\Big\{\mu|\sym \Big((\DR m|d)^T(\DR d|0)\Big)|^2+\frac{\mu\lambda}{2\mu+\lambda}\tr\Big(\sym\Big((\DR m|d)^T(\DR d|0)\Big)\Big)^2 \Big\}\,.
\end{align}
Here, no choice of constitutive parameters reduces the bending energy to the uni-constant case. }
\begin{align}
E_{\text{Reissner}}^{\text{2D}}(m,d)=\int_{\os}h\underbrace{\Big\{\frac{\mu}{4}\underbrace{|\DR m^T\DR m-\id_2|^2}_{\text{in-plane stretch}}+\frac{1}{8}\frac{2\mu\lambda}{2\mu+\lambda}\underbrace{\tr\big(\DR m^T\DR m-\id_2\big)^2}_{\text{elongational stretch}}}_{\text{non-elliptic}}+\mu\big(\underbrace{\iprod{d\.,\partial_x m}^2+\iprod{d\.,\partial_ym}^2}_{\text{transverse shear}}\big)+\underbrace{\mu\frac{L_c^2}{2}|\nabla d|^2}_{\text{curvature}}\Big\}\;{\rm{dx}}\,.
\end{align}
Here, the membrane energy part is not rank-one elliptic due to the presence of the membrane strain $\DR m^T\DR m-\id_2$. The uni-constant curvature energy could be generalized to the Oseen-Frank form, cf. subsection \ref{Oseen-Frank}. We note that
$\iprod{d\,,\partial_xm}^2+\iprod{d\,,\partial_ym}^2=|\DR m^T d|^2_{\R^2}$, and look for simplicity at the energy
\begin{align}
\int_{\os}\frac{1}{2}|\DR m^T\DR m-\id_2|^2_{\R^{2\times 2}}+\frac{1}{2}|\DR m^Td|^2_{\R^2}+\frac{1}{2}|\nabla d|^2=\int_{\os}\frac{1}{2}|(\DR m|n_m)^T(\DR m|n_m)-\id_3|^2+\frac{1}{2}|\DR m^T d|^2+\frac{1}{2}|\nabla d|^2\;{\rm{dx}}\,.
 \end{align}
The Euler-Lagrange equations are then given by
\begin{align}
\nonumber\delta E^{\text{2D}}_{\text{Reissner}}(m,d)\cdot \delta m&=\int_{\os}2\iprod{(\DR m|n_m)^T(\DR m|n_m)-\id_3,(\DR m|n_m)^T(\DR \delta m|0)}_{\R^{3\times 3}}+\iprod{\DR m^T d,(\DR \delta m)^Td}_{\R^2}\;{\rm{dx}}\\
\nonumber&=\int_{\os}\iprod{(\DR m|n_m)\cdot 2((\DR m|n_m)^T(\DR m|n_m)-\id_3),(\DR \delta m|0)}_{\R^{3\times 3}}\;{\rm{dx}}+\iprod{\DR m^T d\otimes \D\delta m^Td,\id_2}_{\R^{2\times 2}}{\rm{dx}}\\
\nonumber&=\int_{\os}\iprod{(\DR m|n_m)\cdot 2((\DR m|n_m)^T(\DR m|n_m)-\id_3),(\DR \delta m|0)}_{\R^{3\times 3}}\;{\rm{dx}}+\iprod{(\DR m^T d\otimes d) \D\delta m,\id_2}_{\R^{2\times 2}}{\rm{dx}}\\
&=\int_{\os}\iprod{\pj((\DR m|n_m)\cdot 2((\DR m|n_m)^T(\DR m|n_m)-\id_3))+d\otimes \DR m^Td,\DR \delta m}_{\R^{3\times 2}}\;{\rm{dx}}\\
\nonumber&=-\int_{\os}\iprod{\Div \Big[\pj((\DR m|n_m)\cdot 2((\DR m|n_m)^T(\DR m|n_m)-\id_3))+d\otimes \DR m^Td\Big],\delta m}_{\R^3}\;{\rm{dx}}=0\quad\forall \delta m\in C_0^\infty(\os,\R^3)\,.
\end{align}
Here, we can define the first Piola-Kirchhoff type stress tensor
\begin{align}
S(\DR m,d)=\pj((\DR m|n_m)\cdot 2((\DR m|n_m)^T(\DR m|n_m)-\id_3)+d\otimes \DR m^Td\,.
\end{align}
For variations w.r.t $d\in \mathbb{S}^2$ we note that 
\begin{align}
|d+\delta d|^2=1\quad\Longleftrightarrow\quad |d|^2+2\iprod{d,\delta d}+|\delta d|^2=1\,.
\end{align}
Hence, the variation $\delta d$ is orthogonal to $d$, i.e. $\iprod{d,\delta d}=0$. Without loss of generality, we express $\delta d$ as $\delta d=d\times \delta v$ for some $\delta v\in C_0^\infty (\os,\R^3)$. Therefore, taking variations w.r.t. $d$ gives
\begin{align}
\nonumber \delta E_{\text{Reissner}}^{\text{2D}}(m,d)\cdot \delta d=\int_\os\iprod{\nabla d,\nabla \delta d}+\iprod{\DR m^Td,\DR m^T\delta d}\,{\rm{dx}}=0 &\Longleftrightarrow -\int_\os\iprod{\Delta d,d\times \delta v}+\iprod{\DR m\,\DR m^T d,d\times \delta v}\,{\rm{dx}}=0\\
\nonumber\Longleftrightarrow \int_\os-\iprod{\Delta d,\Anti (d)\delta v}+\iprod{\DR m\,\DR m^T d,\Anti(d) \delta v}\,{\rm{dx}}=0 &\Longleftrightarrow \int_\os\iprod{\Anti(d) \Delta d,\delta v}_{\R^3}-\iprod{\Anti(d)\DR m\,\DR m^T d,\delta v}\,{\rm{dx}}=0\,,\\
&\hspace{4.5cm}\forall \delta v\in C_0^\infty (\os,\R^3)\,.
\end{align}
The latter leads to the strong form
\begin{align}\label{anti}
\Anti (d)\Big(\Delta d-\DR m\,\DR m^T d\Big)=0\Longleftrightarrow d\times(\Delta d- \DR m\,\DR m^T d)=0\,.
\end{align}
However, since $d\in \mathbb{S}^2$, we know $|d|^2=1$. Therefore, in addition, taking partial derivatives, we obtain
\begin{align}
\iprod{\partial_x d,d}=0\,,\qquad\qquad \iprod{\partial_y d,d}=0\,.
\end{align}
Taking second partial derivatives yields
\begin{align}
\iprod{\partial^2_xd,d}+\iprod{\partial_xd,\partial_y d}=0\,,\qquad\iprod{\partial^2_yd,d}+\iprod{\partial_yd,\partial_yd}=0\,.
\end{align}
Summing up shows
\begin{align}\label{Deltad}
\iprod{\Delta d,d}+|\partial_x d|^2+|\partial_yd|^2=\iprod{\Delta d,d}+|\D d|^2=0\,.
\end{align}
Adding (\ref{anti}) and (\ref{Deltad}) shows
\begin{align}
d*\Delta d:=\hspace{-0.5cm}\underbrace{d\times \Delta d+\iprod{d,\Delta d}}_{\text{geometric product, Clifford product}}\hspace{-0.6cm}=\quad\underbrace{d\times (\DR m\,\DR m^Td)}_{\in\R^3}-\underbrace{|\DR d|^2}_{\in \R}\,.
\end{align}
Formally, this implies 
\begin{align}
\Delta d=\frac{d}{|d|^2}*\Big[d\times (\DR m\,\DR m^Td)-|\DR d|^2\Big]\,.
\end{align}
In terms of equations, we have altogether from (\ref{anti})$_1$ and (\ref{Deltad}), respectively
\begin{align}
\Anti(d)\,\Delta d=\DR m\,\DR m^T d\,,\qquad\qquad \iprod{d,\Delta d}=-|\DR d|^2\,,
\end{align}
equivalently
\begin{align}\label{matrix}
\underbrace{\matr{&\Anti(d)&\\d_1&\hspace{-0.2cm}d_2&\hspace{-0.2cm}d_3}}_{=:\widehat{A}\in \R^{4\times 3}}\matr{\Delta d\\|}=\matr{\DR m\,\DR m^T d\\-|\DR d|^2}_{\R^4}\,.
\end{align}
We multiply (\ref{matrix}) with $\widehat{A}^T$ to get
\begin{align}
\widehat{A}^T\widehat{A}\,\Delta d=\widehat{A}^T\matr{\DR m\,\DR m^T d\\-|\DR d|^2}\in \R^3\,,\qquad |d|=1\,.
\end{align}
Since in fact (sic)
$$\widehat{A}^T\widehat{A}=|d|^2\cdot \id_3=\id_3\,,$$ we obtain the system of Euler-Lagrange equations
\begin{align}\label{balan-forc}
\Div S(\DR m,d)&=0\,,\qquad\qquad\text{"balance of forces"} 
\end{align}
with a first Piola-Kirchhoff type stress tensor
\begin{align}
 S(\DR m ,d)=\pj((\DR m|n_m)\cdot 2((\DR m|n_m)^T(\DR m|n_m)-\id_3)+d\otimes \DR m^Td\,,
\end{align}
and
\begin{align}\label{Deltaanti}
 \nonumber\Delta d&=\Big(-\Anti(d)|d\Big)_{\R^{3\times 4}}\matr{\DR m\,\DR m^T d\\-|\DR d|^2}\\
 \nonumber&=-\Anti(d)\,(\DR m\,\DR m^Td)-|\DR d|^2\cdot d\,,\hspace{1cm}\text{"balance of director equilibirium"}\\
 &=-\underbrace{d\times (\DR m\,\DR m^T d)}_{\in L^2(\os,\R^3)}-|\DR d|^2\cdot d\,.
\end{align}
We observe that (\ref{balan-forc}) constitutes a nonlinear, nonconvex problem for the midsurface $m$ once the unit director $d$ is determined. Therefore, existence to (\ref{balan-forc},\ref{Deltaanti}) is not yet known and likely not true. We note that the right-hand side in (\ref{Deltaanti}) contains an $L^2(\os)$-term, since $\D m\in L^4(\omega)$ instead of our $L^1(\os)$-term in equation (\ref{Del}). For $\DR m\equiv0$ we recover from (\ref{Deltaanti}) the director equilibrium for the uni-constant liquid crystal problem equation (\ref{euler-lag}).
\subsection{Numerical experiments}
We present a sequence of numerical experiments for the problem~\eqref{E2D}. In these experiments we compare the dimensionally reduced energy~\eqref{compa ener} (where the transverse shear energy is multiplied by the arithmetic mean of $\mu$ and $\mu_c$) to the energy~\eqref{curv} of the rigorously derived $\Gamma$-limit membrane model (where the transverse shear energy is multiplied by the harmonic mean of $\mu$ and $\mu_c$). We set the Lamé parameters to $\mu = 2.7191 \cdot 10^{4}$, $\lambda = 4.4364\cdot 10^{4}$, and vary $\mu_c$ and $L_c$.

For the domain we choose the unit disk, which we discretized by $6 \cdot 4^6 = 24\:576$ triangular elements.
We used Lagrange finite elements of second order for the midsurface deformation~$m$ and geodesic finite elements of second order for the microrotation field~$R$ \cite{sander2016geodesic,sander2016numerical}.

To trigger the deformation process, we radially compressed the membrane to a new radius $r<1$ by Dirichlet boundary conditions for the deformation on the entire domain boundary.
The microrotation field was not subject to Dirichlet boundary conditions at all.
We minimized the discrete energy using a trust-region method \cite{sander2016numerical} starting from the cap function $m_0(x,y) = (x, y, 0.1 - 0.1 \sqrt{x^2 + y^2})$ for all $(x,y)$ in the interior of the unit disk and $m_0(x,y) = (r x, ry, 0)$ on the boundary. The initial microrotation was $R = \id$. We conducted several simulations
resulting in different wrinkle patterns depending on the Cosserat couple modulus $\mu_c$ and the characteristic length $L_c$ as shown in Figures \ref{numeric figure A}, \ref{numeric figure B} and \ref{numeric figure C}. The numerical algorithms were implemented in C++ using the DUNE libraries (www.dune-project.org) \cite{sander2020dune}.

From the figures one can see that wrinkling only happens if the characteristic length $L_c$ is small enough. Indeed, if $L_c = 10^{-3}$
then the deformation is largely bending-dominated, with small wrinkles only appearing next to the boundary,
if $\mu_c$ is large enough. With smaller values for $L_c$ one can see wrinkling in larger parts of the domain,
even if the radial compression factor $r$ is much smaller. Note that the choice $\mu_c = 0$ does not lead to a
well-posed problem when used in the energy~\eqref{curv}, because there it makes the transverse shear energy term disappear.

\begin{figure}
	\centering
	\begin{tabular}{ c ||c | c }
		&  	\textbf{Arithmetic Mean~\eqref{compa ener}} & 	\textbf{Harmonic Mean~\eqref{curv}}\\
		\hline
		\shortstack{$\mu_c = 0$\vspace{2.8em}}& \includegraphics[width=0.24\textwidth]{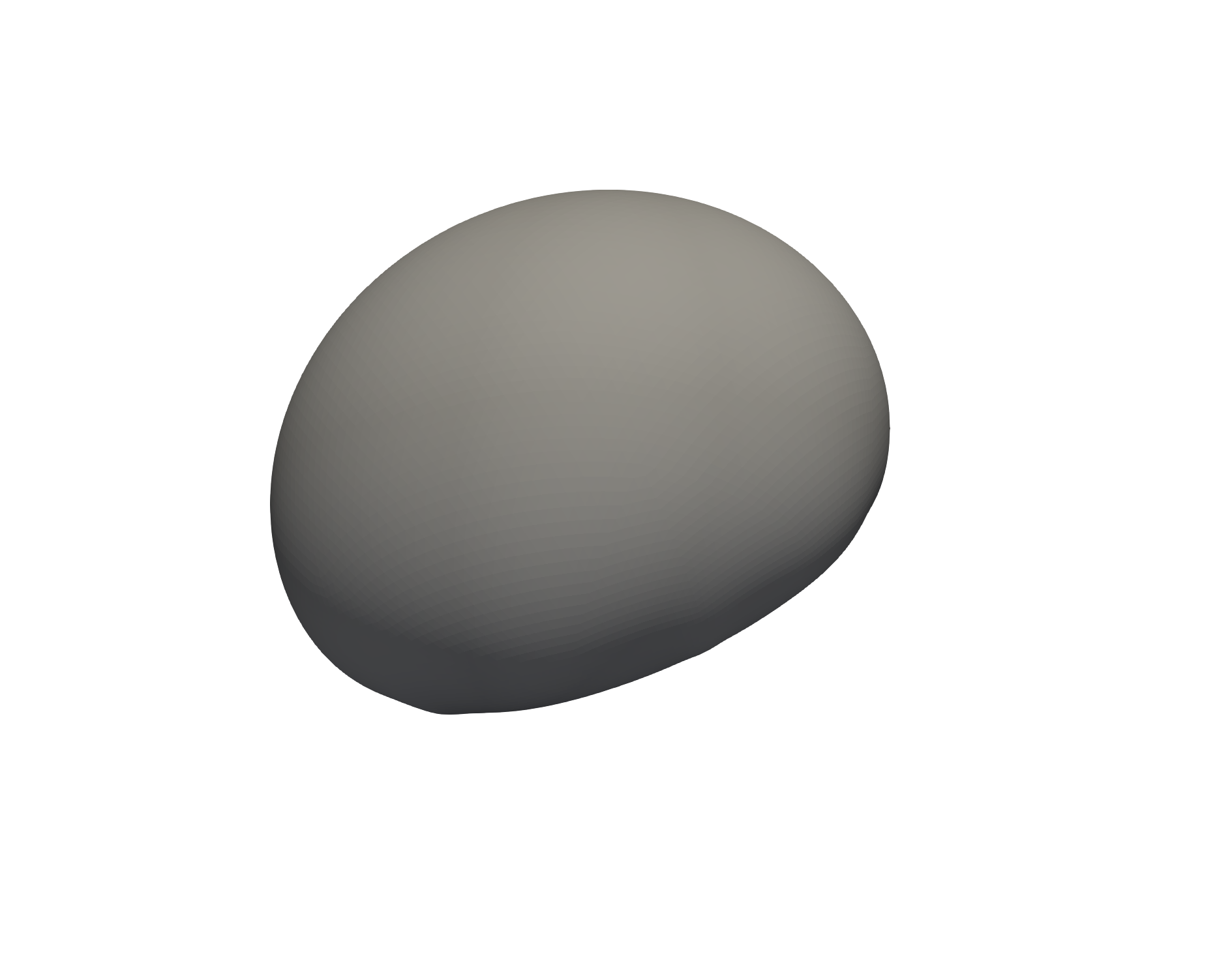} {\tiny$r = 0.9$}& {\tiny not well-posed}\\
		\hline
		\shortstack{$\mu_c = 10^{-5}\mu$\vspace{2.4em}}
		& \includegraphics[width=0.24\textwidth]{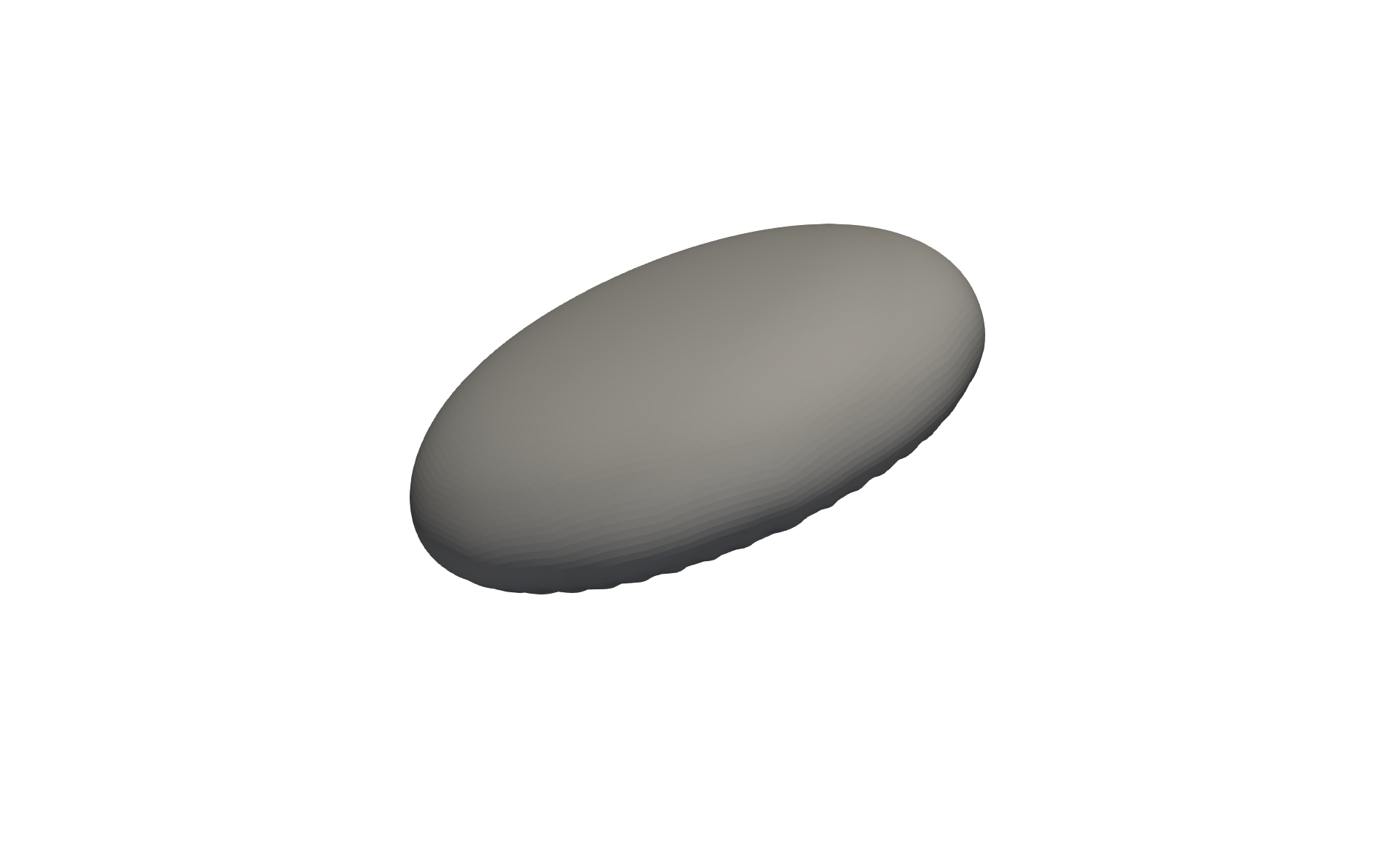} {\tiny$r = 0.9$} &  \includegraphics[width=0.24\textwidth]{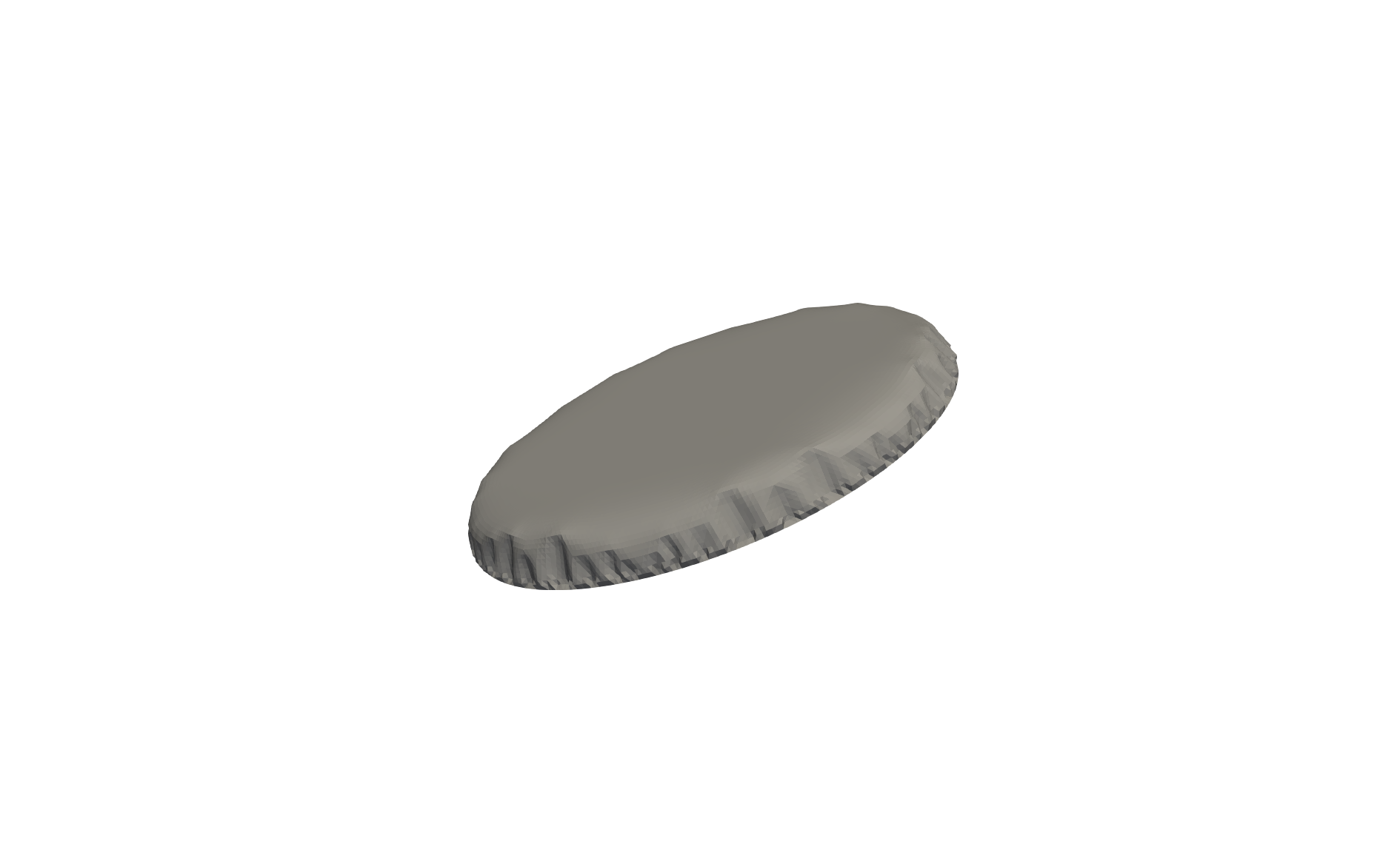} {\tiny$r = 0.9$}\\
		\hline 
		\shortstack{$\mu_c = 10^{-2}\mu$ \vspace{2.4em}}
		& \includegraphics[width=0.24\textwidth]{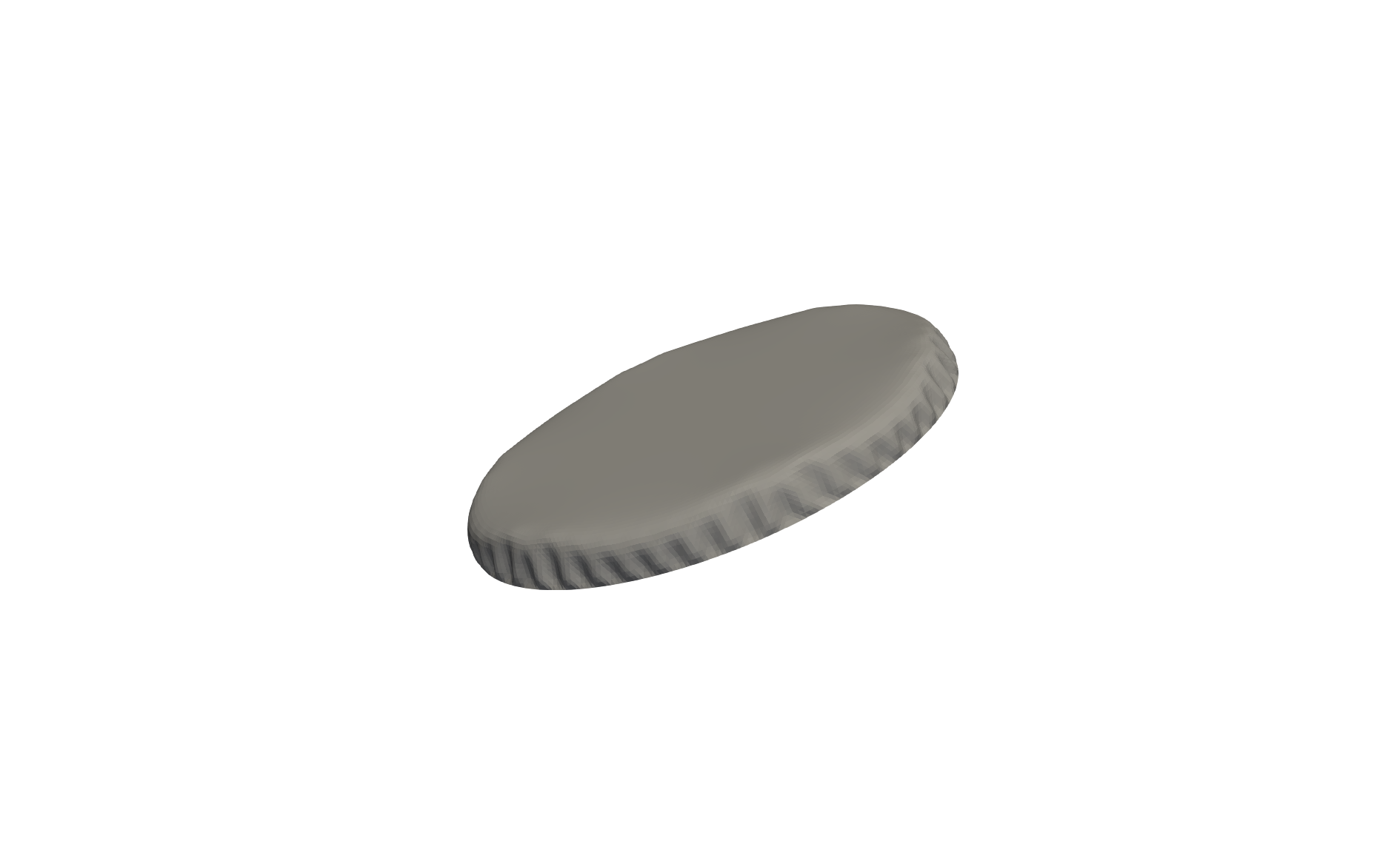}{\tiny$r = 0.9$} &  \includegraphics[width=0.24\textwidth]{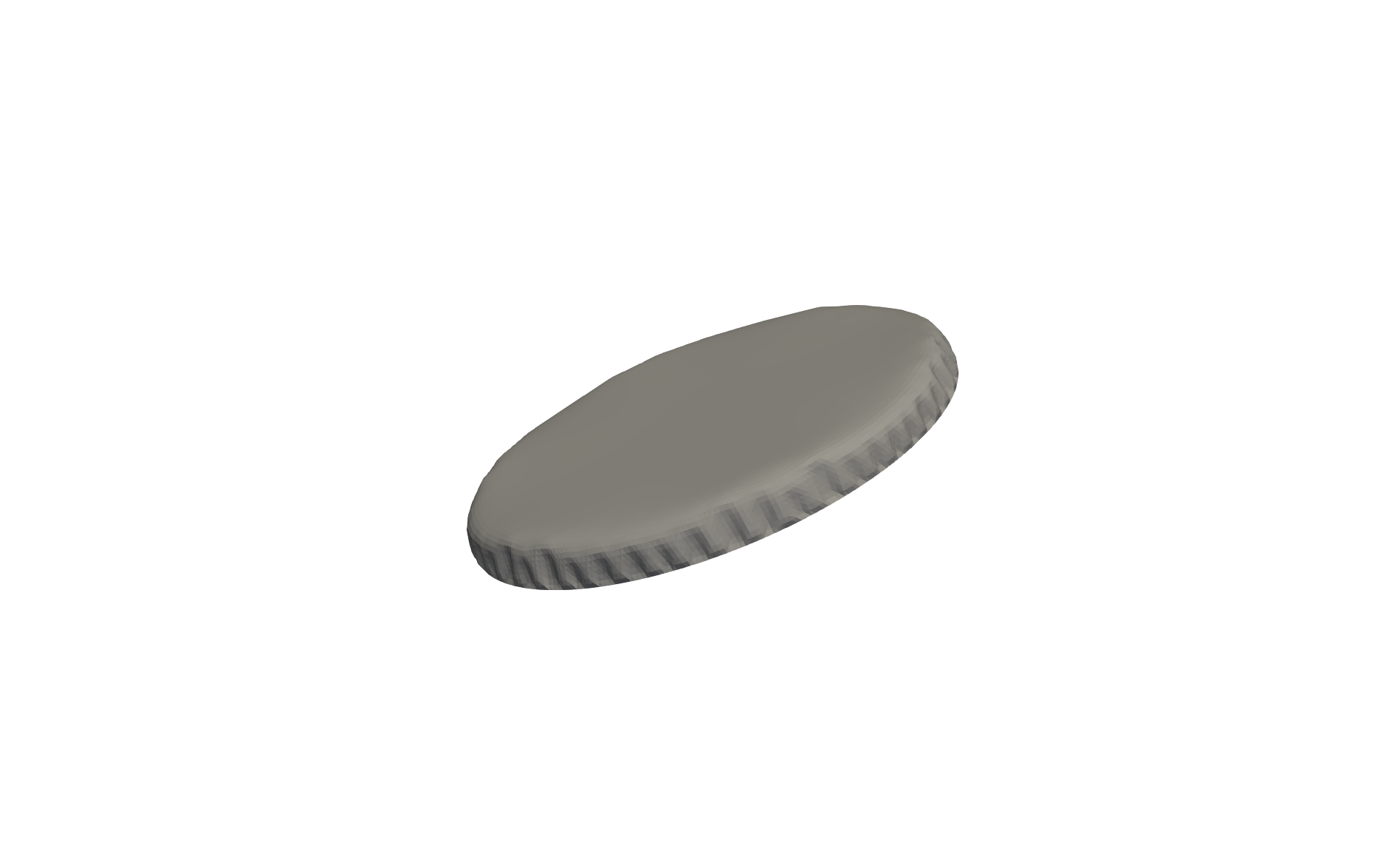} {\tiny$r = 0.9$}\\
		\hline 
		\shortstack{$\mu_c = \mu$\vspace{2.4em}}
		& \multicolumn{2}{c}{\includegraphics[width=0.24\textwidth]{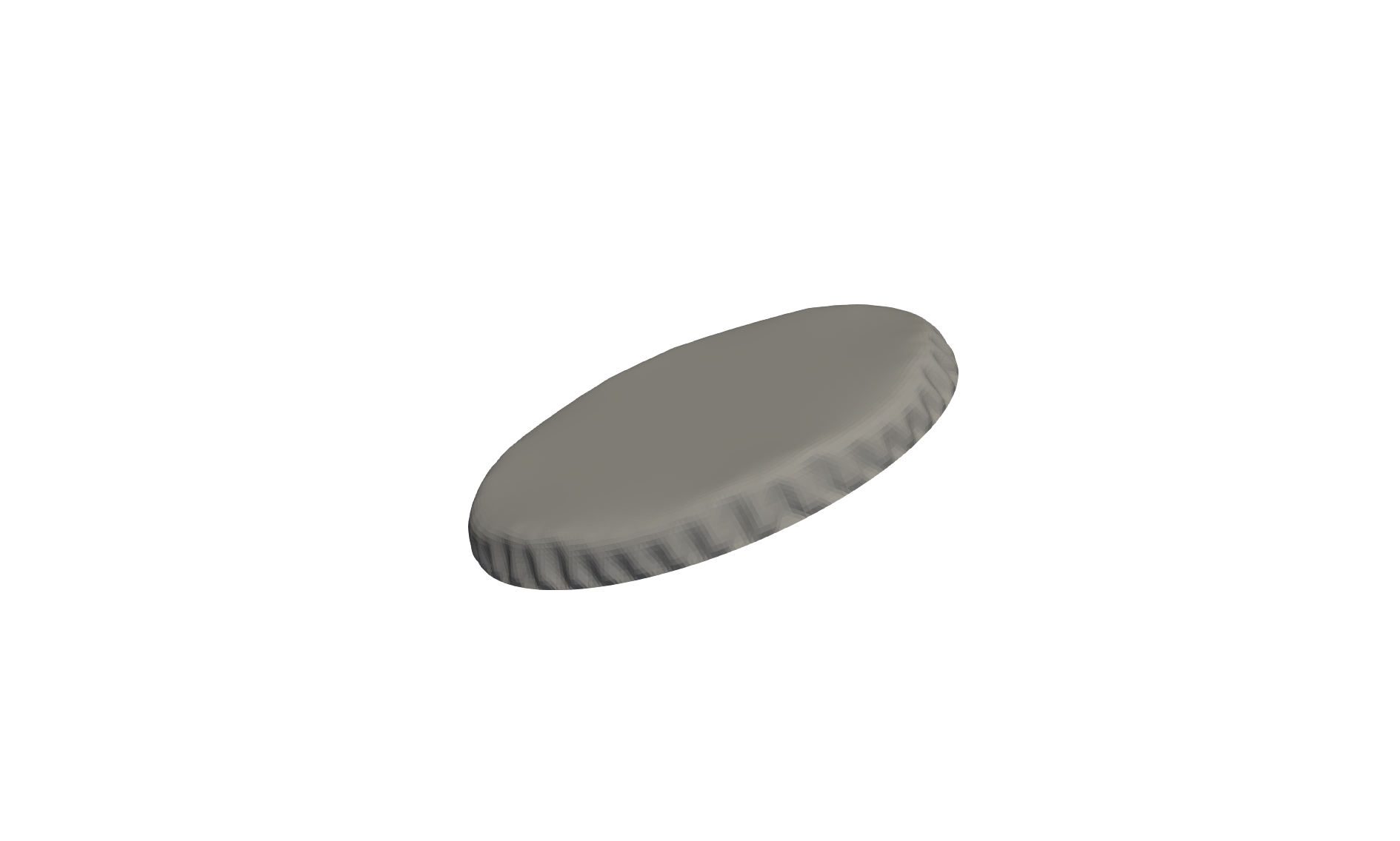}}{\tiny$r = 0.9$}\\
		\hline 
	\end{tabular}
	\caption{Deformation of a radially compressed shell with $L_c = 10^{-3}$. Simulations by Lisa Julia Nebel and Oliver Sander (TU Dresden).}\label{numeric figure A}
\end{figure}
\begin{figure}
	\centering
	\begin{tabular}{c || c | c}
		&  	\textbf{Arithmetic Mean~\eqref{compa ener}} & 	\textbf{Harmonic Mean~\eqref{curv}}	\\
		\hline 
		\shortstack{$\mu_c = 0$ \vspace{2.4em}}
		&\includegraphics[width=0.24\textwidth]{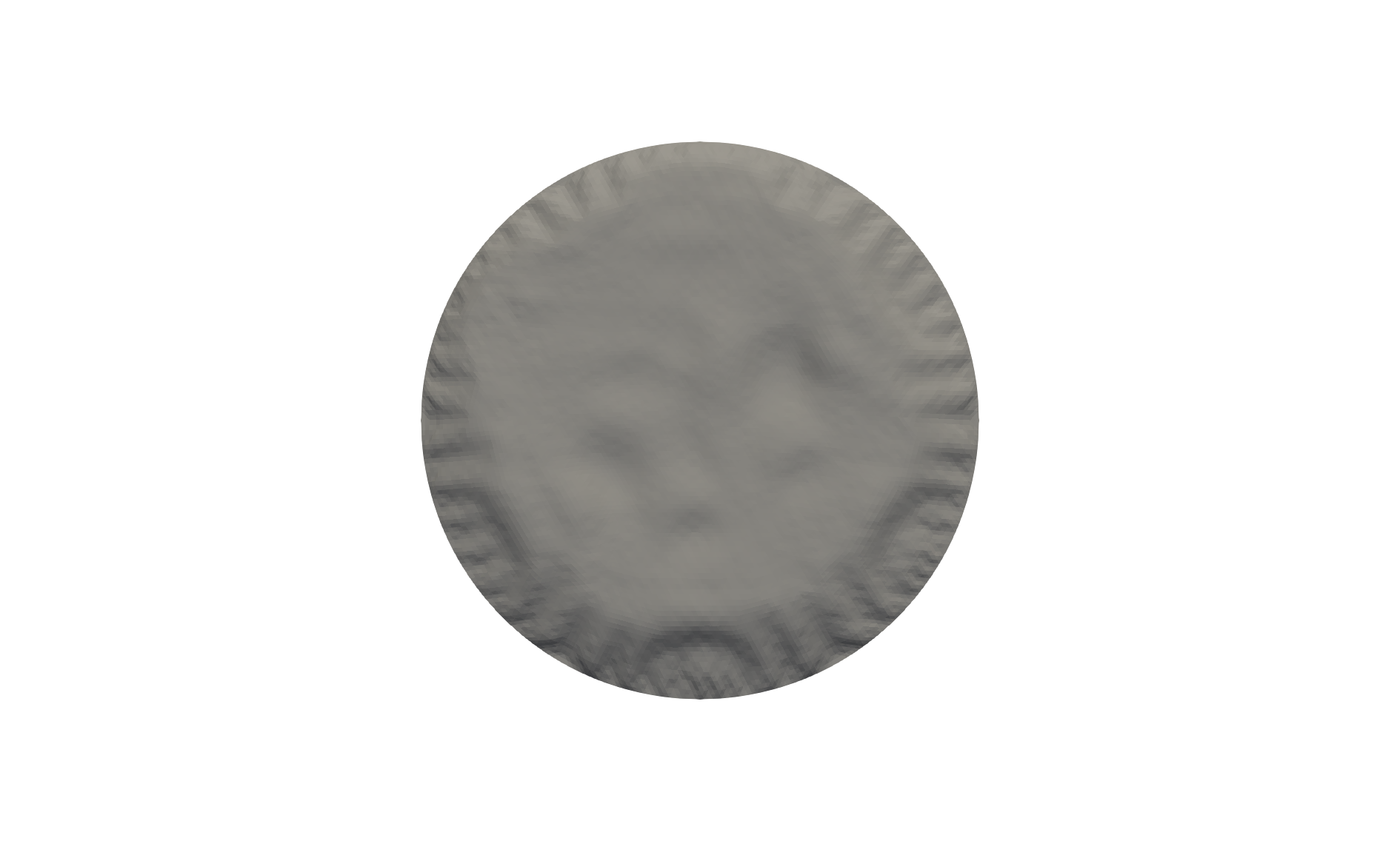}{\tiny$r = 0.98$} & {\tiny not well-posed}\\
		\hline 
		\shortstack{$\mu_c = 10^{-2}\mu$ \vspace{2.4em}}
		& \includegraphics[width=0.24\textwidth]{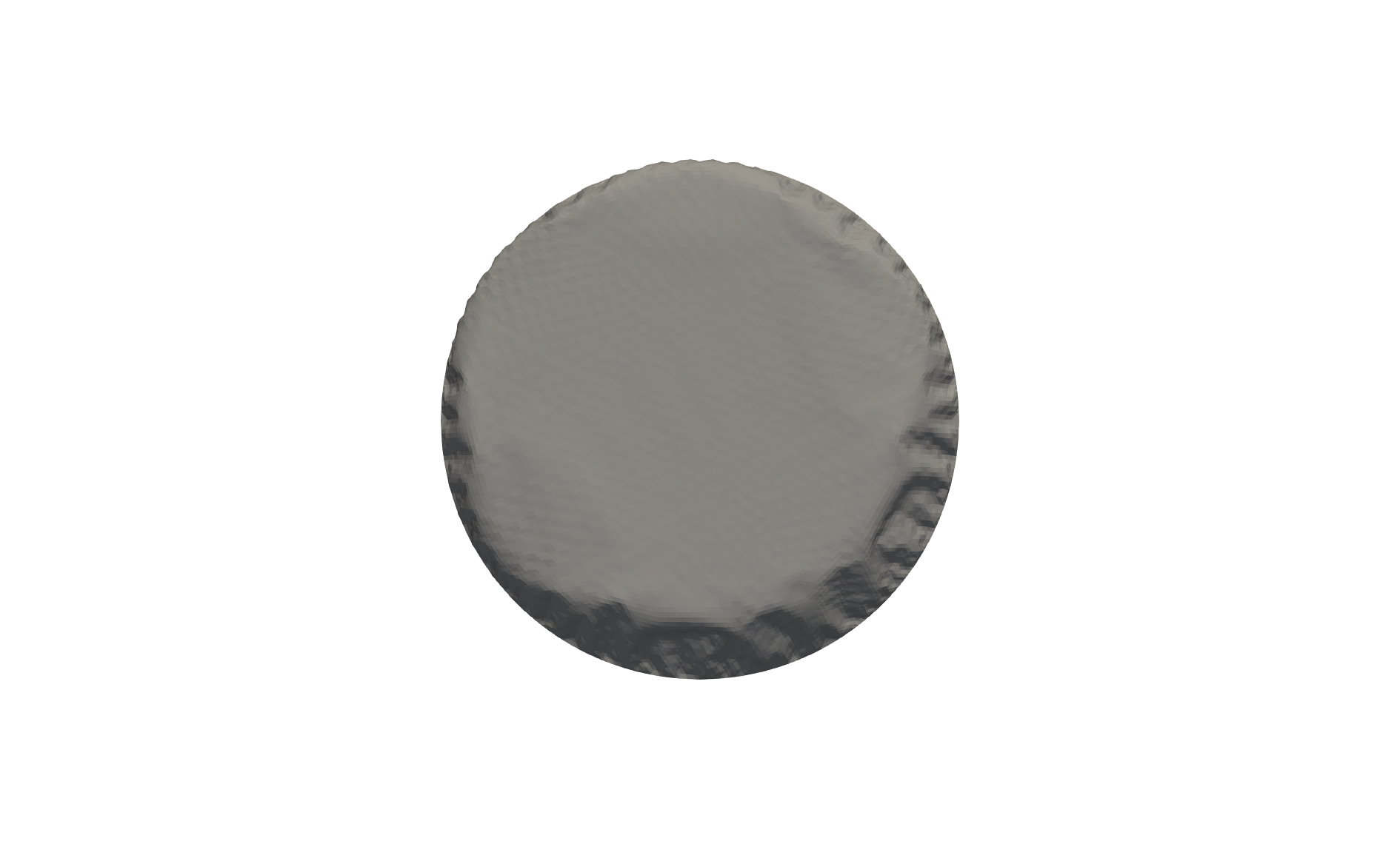} {\tiny$r = 0.91$} &  \includegraphics[width=0.24\textwidth]{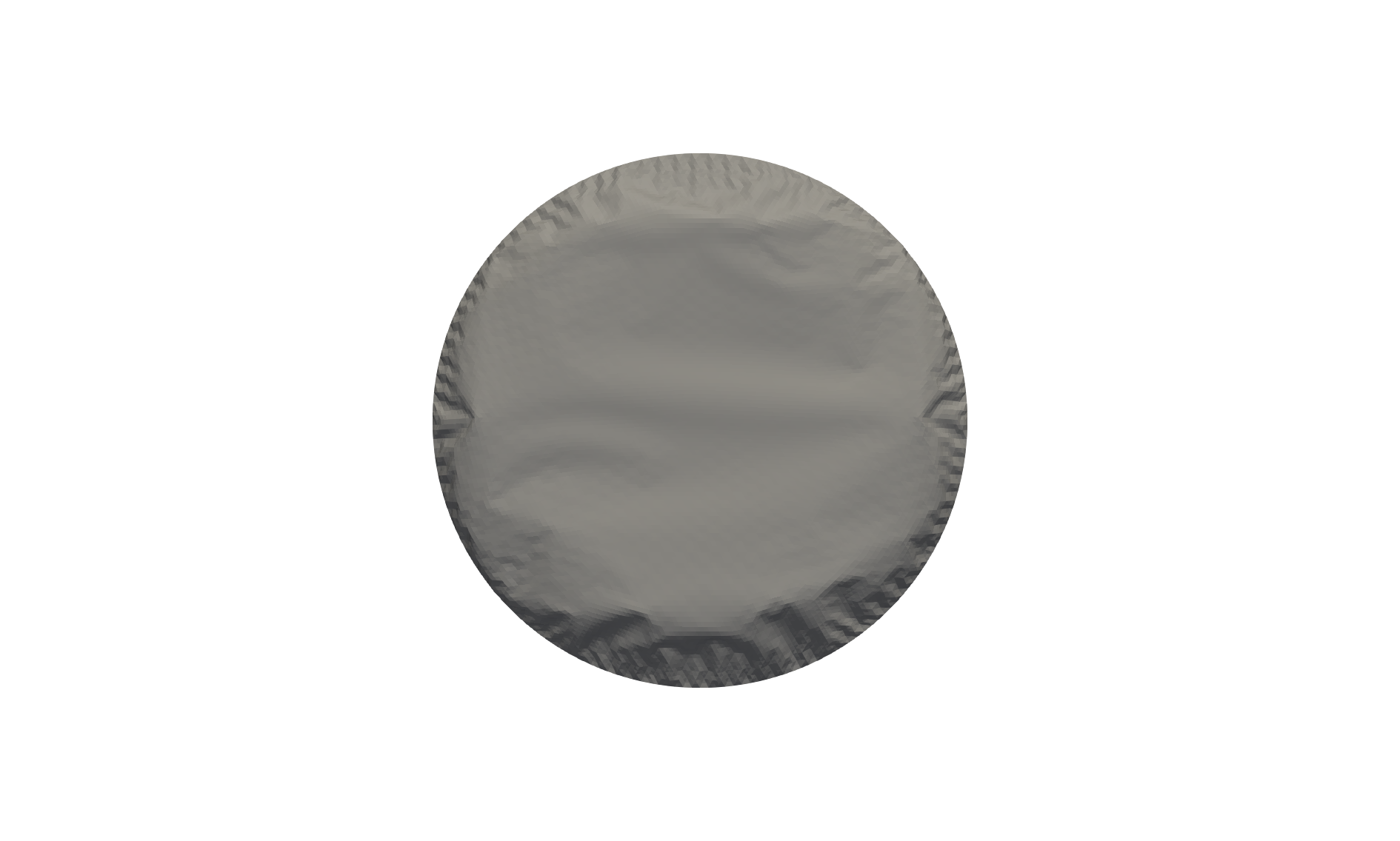}{\tiny$r = 0.94$} \\
		\hline 
		\shortstack{$\mu_c = \mu$\vspace{2.4em}}
		& \multicolumn{2}{c}{\includegraphics[width=0.25\textwidth]{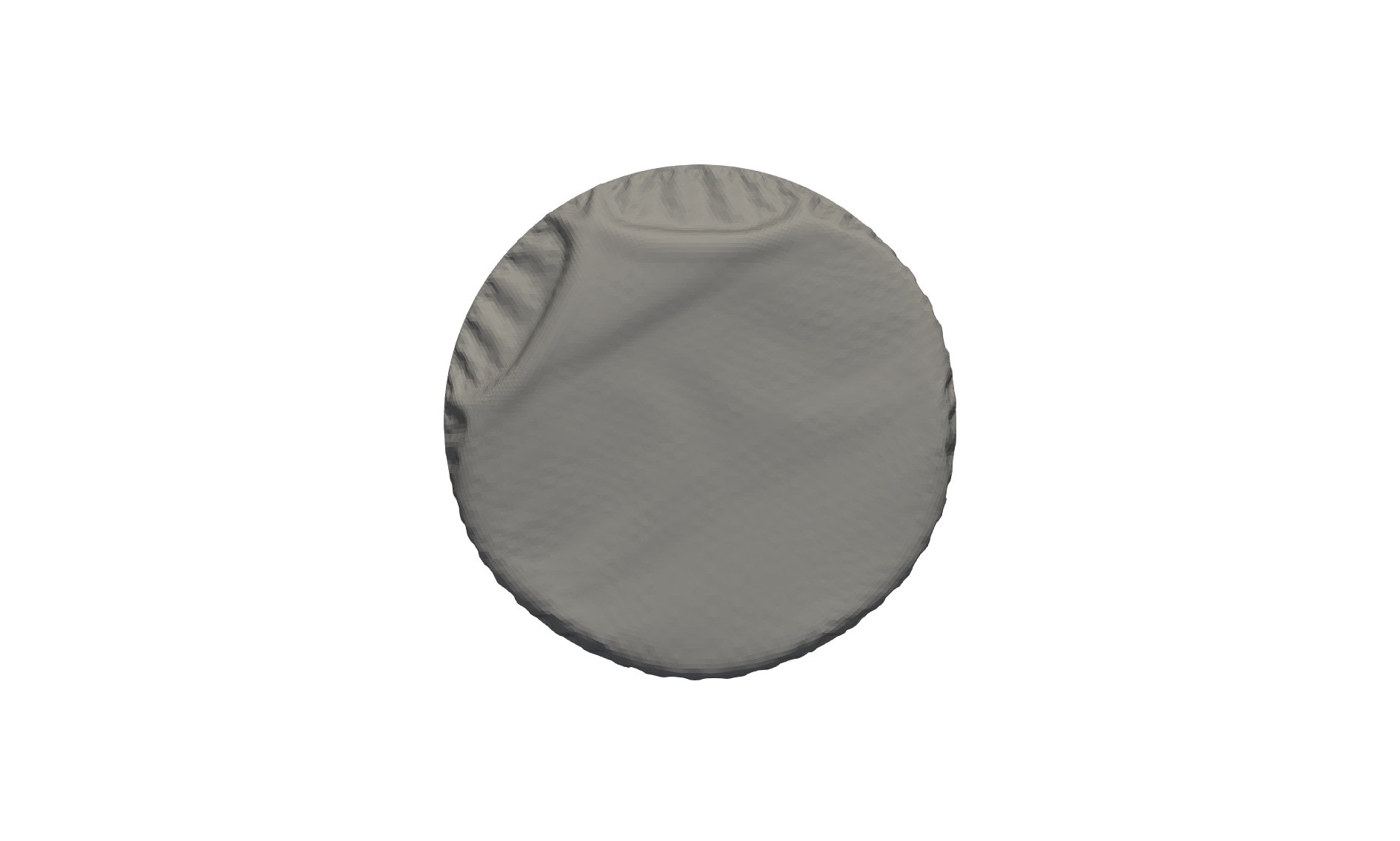}}{\tiny$r = 0.9$}\\
		\hline 
	\end{tabular}
	\caption{Deformation of a radially compressed shell with $L_c = 10^{-5}$. Simulations by Lisa Julia Nebel and Oliver Sander (TU Dresden).}\label{numeric figure B}
\end{figure}
\begin{figure}\centering
	\begin{tabular}{c || c | c}
		&  	\textbf{Arithmetic Mean~\eqref{compa ener}} & 	\textbf{Harmonic Mean~\eqref{curv}}	\\
		\hline 
		\shortstack{$\mu_c = 0$ \vspace{2.4em}} &\includegraphics[width=0.24\textwidth]{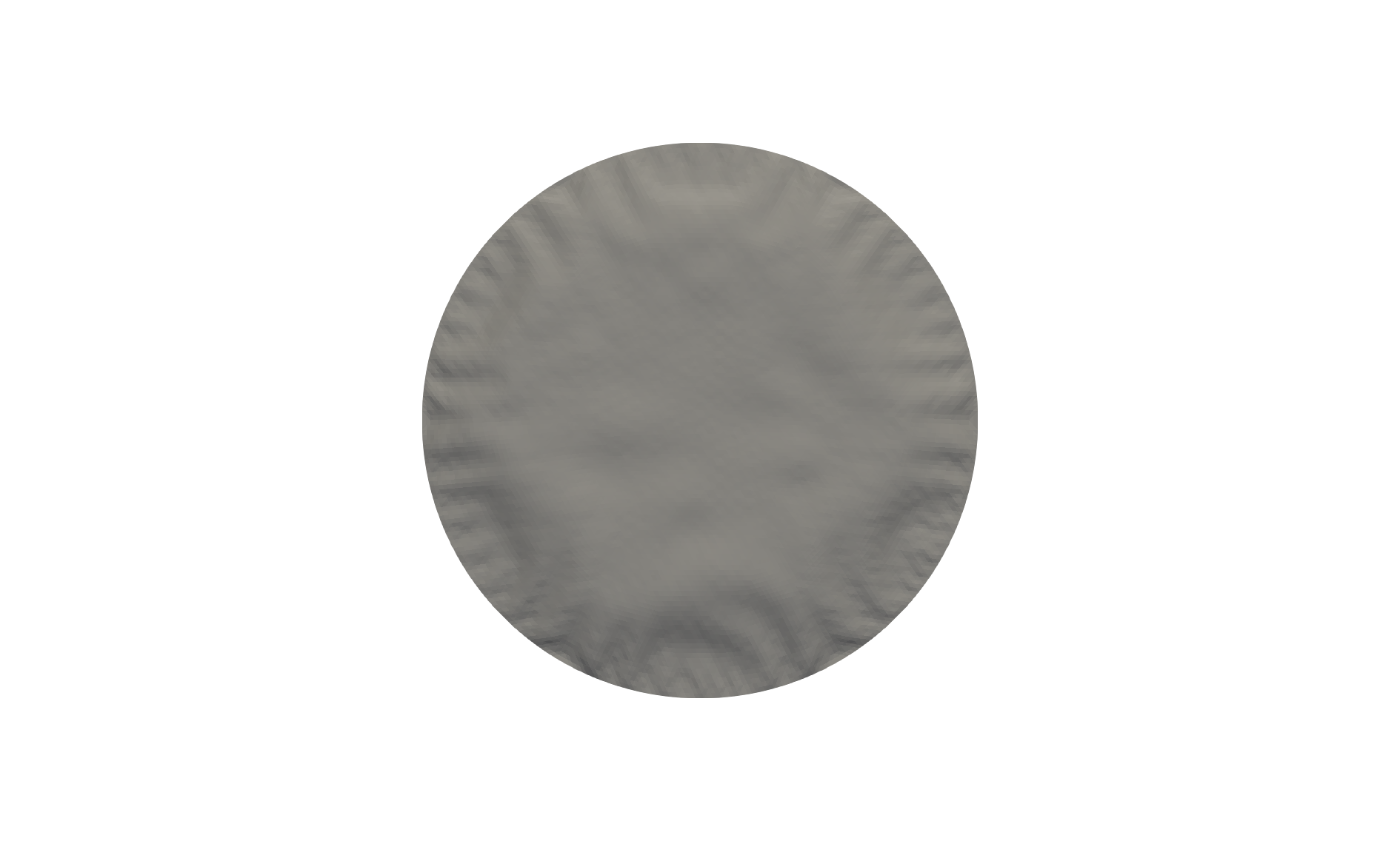}{\tiny$r = 0.99$} & {\tiny not well-posed}\\
		\hline 
		\shortstack{$\mu_c = 10^{-2}\mu$ \vspace{2.4em}}
		& \includegraphics[width=0.24\textwidth]{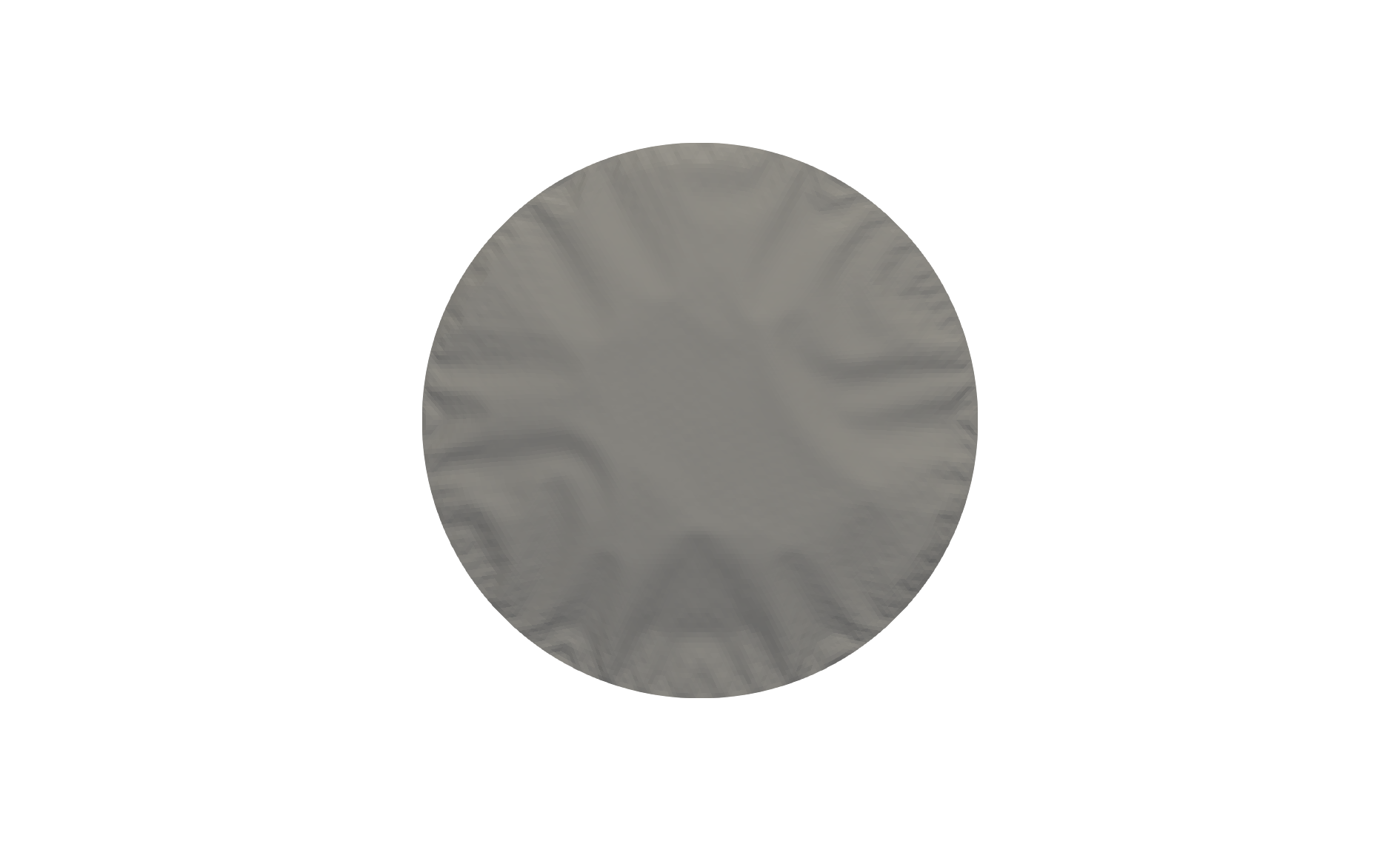}{\tiny$r = 0.99$} & \includegraphics[width=0.24\textwidth]{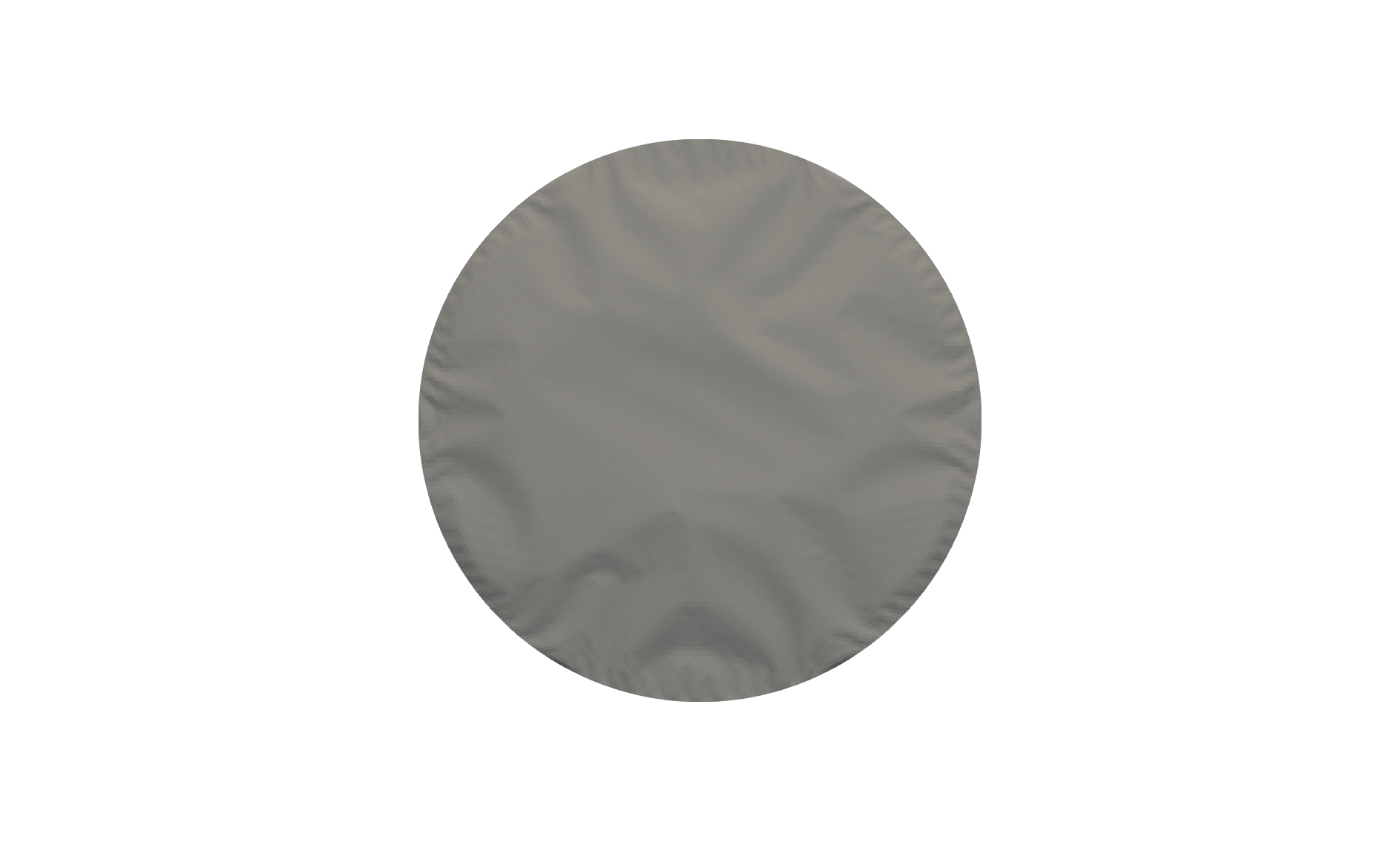}{\tiny$r = 0.99$}
		\\
		\hline 
		\shortstack{$\mu_c = \mu$\vspace{2.4em}}
		& \multicolumn{2}{c}{\includegraphics[width=0.24\textwidth]{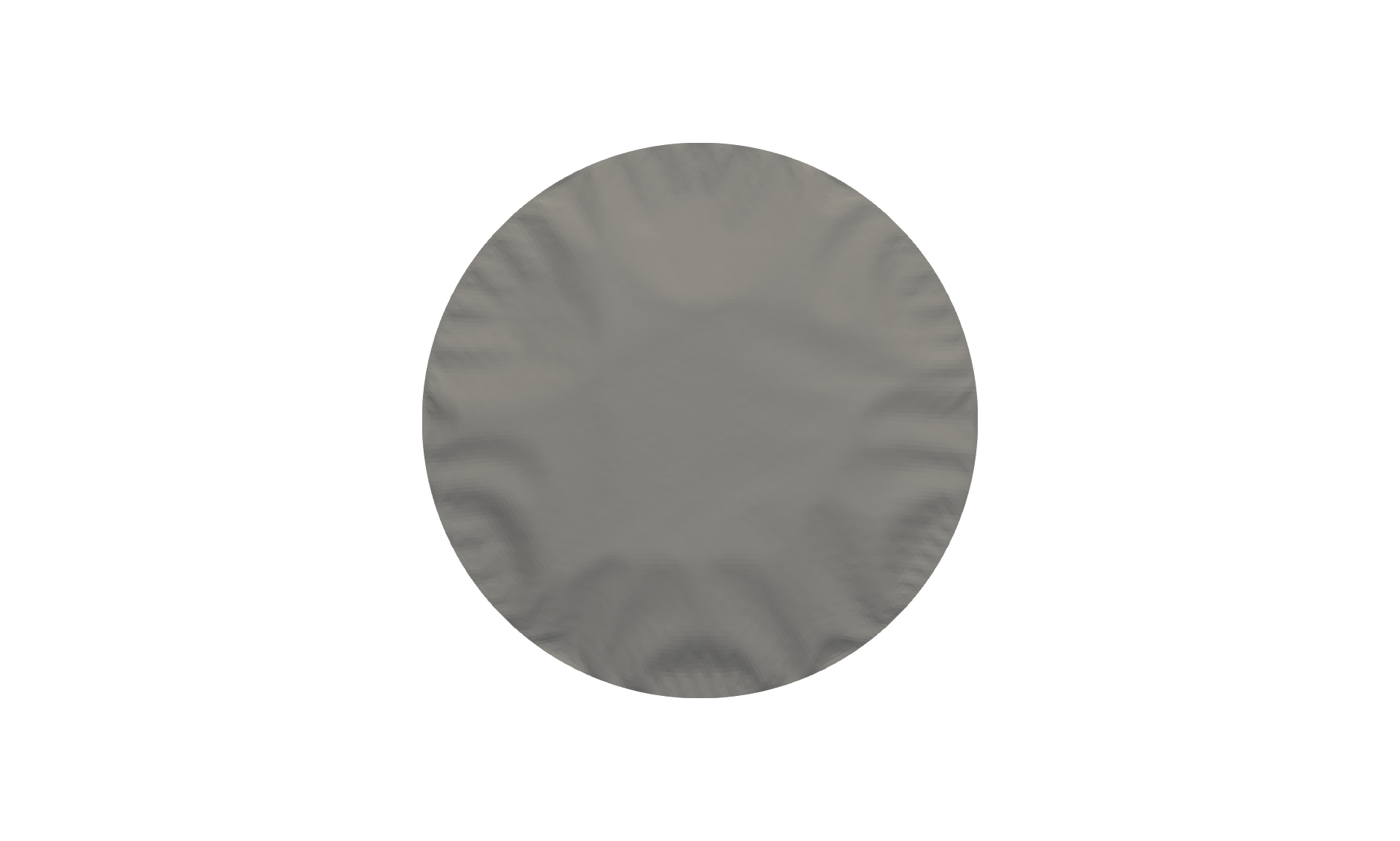}} {\tiny$r = 0.99$}\\
		\hline
	\end{tabular}
	\caption{Deformation of a radially compressed shell with $L_c = 10^{-8}$. Simulations by Lisa Julia Nebel and Oliver Sander (TU Dresden).}\label{numeric figure C}
\end{figure}

\end{appendix}
\end{document}